 \documentclass[reqno]{amsart}
\usepackage{mathrsfs}
\usepackage{amsmath}
\usepackage{amsthm}
\usepackage{amsfonts}
\usepackage{amssymb}
\usepackage{url}
\usepackage{enumerate}
\usepackage[pdftex,bookmarks=true]{hyperref}
\usepackage[usenames, dvipsnames]{color}
\usepackage{verbatim}
\usepackage{tikz}
\usepackage{tikz-cd}
\usepackage{subcaption}
\usepackage{pgfplots}
\usepackage{adjustbox}
\usepackage{bbm}
\usepackage{xcolor}
\usepackage{pdfsync}
\usepackage{extpfeil}%from roger, needed for labeled surjection arrows

\renewcommand\Im{{\operatorname{Im}}}

\newcommand\rank{{\operatorname{rank}}}

\newcommand\R{{\mathbb{R}}}
\newcommand\C{{\mathbb{C}}}

\renewcommand\P{{\mathbf{P}}}
\newcommand\E{{\mathbf{E}}}

\newcommand\diag{{\operatorname{diag}}}

\newcommand\Z{{\mathbb{Z}}}

\newcommand\F{{\mathbb{F}}}
\newcommand\al{\alpha}
\newcommand\la{\lambda}

\newcommand\Bv{{\mathbf v}}

\renewcommand\Pr{{\mathbf P }}

\newcommand\CA{{\mathcal A}}

\newcommand\CC{{\mathcal C}}
\newcommand\CE{{\mathcal E}}
\newcommand\CF{{\mathcal F}}
\newcommand\CG{{\mathcal G}}

\newcommand\BBQ {{\mathbb Q}}

\newcommand\BBR {{\mathbb R}}
\newcommand\BBZ {{\mathbb Z}}
\newcommand\BBY{{\mathbb Y}}

\newcommand\eps{\varepsilon}

\newcommand\lang{\langle}
\newcommand\rang{\rangle}

\newcommand\cok{\mathbf{Cok}}
\newcommand\Aut{\operatorname{Aut}}

\newcommand{\ra}{\rightarrow}
\newcommand\isom{\simeq}
\newcommand\Hom{\operatorname{Hom}}
\newcommand\Sur{\operatorname{Sur}}

\newcommand\GL{\operatorname{GL}}

\newcommand{\cC}{\mathcal{C}}
\renewcommand{\S}{\mathcal{S}}
\newcommand{\tP}{\widetilde{\P}_{p}^{(k)}}
\newcommand{\tPP}{\widetilde{\P}_{P}^{(k)}}
\newcommand{\KK}{\mathcal{K}}
\newcommand{\tI}{\tilde{I}}

\newcommand\bq{\begin{equation}}
\newcommand\eq{\end{equation}}

\newcommand\Inj{\operatorname{Inj}}

\newcommand{\PP}[1]{P_{#1}(t,\ldots)}
\newcommand{\QQ}[1]{Q_{#1}(1,\ldots)}
\newcommand{\inj}{\hookrightarrow}
\newcommand{\op}{\oplus}
\newcommand{\len}{\operatorname{len}}
 
\newcommand{\bx}{\mathbf{x}}
\renewcommand{\L}{\Lambda}
\newcommand{\BBC}{\mathbb{C}}
\newcommand{\by}{\mathbf{y}}
\newcommand{\ot}{\otimes}
\newcommand{\balpha}{\mathbf{\alpha}}
\newcommand{\Mat}{\operatorname{Mat}}
\newcommand{\bver}{\big{\vert}}

\newcommand{\surj}{\twoheadrightarrow}
\newcommand{\xsurj}{\xtwoheadrightarrow}

\usepackage{fullpage}

\theoremstyle{plain}

\newtheorem{prop}{Proposition}[section]
\newtheorem{proposition}[prop]{Proposition}
\newtheorem{theorem}[prop]{Theorem}
\newtheorem{corollary}[prop]{Corollary}

\newtheorem{example}[prop]{Example}

\newtheorem{lemma}[prop]{Lemma}

 \newtheorem{claim}[prop]{Claim}

\theoremstyle{definition}
  \newtheorem{definition}{Definition}
  \newtheorem{remark}{Remark}

\begin{document}

 \title{Universality for cokernels of random matrix products}

\author{Hoi H. Nguyen}
 \address{Department of Mathematics, Ohio State University}
 \email{nguyen.1261@osu.edu}
 
 \author{Roger Van Peski}
\address{Department of Mathematics, Massachusetts Institute of Technology}  
\email{rvp@mit.edu}

%\subjclass[2010]{15B52, 60B20}

\begin{abstract} 
For random integer matrices $M_1,\ldots,M_k \in \Mat_n(\Z)$ with independent entries, we study the distribution of the cokernel $\cok(M_1 \cdots M_k)$ of their product. We show that this distribution converges to a universal one as $n \to \infty$ for a general class of matrix entry distributions, and more generally show universal limits for the joint distribution of $\cok(M_1),\cok(M_1M_2),\ldots,\cok(M_1 \cdots M_k)$. Furthermore, we characterize the universal distributions arising as marginals of a natural generalization of the Cohen-Lenstra measure to sequences of abelian groups with maps between them, which weights sequences inversely proportionally to their number of automorphisms. The proofs develop an extension of the moment method of Wood to joint moments of multiple groups, and rely also on the connection to Hall-Littlewood polynomials and symmetric function identities. As a corollary we obtain an explicit universal distribution for coranks of random matrix products over $\F_p$ as the matrix size tends to infinity.

\end{abstract}

\maketitle

\tableofcontents

\section{Introduction}

Products of random matrices have been studied as far back as the works of Bellman \cite{Bellman} and Furstenberg-Kesten \cite{FursKes} around 1960, and many works since then have connected them to other problems in pure and applied mathematics and in physics, see e.g. \cite{ABK,CPV,GS,HN}. Such products have two natural parameters to vary, namely the size $n$ of the matrices and the number $k$ of matrices in the product. Different limit regimes of $n,k$ yield different behaviors; at one extreme, \cite{FursKes} and later works consider the singular values of a product $M_1 \cdots M_k$ of $n \times n$ matrices over $\R$ or $\C$, for fixed $n$, as the number of matrices $k$ in the product goes to infinity. At the other, works such as \cite{GTprod} consider the singular values of such a product in the limit as $n \to \infty$, $k$ fixed. 

Another direction of random matrix theory, at first sight orthogonal, concerns asymptotics of random matrices over finite fields $\F_p$. Assume that $M_u=(m_{ij})_{1\le i,j\le n}$ is a random matrix of size $n$ whose entries are iid uniform over $\F_p$. Then for each given $n,d$ it is elementary to compute the probability that $M_u$ has corank $d$, from which one can show
\begin{equation}\label{eq:1-matrix_rank}
\lim_{n\to \infty} \P(\rank(M_u) = n-d) =\frac{1}{p^{d^2}} \frac{\prod_{i=d+1}^\infty (1-1/p^i)}{\prod_{i=1}^d (1-1/p^i)}.
\end{equation}
Quite interestingly, it turns out that the above statistics are universal: it has been shown in \cite{M1, NgP, NgW-iid, W1} that if the entries of $M$ are iid from some nonconstant distribution which is independent of $n$, then the rank (or corank) statistics of $M$ over $\F_p$ matches with that of the uniform model above\footnote{Results from \cite{M1,NgP,NgW-iid} also allow $d$ to vary with $n$, and with explicit rate of convergence.}. 

%In particular, 
%$$\lim_{n\to \infty}\P(\rank(M_u) = n) =\prod_{i=1}^\infty (1-1/p^i), \mbox { and } \lim_{n\to \infty} \P(\rank(M_u) = n-1) =\frac{p}{(p-1)^2} \prod_{i=1}^\infty (1-1/p^i).$$
Returning to matrix products, our first goal is to understand the corank statistics; although these problems are basic, we could not find any references in the literature. Assume that $M_{u,1},M_{u,2}$ are two independent random matrices of size $n$ whose entries are uniform over $\F_p$. It is clear that the product matrix $M_{u,1}M_{u,2}$ has higher probability to be degenerate. More precisely, one can use \eqref{eq:1-matrix_rank} to show 
\begin{align*}
\lim_{n\to \infty} \P(\rank(M_{u,1}M_{u,2})=n) &= \left(\prod_{i=1}^\infty (1-1/p^i)\right)^2  \\
\lim_{n\to \infty} \P(\rank(M_{u,1}M_{u,2})=n-1)  &=  \frac{2p^2-p}{(p-1)^3} \left(\prod_{i=1}^\infty (1-1/p^i)\right)^2.
\end{align*}
In particular, the second probability is actually greater than the probability that $M_{u,1},M_{u,2}$ have coranks $0,1$ or $1,0$, which may be computed by \eqref{eq:1-matrix_rank}. This is because both have corank $1$ and still have $\rank(M_{u,1}M_{u,2})=n-1$ with non-negligible probability, which is computed in the proof of Theorem \ref{thm:prod:rank}. Even in this simple example, the complexity of matrix products begins to manifest.

This complexity increases further in the more general setting of random matrices over $\Z$. Now the key object is not only the rank but the cokernel, an abelian group
$$\cok(M) := \BBZ^n/M \BBZ^n,$$
viewing $M \in \Mat_n(\BBZ)$ as a linear map $\BBZ^n \to \BBZ^n$. By reducing modulo $p$, results on the cokernel naturally yield results on coranks of matrices over $\F_p$. For matrices with iid entries from the large class of ``$\alpha$-balanced" distributions described shortly, the distribution of the $p$-Sylow subgroup $\cok(M)[p^\infty]$----often referred to equivalently as the $p^\infty$-torsion----is universal. Namely, it was shown in \cite[Corollary 3.4]{W1} that it has the so-called \emph{Cohen-Lenstra distribution}, 
\begin{equation}\label{eq:cl}
\lim_{n \to \infty} \P\left(\cok(M)[p^\infty] \simeq G\right) = \frac{(p^{-1};p^{-1})_\infty}{\#\Aut(G)}
\end{equation}
for any finite abelian $p$-group $G$, where here and later we use the $q$-Pochhammer notation 
$$(a;q)_\infty := \prod_{i \geq 0} (1-aq^{i}).$$
This universality result was motivated by the Cohen-Lenstra heuristics for the distribution of class groups of quadratic imaginary number fields \cite{CL,FW}. 

In light of the above it is natural to ask about cokernels of random matrix products over $\BBZ$, but very little work has been done. To our knowledge cokernels of matrix products were first considered in \cite{VP21limits}, which studied the related setting of random matrices over the $p$-adic integers $\BBZ_p$ in the regime of fixed matrix size $n$ and growing number of products $k$. The present work considers random matrix products over $\BBZ$ in the opposite regime where $n$ grows and $k$ is fixed\footnote{While we state our results over $\BBZ$ in this section, we simultaneously obtain results on matrices over $\BBZ_p$, see Theorems \ref{theorem:main:prod_Zp} and \ref{theorem:main_joint_Zp} in the body of the paper.}. We seek to answer the following natural questions for cokernels, together with their analogues for coranks over $\F_p$:
\begin{enumerate}
  \item[(Q1)] What is the $n \to \infty$ limiting distribution of $\cok(M_1 \cdots M_k)$, where $M_i \in \Mat_n(\BBZ)$ are iid? Are the results universal, insensitive to the distribution of the entries of $M_i$?
  \vskip .1in
  \item[(Q2)] More generally, what is the joint distribution of $\cok(M_1),\cok(M_1M_2),\ldots,\cok(M_1 \cdots M_k)$?
\end{enumerate}

\subsection{Main results.} 

The existing single-matrix universality results of \cite{W1} are proven for matrices with iid entries satisfying the following condition, which makes it a natural candidate to probe universality for products as well.

\begin{definition}\label{def:alpha_balanced}
Given a real number $\alpha \in (0,1/2]$, we say a random integer $\xi$ is \emph{$\al$-balanced} if for every prime $p$ we have
\begin{equation}\label{eqn:alpha}
\max_{r \in \Z/p\BBZ} \P(\xi\equiv r \pmod{p}) \le 1-\al.
\end{equation}
\end{definition}
% We say a random integer $\xi$ is \emph{not deterministic mod any prime} if it is $\alpha$-balanced for some $\alpha$.

Our main results generalize \cite[Corollary 3.4]{W1} to products of matrices, answering (Q1) and (Q2) above. We begin with the simpler (Q1). In what follows, for a finite set $P$ of primes we write $G[P] :=\bigoplus_{p \in P} G[p^\infty]$, where we recall that $G[p^\infty]$ is the $p$-Sylow subgroup of $G$. 

\begin{theorem}\label{theorem:main:prod} Let $M_1,\dots, M_k$ be $k$ independent random integral matrices with entries 
 iid copies of an {\it $\alpha$-balanced} random integer $\xi$. Let $B$ be any finite abelian group, and let $P$ be a finite set of primes including all those that divide $|B|$. Then%Then for each $p\in P$
%$$\lim_{n\to \infty}\P\left((\cok(M_1\cdots M_k ))[p^\infty] \simeq  B_p \right) = \frac{\left(\prod_{i=1}^\infty 1-p^{-i}\right)^k \#\{0 = G_0 \leq G_1 \leq \}.$$
%Furthermore,
\begin{equation}\label{eq:limit_prod_intro}
\lim_{n\to \infty}\P\left((\cok(M_1\cdots M_k ))[P] \simeq  B\right) =\left(\prod_{p\in P} (p^{-1};p^{-1})_\infty^k\right)  \frac{\#\{0 = G_0 \leq G_1 \leq \ldots \le G_k = B\}}{\#\Aut(B)}.
\end{equation}
\end{theorem}

For (Q2), we note first that as $M_1 \cdots M_{j+1}\Z^n \subset M_1 \cdots M_j \Z^n$, there is a natural surjection $\cok(M_1 \cdots M_{j+1}) \surj \cok(M_1 \cdots M_j)$. We also define the notation $\Sur(G,H) := \{\phi: G \to H \text{ surjective}\}$ and similarly for $\Inj(G,H)$. 

\begin{theorem}\label{theorem:main:joint} For matrices under the same assumptions as in Theorem \ref{theorem:main:prod}, finite abelian groups $B_1, \ldots, B_k$, and $P$ a finite set of primes including all those which divide every $|B_i|,1\leq i \leq k$, we have
\begin{equation}\label{eq:limit_joint_intro}
\lim_{n\to \infty}\P\left((\cok(M_1 \cdots M_j)[P] \simeq  B_j, 1\le j\le k\right)=\left(\prod_{p\in P}(p^{-1};p^{-1})_\infty^k\right) \prod_{i=1}^k \frac{\#\Sur(B_i, B_{i-1})}{\#\Aut(B_i)},
\end{equation}
where we take $B_0 = 0$.
\end{theorem}

\begin{remark}
Theorem \ref{theorem:main:joint} reduces to Theorem \ref{theorem:main:prod} by a simple computation. We also note that, since $\cok(M_1 \cdots M_j) \cong \cok(M_j^T \cdots M_1^T)$ and $M_j=M_j^T$ in distribution for all $j$, it is immediate that Theorem \ref{theorem:main:joint} holds with $\cok(M_1 \cdots M_j)[P]$ replaced by $\cok(M_j \cdots M_1)[P]$. %We will see also that the right hand sides of \eqref{eq:limit_prod_intro} and \eqref{eq:limit_joint_intro} define probability measures on the set of abelian groups with order divisible only by primes in $P$, and on sequences of $k$ such groups, respectively. 
\end{remark}

The $k=1$ case of either above result yields \cite[Corollary 3.4]{W1}, and our results can be seen as a dynamical analog of this one. Within one matrix, the evolution of the cokernel after exposing each new row and column of the matrix was previously studied by the first author and Wood \cite{NgW-iid, NgW-lap}, while in this current model we study the cokernel evolution by multiplying the matrices one by one. An interesting related body of work \cite{Bore,CK,CLS,Lee} studies the joint cokernel distributions of matrices obtained from different polynomials of a single random matrix, and it is natural in light of the above to consider joint cokernel distributions of more complicated multivariate polynomials in several random matrices.

%We note also that joint cokernel distributions of matrices obtained from different polynomials of a single random matrix have been studied in \cite{}. It would be interesting to understand the joint cokernel distribution of several multivariate polynomials in multiple random matrices

%Theorem \ref{theorem:main:joint} reduces to Theorem \ref{theorem:main:prod} by considering the marginal distribution of $B_k$ and noting that by duality
%$$\frac{\#\Sur(B_i, B_{i-1})}{\#\Aut(B_i)} = \frac{\#\Inj(B_{i-1},B_i)}{\#\Aut(B_i)} = \#\{H \leq B_i: H \simeq B_{i-1}\}.$$

%, but the extra data of Theorem \ref{theorem:main:joint} gives an attractive interpretation of the limiting distribution. 
Recall that for cokernels of a single matrix, the distribution \eqref{eq:cl} features weights inversely proportional to the number of automorphisms. It is a general heuristic that distributions on algebraic objects occuring in these contexts should feature probabilities inversely proportional to the number of automorphisms, for the appropriate notion of automorphism, the reason essentially being the orbit-stabilizer theorem---see for instance \cite[Section 5]{woodexpos}. Our next result gives such an interpretation for the distributions appearing above.

As mentioned, the groups $\cok(M_1 \cdots M_j)$ come with additional structure of a sequence of maps 
\begin{equation}\label{eq:coks_with_maps_intro}
\cok(M_1 \cdots M_k) \surj \cok(M_1 \cdots M_{k-1}) \surj \ldots \surj \cok(M_1).
\end{equation}
There is a natural notion of automorphism of such a sequence $G_k \surj \ldots \surj G_1$ of abelian groups with maps between them, namely an element of $\prod_{i=1}^n \Aut(G_k)$ for which the appropriate diagram commutes---see Section \ref{sec:automorphisms}. This provides the right setting to interpret the distribution of Theorem \ref{theorem:main:joint}, and hence Theorem \ref{theorem:main:prod} as well, as the following result shows.

\begin{theorem}\label{theorem:seq_aut_intro}
Let $k \in \BBZ_{\geq 1}$ and $P$ be a finite set of primes. Then there is a well-defined probability measure $\tPP$ on the set
$$\left\{\text{sequences }G_k \xsurj{\phi_{k-1}} G_{k-1} \xsurj{\phi_{k-2}} \ldots \xsurj{\phi_1} G_1 \text{ up to automorphism}: |G_k| \text{ only divisible by primes in }P\right\}$$
defined by assigning to each automorphism class $[G_k \surj \ldots \surj G_1]$ of sequences the probability
$$\tPP([G_k \surj \ldots \surj G_1]) = \frac{\prod_{p \in P} (p^{-1};p^{-1})_\infty^k}{\#\Aut\left(G_k \surj \ldots \surj G_1\right)}.$$
Furthermore, under the above distribution, the marginal joint distribution of the isomorphism types of $G_1,\ldots,G_k$ (after forgetting the data of the maps between them) is the limit distribution of Theorem \ref{theorem:main:joint}.
\end{theorem}

We refer to Section \ref{sec:automorphisms} for more detail. The closest previous work we are aware of is \cite{BKLPR}, which considers certain random short exact sequences of abelian $p$-groups as models for the distribution of exact sequences relating Selmer and Tate-Shafarevich groups of elliptic curves. However, these distributions are supported on split short exact sequences, so the maps between them do not add information beyond the isomorphism types of the groups, in contrast to our case. Nonetheless, such heuristics for groups with maps between them motivate the general problem of developing technology to prove universality for joint distributions of multiple groups, and we hope that the present paper may conversely help to spur more work on heuristics for sequences of groups in number theory. It would certainly be natural and interesting to prove a generalization of the universality result Theorem \ref{theorem:main:joint} which incorporates the extra data of the sequence of maps \eqref{eq:coks_with_maps_intro}, but we leave this question to the future as our methods are not currently adapted to it. 

% with first term a finite power of $\mathbb{Q}_p/\Z_p$, making them in some sense much simpler than ours

In another direction, by taking everything modulo $p$, Theorem \ref{theorem:main:joint} and a linear algebra computation imply the following corollary on the joint distribution of ranks of matrix products over $\F_p$.

\begin{theorem}\label{thm:prod:rank} Let $p$ be a given prime. Let $\xi$ be a nonconstant random variable valued in $\F_p$, $r_1,\ldots,r_k \in \Z_{\geq 0}$, and $M_1,\ldots,M_k$ be independent random elements of $\Mat_n(\F_p)$ with entries iid copies of $\xi$. Then
$$\lim_{n \to \infty} \P\left(\rank(M_1 \cdots M_i) = n-(r_1+\ldots+r_i) \text{ for all }1 \leq i \leq k\right) = (p^{-1};p^{-1})_\infty^k \prod_{i=1}^k \frac{p^{-r_i(r_i+\ldots+r_1)}}{(p^{-1};p^{-1})_{r_i} (p^{-1};p^{-1})_{r_i+\ldots+r_1}}.$$
\end{theorem}

\subsection{Parallels with complex random matrices.}

Most universality results of random matrices in the literature concern the spectral distributions, most notably the Wigner semi-circular law, the quarter-circle law, and the circular law. We recite below the last two laws for the model closely related to ours. 

\begin{theorem}Assume that $M=(m_{ij})$ is a random matrix  where $m_{ij}$ are iid copies of a real-valued random variable $\xi$ of mean-zero and variance one. 

\begin{itemize}
\item (Marchenko-Pastur law, see for instance \cite{MP} and also \cite{GTcov, P, Wa,Yin}) Let $x_1,\dots, x_n$ be the eigenvalues of $(1/n)MM^\ast$. Then almost surely
$$\lim_{n \to \infty}\frac{1}{n} \sum_{i=1}^n \delta_{x_i} = \frac{1}{2\pi} \frac{\sqrt{(4-x)x}}{x}\mathbf{1}_{x\in (0,4)}dx=:\mu_{MP}(dx).$$
\item (Circular law, see for instance \cite{TVcir} and also \cite{Gin, GT, Mehta, PZ}) Let $z_1,\dots, z_n$ be the (possibly complex) eigenvalues of $(1/\sqrt{n})M$. Then almost surely
$$\lim_{n \to \infty}\frac{1}{n} \sum_{i=1}^n \delta_{z_i} = \frac{1}{\pi} \mathbf{1}_{|z| \le 1}dz.$$

\end{itemize}
\end{theorem}

As mentioned in the first paragraph, generalizations of the above exist for products of a finite number of matrices.

\begin{theorem} \label{thm:cplx_products} Assume that $M_1,\dots, M_k$ are independent and their entries are iid copies of a real-valued random variable $\xi$ of mean zero, variance one, and bounded $(2+\eps)$-moment for some $\eps>0$. Let $M=\prod_{i=1}^k \frac{1}{\sqrt{n}}M_i$.

\begin{itemize}
\item (See \cite{BBCC, Ben, BJW,Burda, Mu, ZPNC} for the Gaussian case, and \cite[Theorem 3.1]{GTprod} for the general case) Let $x_1,\dots, x_n$ be the singular values of $M$, then in probability
$$\lim_{n \to \infty}\frac{1}{n} \sum_{i=1}^n \delta_{x_i} = \mu_{MP}^{\boxtimes k}(dx),$$
where we refer the reader to \cite{BBCC} for more discussion on the free multiplicative convolution.

 \item (See for instance \cite{GTprod,OS,ORSV}) If $z_1,\dots, z_n$ are the eigenvalues of $M$, then in probability
$$\lim_{n \to \infty}\frac{1}{n} \sum_{i=1}^n \delta_{z_i} = \frac{1}{k \pi} |z|^{\frac{2}{k} -2} \mathbf{1}_{|z| \le 1} dz.$$
\end{itemize}
\end{theorem}

The groups $\cok(A)$ in the discrete setting are in fact structurally analogous to singular values in the continuous setting. In complex (or real) random matrix theory, singular value decomposition tells that for any $A \in \Mat_n(\BBC)$ there exist unitary $U,V \in U(n)$ (or, if $A$ is real, $O(n)$) so that $UAV$ is diagonal with nonnegative reals on the diagonal—the singular values. Analogously, for $A \in \Mat_n(\BBZ)$ there exist $U,V \in \GL_n(\BBZ)$ for which $UAV=\diag(a_1,\ldots,a_n)$ is diagonal with nonnegative integers on the diagonal. This result is known as Smith normal form, and the diagonal entries furthermore determine the isomorphism type of the cokernel by 
$$\cok(A) \simeq \bigoplus_{i=1}^n \Z/a_i\Z.$$
Theorem \ref{theorem:main:prod} can thus be viewed as a cokernel analog of the above results concerning the singular values and eigenvalues of a product of $k$ independent iid matrices. However, an important difference between the two settings is that the \emph{random} empirical spectral measures above limit (with scaling) to \emph{deterministic} measures, while in the cokernel setting there is no rescaling and the limit object is a \emph{random} group or collection of integers $a_i$. %In particular, it is not clear what the analogue of our joint distribution result Theorem \ref{theorem:main:joint} would be in the complex setting above. 

A related setting in classical random matrix theory, where the limits are not deterministic, is that of \emph{local limits}. Singular values of a random $A \in \Mat_n(\BBC)$ form a random collection of points on $\R_{\geq 0}$, and by zooming in at the scale of individual singular values as $n \to \infty$, one may obtain a random collection of infinitely many points. The limit object differs depending on whether one zooms in close to the largest singular value (the \emph{soft edge}), close to the smallest singular value (the \emph{hard edge}), or in between the two (the \emph{bulk}), and is a random collection of points with a rightmost point, leftmost point, or infinitely many points in both directions in each case respectively. For singular values of a product of a fixed number $k$ of complex Gaussian matrices, these scaling limits were computed for the bulk and soft edge in \cite{LWZ}. Unlike our Theorem \ref{theorem:main:prod} and Theorem \ref{thm:cplx_products} above, the limit in \cite{LWZ} does not depend on the number of products, and matches the one for a single matrix. However, at the hard edge the limit does depend on the number of products: the $n \to \infty$ limiting joint distribution of singular values of products $M_1, M_2M_1,\ldots,M_k \cdots M_1$ of $n \times n$ complex Gaussian matrices $M_i$ was computed in \cite{KZ}, and for other explicit cases outside the Gaussian on the hard edge limit was computed in \cite{KS,KKS}. The work \cite{KZ} is probably the closest complex analogue of our Theorem \ref{theorem:main:joint}.

%To make the analogy between cokernels and singular values more direct, it is convenient to work not over $\Z$ but over the $p$-adic integers $\Z_p$. In this setting, for $A \in \Mat_n(\Z_p)$ there exist $U,V \in \GL_n(\Z_p)$ such that $UAV = \text{diag}(p^{\la_1},\ldots,p^{\la_n})$, for nonnegative integers $\la_1 \geq \ldots \geq \la_n$. The integers $\la_i$ likewise determine the cokernel, and are also exactly analogous to the (negative logarithms of) singular values. Upcoming work of the second author \cite{VP23+} addresses the question of local limits of these analogues of singular values in regimes where both the matrix size and the number of products go to $\infty$ simultaneously.

%get across: local limits are a better analogy regarding randomness, however soft and bulk ones are the same for a product of k matrices, but hard edge limits are different---and appropriate references to matrix product literature. 

%To make the analogy with singular values more direct, it is convenient to work not over $\Z$ but over the $p$-adic integers $\Z_p$. In this setting, for $A \in \Mat_n(\Z_p)$ there exist $U,V \in \GL_n(\Z_p)$ such that $UAV = \text{diag}(p^{\la_1},\ldots,p^{\la_n})$, for nonnegative integers $\la_1 \geq \ldots \geq \la_n$. The integers $\la_i$ likewise determine the cokernel, and are also exactly analogous to the (logarithms of) singular values. 

\subsection{Methods, moments and Hall-Littlewood polynomials.} Previous works such as \cite{NgW-iid,NgW-lap} and especially \cite{W0,W1} show that the law of $\cok(M)$ converges to some universal distribution\footnote{Depending on the class of $M$; symmetric matrices \cite{W0} or rectangular matrices \cite{W1} yield different limiting laws for the cokernels.} by the following general strategy:
\begin{enumerate}[(i)]
  \item Find a candidate universal random group $G$ and compute its \emph{moments} $\E[\#\Sur(G,H)]$ for each finite abelian group $H$ with $|H|$ divisible by the appropriate primes. %Often $G$ may be found as the $n \to \infty$ limit of cokernels of the so-called \emph{uniform model}, where the entries of $M$ are distributed by the additive Haar measure on $\Z_p$.\label{item:univ}
  \label{item:univ}
  \vskip .05in
  \item Compute the asymptotics of the moments $\E[\#\Sur(\cok(M),H)]$ and check they agree. \label{item:asymp}
  \vskip .05in
  \item Show that the moments determine the distribution. \label{item:determine}
\end{enumerate}
For Theorem \ref{theorem:main:prod}, we follow exactly this strategy: the computation \eqref{item:univ} is in Section \ref{sect:Haar}, the asymptotics \eqref{item:asymp} are in Section \ref{section:prod:moments}, and for \eqref{item:determine} we slightly strengthen existing moment determinacy results of \cite{W0} and combine these ingredients to prove the theorem in Section \ref{sect:comparison}. For Theorem \ref{theorem:main:joint}, however, we must introduce an appropriate notion of joint moments of a sequence of random groups. We are able to generalize \eqref{item:asymp} and \eqref{item:determine} to joint moments
\begin{equation}\label{eq:joint_moment_def}
\E[\#\Sur(\cok(M_1),G^{(1)}) \cdots \#\Sur(\cok(M_1 \cdots M_k),G^{(k)})],
\end{equation}
see Theorems \ref{theorem:jointsur:k} and \ref{theorem:t:distribution:joint}. We find it helpful from an expository standpoint to prove Theorem \ref{theorem:main:prod} separately beforehand, as many ingredients are shared, and for this result \eqref{item:asymp} and \eqref{item:determine} correspond to Theorems \ref{theorem:surmoment:k} and \ref{theorem:t:distribution}. In both cases, we rely heavily on an existing analytic result \cite[Theorem 8.2]{W0} proven in a related context. It is also worth noting that a generalization of moment determinacy \eqref{item:determine} to joint moments of multiple groups as defined by \eqref{eq:joint_moment_def} was carried out independently in \cite{Lee2}, which appeared shortly after the first posting of the present paper, and applied to different joint distributions. 

For \eqref{item:univ}, the candidate for the joint distribution of cokernels comes from previous work \cite{VP21limits} (specifically Corollary 3.4) in the setting of random matrices over the $p$-adic integers $\Z_p$. The analogous cokernel joint distribution corresponds to the distribution in Theorem \ref{theorem:main:joint} when $P=\{p\}$. However, in \cite{VP21limits} it was phrased in a nontrivially equivalent manner in terms of Hall-Littlewood polynomials, certain symmetric polynomials in $n$ variables which encode harmonic analysis on the groups $\GL_n(\Z_p) \subset \GL_n(\mathbb{Q}_p)$ and (equivalently) combinatorics of abelian $p$-groups, see \cite[Chapters II, III, V]{Macd}. Previous to \cite{VP21limits}, Hall-Littlewood polynomials had been connected to the Cohen-Lenstra measure in \cite{L}, following their connection to an essentially equivalent measure arising in random matrix theory over finite fields in \cite{Fulman1999probabilistic} (see also \cite{Fulman2014cohen}). Recent applications to $p$-adic random matrix theory include \cite{DJ,Fulman2016hall,FK,VP21limits,VP22hall,VP22+universal}.

After \cite[Corollary 3.4]{VP21limits}, the subsequent work \cite[Theorem 1.4]{VP22hall} further gave explicit elementary formulas for this distribution, not featuring Hall-Littlewood polynomials. However, a more structural interpretation of these formulas was still lacking. Such an interpretation is furnished by the explicit group-theoretic formulation afforded by our Theorem \ref{theorem:seq_aut_intro}, finally placing the distribution in the context of similar ``$1/\#\Aut$'' distributions which have appeared previously in integer and $p$-adic random matrix theory. 
%Our results Theorem \ref{theorem:main:joint} and Theorem \ref{theorem:seq_aut_intro} give equivalent group-theoretic formulations of this distribution, and the latter finally places it in the context of similar ``$1/\#\Aut$'' distributions which have appeared previously in integer and $p$-adic random matrix theory. 

We have phrased our results in the group-theoretic language above, but Hall-Littlewood tools continue to be useful in our computations for \eqref{item:univ} of the moments of the limiting distributions, in Section \ref{sect:Haar}. To this end, Section \ref{sec:HL_and_groups} states many basic results translating between Hall-Littlewood and group-theoretic notation, and some purely group-theoretic results, all of which are not difficult to derive from \cite{Mac} but many of which we are not aware of in the random matrix literature. We hope that the dictionary we give there, between Hall-Littlewood formulas and moments of abelian groups and maps between them, will be useful in the field beyond our matrix product setting. We note also that the analogies between cokernels and singular values mentioned above are somewhat cleaner with cokernels of $p$-adic---rather than integral---matrices, with structurally identical formulas appearing in both settings in terms of either Hall-Littlewood polynomials or the analogous special functions on the complex side, see \cite{VP21limits} and the references therein.

For \eqref{item:asymp}, as mentioned, our main contributions are Theorem \ref{theorem:surmoment:k} and Theorem \ref{theorem:jointsur:k}. Our proofs of these results focus on the evolution of ``code" and ``non-code" vectors after the application of each random matrix $M_i$ in the product. Roughly speaking, for a code vector $\Bv$, the vector $M_i\Bv$ is close to being a random uniform vector, and hence the main contribution in the moment computation comes from these vectors. For the non-code vectors $\Bv$, the laws of $M_i\Bv$ are intractable, but fortunately we can avoid this by relying on the fact that most vectors are codes. In a way, this approach is similar to \cite{Ng-L}, where similar dynamical aspects of ``structured" and ``non-structured" vectors were studied. The evolution in the joint distribution setting is more complicated, as one has to keep track of many code and non-code vectors at the same time. As such, for expository purposes we will present the simpler case $k=2$ first, and then use induction to proceed further.  
    
Lastly, for (iii) our contributions are Theorem \ref{theorem:t:distribution}, Theorem \ref{theorem:t:distribution:joint} and Theorem \ref{t:distribution:joint'}. Although our approach mainly follows \cite[Theorem 8.3]{W0}, these results, especially Theorem \ref{theorem:t:distribution:joint} and Theorem \ref{t:distribution:joint'} require some non-trivial modifications as our focus is on the joint moments, where the growth rates are not straightforward to check. We hope that our joint moment comparison result, together with the developments in \cite{W0,W1} (see also \cite{WICM}), will provide useful tools to prove universality.

\subsection{Plan of paper.} In Section \ref{sect:support} we state many basic definitions and results from \cite{W1} pertaining to the moment method for abelian groups. In Sections \ref{section:prod:moments} and \ref{section:joint:moments} we compute the moments and joint moments of matrix products, needed for Theorems \ref{theorem:main:prod} and \ref{theorem:main:joint} respectively (while the latter theorem implies the former, for simplicity of exposition we usually prove needed results for the former first). General background on Hall-Littlewood polynomials and processes is in Section \ref{sec:HL_background}, and we relate it to abelian $p$-groups in Section \ref{sec:HL_and_groups}. We use this to compute the moments and joint moments of the limit distributions of Theorems \ref{theorem:main:prod} and \ref{theorem:main:joint} in Section \ref{sect:Haar}. In Sections \ref{sect:comparison} and \ref{sect:comparison:2} we combine these ingredients to prove Theorems \ref{theorem:main:prod} and \ref{theorem:main:joint} respectively, along with their analogues for $\Z_p$. In Section \ref{sec:automorphisms} we set up and prove Theorem \ref{theorem:seq_aut_intro}. Finally, in Section \ref{sec:fp} we reduce to $\F_p$ and prove Theorem \ref{thm:prod:rank}.

\subsection{Acknowledgements.}

We thank Melanie Matchett Wood for helpful discussions and for asking about interpretations of the distribution of \cite[Corollary 3.4]{VP21limits} in terms of automorphisms, and the anonymous referees for many helpful questions and comments. RVP also thanks Alexei Borodin for discussions and feedback, Alisa Knizel for asking the same question about automorphisms, and Oron Propp for helpful discussions on characterizing automorphism classes of sequences of modules. HN was supported by NSF CAREER grant DMS-1752345, and RVP was supported by an NSF Graduate Research Fellowship under grant \#$1745302$.

\section{Supporting lemmas}\label{sect:support}
Throughout this section fix $a \in \BBZ_{> 1}$ and set $R= \Z/a\Z$. Let $V=R^n$ with standard basis $\Bv_i,1 \leq i \leq n$. For $\sigma \subset [n]$ we denote by $V_{\sigma^c}$ the submodule generated by $\{\Bv_i: i \in \sigma^c\}$. Throughout the paper, to declutter notation we will write $(x_1,\ldots,x_n) \in R^n$ for usual (column) vectors, and similarly for vectors in e.g. $G^n$ where $G$ is a group, rather than using the notation $(x_1,\ldots,x_n)^T$.
%For each $F \in \Hom(V,G)$ we might interpret as $F=(f_1,\dots, f_n)$ where $f_i:= F(\Bv_i)$.

\begin{definition}\label{def:R_alpha_balanced}
Given real $\alpha \in (0,1/2]$, we say an $R$-valued random variable $\xi$ is \emph{$\al$-balanced} if for every prime $p|a$ we have
\begin{equation}\label{eqn:alpha_R}
\max_{r \in \Z/p\BBZ} \P(\xi\equiv r \pmod{p}) \le 1-\al.
\end{equation}
\end{definition}

Clearly if $\xi$ is a $\Z$-valued $\alpha$-balanced random variable as in Definition \ref{def:alpha_balanced}, then $\xi \pmod{a}$ is an $\alpha$-balanced $R$-valued random variable as in Definition \ref{def:R_alpha_balanced}. Hence the random matrices of Theorems \ref{theorem:main:prod} and \ref{theorem:main:joint}, reduced modulo $a$, have iid $\alpha$-balanced entries in $R$. From this section through Section \ref{section:joint:moments} we will work in this setting, and work with abelian groups $G$ with exponent dividing $a$ (i.e. $R$-modules). Most of the results below are from \cite{W1}.

\subsection{Codes} 
\begin{definition} Given $w \le n$, we say that $F \in \Hom(V,G)$ is a code of distance $w$ if for every $\sigma \subset [n]$ with $|\sigma| <w$ we have $F (V_{\sigma^c})=G$.
\end{definition}

Sometimes it is convenient to identify $F$ with the vector $(F(\Bv_1),\dots, F(\Bv_n)) \in G^n$, and we will usually abuse notation and view $F$ as a vector rather than a map. In particular, if $X = (x_1,\ldots,x_n) \in R^n$ is a vector, we write $\lang F, X \rang := \sum_{i=1}^n x_i F(\Bv_i)$; note this is not a usual dot product because $(F(\Bv_1),\dots, F(\Bv_n)) \in G^n$ and $(x_1,\ldots,x_n) \in R^n$ live in different spaces, though the formula is the same. If $M$ is an $n \times n$ matrix with entries in $R$, then for any $R$-module $G$, $M$ defines a linear map $G^n \to G^n$ by usual matrix multiplication, and we write $MF$ for the image of the vector $(F(\Bv_1),\dots, F(\Bv_n)) \in G^n$ under this map.

It is convenient to work with codes because the random walk $S_k = \sum_{i=1}^k x_i F(\Bv_i)$ (in discrete time indexed by $k=1,2,\ldots,n$) spreads out in $G$ very fast, as the following lemma shows. 

\begin{lemma}\cite[Lemma 2.1]{W1}\label{lemma:code:single:1} Assume that $x_i\in R$ are iid copies of $\xi$ satisfying \eqref{eqn:alpha_R}. Then for any code $F$ of distance $\delta n$ and any $g\in G$,
$$\left|\P(\lang F, X \rang  =g) - \frac{1}{|G|}\right| \le \exp(-\al \delta n/a^2),$$
where $X=(x_1,\dots,x_n)$.
\end{lemma}
In what follows, if not specified otherwise, $X$ is always understood as the random vector $(x_1,\dots, x_n)$ where $x_i$ are iid copies of $\xi$ satisfying \eqref{eqn:alpha_R} as in Lemma \ref{lemma:code:single:1}.

Using the above result, it is not hard to deduce the following matrix form.
\begin{lemma}\cite[Lemma 2.4]{W1}\label{lemma:code:single:matrix:1} Assume that the entries of $M$ of size $n$ are iid copies of $\xi$ satisfying \eqref{eqn:alpha_R}. For code $F$ of distance $\delta n$, for any vector $A \in G^n$ 
$$\left|\P(M F = A) - \frac{1}{|G|^{n}}\right| \le \frac{K \exp(-cn)}{|G|^{n}},$$
where $K,c$ depend on $a, G, \al$ and $\delta$.
\end{lemma}  

\begin{remark}
In our applications, $G$ will always be a fixed group (or perhaps summed over a finite collection of groups), so the dependence of the constants on $G$ which we allow in Lemma \ref{lemma:code:single:matrix:1} and similar results does not create any issue with our asymptotics.
\end{remark}

We will also need the following useful result.
\begin{lemma}\label{lemma:code:subgp} Let $\delta$ be sufficiently small. Assume that $F\in Hom(V,G)$ is a code of distance $\delta n$. Assume that the entries of the matrix $M$ of size $n$ are iid copies of $\xi$ satisfying \eqref{eqn:alpha_R}. Then for any $H \le G$
$$\P(M F \mbox{ is a code of distance $\delta n$ in $H_{}$}) = |H_{}|^{n} \frac{1+ O( \exp(-c''n))}{|G|^n},$$
where $c''$ depends on $a,G,H, \delta, \al$.
\end{lemma}
\begin{proof}[Proof of Lemma \ref{lemma:code:subgp}] First, by Lemma \ref{lemma:code:single:matrix:1}, for each $A$ a code of distance $\delta n$ of $H_{}$ we have 
$$\P(M F=A) =\frac{1+ O(K \exp(-c n))}{|G|^n}.$$
It remains to count the number of codes of distance $\delta n$ in $H_{}$. 

\begin{claim}\label{claim:code:count} Let $\CC(H)$ be the number of codes (defined as $F(V)$) of distance $\delta n$ in $H$. We have 
$$|\CC(H)| = (1+K'\exp(-c'_\delta n))|H|^n,$$
where $K'$ depends on $H$ and $c_\delta$ depends on $\delta$.
\end{claim}
\begin{proof} Let $g_1,\dots, g_n$ be chosen independently uniformly from $H$. For each $I\subset [n]$ an index set of size $n -\lfloor \delta n \rfloor$, and for each $H'$ a proper subgroup of $H$, let $\CE_{I,H'}$ be the event that $g_i \in H'$ for all $ i\in I$. Then clearly $\P(\CE_{I,H'}) = (|H'|/|H|)^{|I|}$. Taking a union bound over the choices of $I \in \binom{[n]}{n-\lfloor \delta n\rfloor}$ and over $H' < H$ we obtain a bound
$$K'(1/2)^{n-\lfloor \delta n \rfloor} \times \binom{n}{\lfloor \delta n \rfloor},$$
using that $|H'|/|H| \leq 1/2$. Since we assume that $\delta$ is sufficiently small, the above is bounded by $K'\exp(-c_\delta n)$ for some $K'$ and $c'_\delta$ as in the statement. 
% RC{I clarified origin of $1/2$ here (you can delete this comment after reading).}
 \end{proof}
To complete the proof of Lemma \ref{lemma:code:subgp} we have 
 $$\frac{1+ O(K \exp(-c n))}{|G|^n} \times  (1+O(K'\exp(-c_\delta'n))|H_{}|^n =  |H_{}|^{n} \frac{1+ O( \exp(-c'' n))}{|G|^n}.$$ 
\end{proof}

\subsection{Non-codes} Next, for non-code $F$, the random walk $\lang F, X \rang $ does not converge quickly to the uniform distribution on $G$. However it is likely to be uniform over the subgroup where the restriction of $F$ is a code. 

\begin{definition}
For $D =\prod_i p_i^{e_i}$ let 
$$\ell(D):= \sum_i e_i.$$ 
\end{definition}

In all results introduced below we remark that $F$ is not necessarily a surjection. 

\begin{definition}\label{def:depth} For a real $\delta>0$, the $\delta$-depth of $F \in \Hom(V,G)$ is the maximal positive integer $D$ such that there exists $\sigma \subset [n]$ with $|\sigma| < \ell(D) \delta n$ such that $D = |G/F(V_{\sigma^c})|$, or is 1 if there is no such $D$.
 \end{definition}
So roughly speaking the $\delta$-depth measures the maximum of $|G/F(V_{\sigma^c})|$ over $\sigma$ of size significantly smaller than $\delta n$. The depth is large if there exists such $\sigma$ where $F(V_{\sigma^c})$ is a small subgroup of $G$. The reason for this definition of depth is the following lemma, which shows that depth encodes how much one has to restrict $F$ to obtain a code.

\begin{lemma}\label{lemma:restrict_to_code}
If $F \in \Hom(V,G)$ has $\delta$-depth $D>1$, and $\sigma \subset [n]$ is such that $D = |G/F(V_{\sigma^c})|$ and $|\sigma| < \ell(D) \delta n$, then the restriction $F|_{V_{\sigma^c}} \in \Hom(V_{\sigma^c}, F(V_{\sigma^c}))$ is a code of distance $\delta  |\sigma^c|$.
\end{lemma}
\begin{proof}
Suppose for the sake of contradiction that $F|_{V_{\sigma^c}}$ is not a code of distance $\delta |\sigma^c|$. Then there exists a set $\eta \subset \sigma^c$ with 
 $$|\eta| < \delta |\sigma^c| $$
such that $\Im(F|_{V_{(\eta \cup \sigma)^c}}) \subsetneq \Im(F_{V_{\sigma^c}})$. Hence $\tilde{D} := [G:\Im(F|_{V_{(\eta \cup \sigma)^c}})] > D$, and of course $D | \tilde{D}$. So
$$|\eta \cup \sigma| < \delta(n -|\sigma|) + |\sigma| = \delta n +(1-\delta)|\sigma| < \delta n + (1-\delta) \ell(D) \delta n <\delta (\ell(D)+1) n$$
and $\ell(\tilde{D}) \geq \ell(D)+1$, therefore 
$$|\eta \cup \sigma| < \ell(\tilde{D}) \delta n.$$
But this means that $\tilde{D}$ satisfies the condition in the definition of depth, and is larger than $D$, contradicting maximality, which completes the proof.
\end{proof}

\begin{lemma}\cite[Lemma 2.6]{W1}\label{lemma:non-code:count:1} 
The number of $F\in \Hom(V,G)$ with depth $D$ is at most
$$K \binom{n}{\lceil \ell(D) \delta n\rceil -1} |G|^n D^{-n+\ell(D) \delta n},$$
where $K$ depends on $a$ and $G$. 
\end{lemma}

\begin{lemma}\label{lemma:non-code:single:1} 
Let $F \in \Hom(V,G)$ have $\delta$-depth $D>1$ and $|G/F(V)|<D$. Then for any $g\in G$
$$\P(\lang F, X \rang   = g) \le (1-\al) \left(\frac{D}{|G|} + \exp(-\al \delta n/a^2)\right).$$

\end{lemma}
We remark that the assumption above is automatically true if $F$ is a surjection. This result is different from \cite[Lemma 2.7]{W1} in that $g$ is any element instead of just 0.

\begin{proof} We follow the proof of \cite[Lemma 2.7]{W1}. Pick $\sigma \subset [n]$ with $|\sigma| < \ell(D) \delta n$ such that $D=|G/F(V_{\sigma^c})|$. Let $H=F(V_{\sigma^c})$. As $|G/F(V)|<D$, we have $\sigma \neq \emptyset$. 
We write 
\begin{align*}
\P(\lang F, X \rang  =g) &= \P\left(\sum_{i\in \sigma} x_i f_i +\sum_{i\in \sigma^c} x_i f_i =g\right) = \P\left(\sum_{i\in \sigma} x_i f_i  \in H_g \wedge \sum_{i\in \sigma^c} x_i f_i = g- \sum_{i\in \sigma} x_i f_i \right) \\
&= \P\left(\sum_{i\in \sigma} x_i f_i  \in H_g\right) \P\left(\sum_{i\in \sigma^c} x_i f_i = g- \sum_{i\in \sigma} x_i f_i | \sum_{i\in \sigma} x_i f_i  \in H_g\right),
\end{align*}

where $H_g$ is the coset of $H$ containing $g$. Now as $|G/F(V)|<D$, there exists $i_0 \in \sigma$ such that $f_{i_0} \notin H$. Since $x_{i_0}$ is $\al$-balanced, for any fixed values of $x_i, i \in \sigma \setminus i_0$ we have using the randomness of $x_{i_0}$ that
$$\P_{x_{i_0}}\left(\sum_{i\in \sigma }x_i f_i \in H_g\right)  \le 1-\al.$$

Furthermore, by Lemma \ref{lemma:restrict_to_code} $F(V_{\sigma^c})$ is a code of distance $\delta n$ over $H$. Hence
 $$\left|\P\left(\sum_{i\in \sigma^c} x_i f_i = g- \sum_{i\in \sigma^c} x_i f_i | \sum_{i\in \sigma} x_i f_i  \in H_g\right) - \frac{1}{|H|}\right| \le  \exp(-\al \delta n/a^2).$$
Putting together we have 
$$\P\left(\sum_{i\in \sigma }x_i f_i \in H_g \wedge \sum_{i \in \sigma^c} x_i f_i =g-\sum_{i\in \sigma }x_i f_i\right) \le (1-\al) \left(\frac{1}{|H|} + \exp(-\al \delta n/a^2)\right).$$
\end{proof}

Using this result, we can obtain similar bound in matrix form, the same way \cite[Lemma 2.8]{W1} was deduced from \cite[Lemma 2.7]{W1}.
\begin{lemma}\label{lemma:non-code:1:matrix}
 If $F \in \Hom(V,G)$ has $\delta$-depth $D>1$ and $|G/F(V)|<D$ as in the previous lemma, then for any $A \in G^n$,
$$\P(MF = A) \le K \exp(-\al n)\frac{D^n}{|G|^n},$$
where $K$ depends on $a, G, \al$ and $\delta$.
\end{lemma}
\begin{proof} By Lemma \ref{lemma:non-code:single:1},
\begin{align*}
\P(MF = A)  &= \P(\lang F, X_i \rang  = a_i, 1\le i\le n) \le  \left((1-\al) \left(\frac{D}{|G|} + \exp(-\al \delta n/a^2)\right)\right)^n.
\end{align*}
This is bounded above by 
$$\exp(n \log(1-\alpha)) \left(\frac{D}{|G|}\right)^n \exp(n(|G|/D)\exp(-\alpha\delta n/a^2)) \leq K \exp(-\alpha n) \left(\frac{D}{|G|}\right)^n,$$
where we have Taylor expanded the logarithm inside the first exponential and kept only the first term (the rest are also negative), and $K$ is the maximum over $n \geq 1$ of $e^{n(|G|/D)\exp(-\alpha\delta n/a^2)}$. 
\end{proof}

To complete this section we introduce two more definitions that will be crucial to our work. 

\begin{definition}\label{def:n_k}
For a given $k\ge 0$ we let $n_k(G)$ denote the number of sequences of nested subgroups 
$$0=H_0\le H_1 \le H_2 \le \dots \le H_{k-1} \le H_k=G.$$ 
\end{definition}

For projections onto direct summands, when $g=(g_1,\dots, g_k) \in G_1\oplus \dots \oplus G_k$ we write $\pi_i(g) =g_i$, and $\pi_I(g)=(g_i, i\in I)$.

\begin{definition}\label{def:m_k}
For a given $k\ge 0$ and given $k$ finite abelian groups $G_1,\dots, G_k$ we let $m_k(G_1, \dots, G_k)$ denote the number of sequences $(H_1, H_2,\dots, H_k)$ such that $H_k = G_k$ and for each $i=1,\ldots,k-1$,
\begin{itemize}
\item $H_i \le G_i \oplus \dots \oplus G_k$,
\item $\pi_i(H_i) = G_i$, and
\item
$\pi_{\{i+1,\dots, k\}}(H_i) \le H_{i+1}$. 
\end{itemize}
\end{definition}
In the simple case $k=2$, $m_2(G_1, G_2)$ is just the number of subgroups $G' \in G_1 \oplus G_2$ such that $\pi_1(G')=G_1$. Furthermore, for each $H \le G_i \oplus \dots \oplus G_k$ such that $\pi_i(H) = G_i$, we will abuse notation to let $m_{k-i}(H)$ denote the number of sequences $(H_i=H,H_{i+1},\dots, H_k)$ such that $H_k = G_k$ and, similarly to the above, for each $j=i,\ldots,k-1$ we have $H_j \le G_j \oplus \dots \oplus G_k$ and $\pi_j(H_j) = G_j$, as well as $\pi_{\{j+1,\dots, k\}}(H_j) \le H_{j+1}$. By this, we see that 
\begin{equation}\label{eqn:kk-1}
m_k(G_1, \dots, G_k) = \sum_{\substack{H_1\le G_1 \oplus \dots \oplus G_k \\ \pi_1(H_1)=G_1}} m_{k-1}(H_1)
\end{equation}
and in general for each $H_i$ (such that $H_i \le G_i \oplus \dots \oplus G_k$ and $\pi_i(H_i) = G_i$) we have
\begin{equation}\label{eqn:kk-1'}
m_{k-i}(H_i) = \sum_{\substack{\pi_{\{ i+1,\dots, k\}}(H_{i}) \le H_{i+1}\le G_{i+1} \oplus \dots \oplus G_k \\ \pi_{i+1}(H_{i+1})=G_{i+1}}} m_{k-i-1}(H_{i+1}).
\end{equation}

Finally, note also that as $H_1 \le G_1\oplus H_2 \le \dots \le G_1 \oplus \dots \oplus G_{k-1} \oplus H_k \le  G_1 \oplus \dots \oplus G_{k}$, we have
\begin{equation}\label{eq:m_n_bound}
m_k(G_1, \dots, G_k) \le n_k(G_1 \oplus \cdots \oplus  G_k).
\end{equation}

\section{Counting surjections for Theorem \ref{theorem:main:prod}}\label{section:prod:moments}

Let $a,R,V$ be as in the previous section. Throughout the section we write $\Hom(A,B)$ and  $\Sur(A,B)$ for the set of homomorphisms and surjective homomorphisms, respectively, from $A$ to $B$. 

\subsection{Set-up}\label{S:mom}

We know from \cite{W1} that to understand the distribution of $\cok (M)$, it suffices to determine the ``moments"\footnote{We refer the reader to \cite{W0,W1} for the history of these statistics.} of $\cok(M)$, i.e. the quantities $\E[\#\Sur(\cok (M), G)]$ for each finite abelian group $G$. 
To investigate each such moment, we recognize that each such surjection lifts to a surjection $V\ra G$ and so we have
\begin{equation}\label{E:expandF}
\E[\#\Sur(\cok (M), G)]=\sum_{F\in \Sur(V,G)}  \P(F(MV)=0 \mbox{ in $G$})= \sum_{F\in \Sur(V,G)}  \P(MF=0 \mbox{ in $G$}),
\end{equation}
where we view $F$ as a column vector $F=(F(\Bv_1),\dots, F(\Bv_n)) \in G^n$. By the independence of columns, we have
$$
\P(MF=0)= \prod_{j=1}^n \P(\lang F,X_j \rang)=0),
$$
where $X_1,\dots, X_n$ are rows of $M$. So in the case of a single matrix, ones must estimate these probabilities $\P(F(X_j)=0)$, which give the desired moments. In our situation we have random matrices $M_1, M_2,\dots, M_k$, and want to study $\P(M_1M_2 \cdots M_k  F=0)$ for surjections $F:V \to G$.

Recall $n_k(G)$ from Definition \ref{def:n_k}. Our key result in this section is a generalization of Lemma \ref{lemma:code:single:matrix:1} and Lemma \ref{lemma:non-code:1:matrix} (though in what comes later we will not use the result itself as stated below, but actually use several intermediate steps of its proof).
 
\begin{proposition}\label{prop:single:k} With the same assumption as in Theorem \ref{theorem:main:prod}, the following holds for $\delta$ sufficiently small: there exist $c,K$ depending on $k,\al, G,a,\delta$ such that

\begin{enumerate}[(i)]
\item  (Code) assume that $F$ spans $H_k=G$ and is a code of distance $\delta n$ in $G$. Then 
$$\Big |\P( M_1\cdots M_k  F=0) -\frac{n_k(G)}{|G|^n} \Big |\le  K \frac{\exp(-c n)}{|G|^n}.$$ 
\item (Non-code) Assume that $F$ spans $H_k=G$ and the $\delta$-depth of $F$ is $D_k\ge 2$. Then 
$$\P(M_1\cdots M_k  F=0) \le K \exp(-\al n/2)  \frac{D_k^n}{|G|^n}.$$
\end{enumerate}
 
\end{proposition}
\begin{proof}%[Proof of Proposition \ref{prop:single:k}] 
In what follows $K$ and $c$ may vary, and the implied constants in $O(.)$ are allowed to depend on $k,\al, G,a$ and $\delta$.  

We prove (i) and (ii) together by induction on $k$, assuming both (i) and (ii) hold for $k-1$ as the inductive hypothesis. When $k=1$, (i) and (ii) follow from Lemma \ref{lemma:code:single:matrix:1} and Lemma \ref{lemma:non-code:1:matrix} respectively. Next we consider $k\ge 2$.

{\bf Codes.} We first prove (i) by working with $F$ a code of distance $\delta n$. 

Let $H_{k-1}$ be a subgroup of $H_k=G$. We consider the event (in the $\sigma$-algebra generated by $M_k$) that $M_kF$ spans $H_{k-1}$ in two ways 
\begin{enumerate}[(1)]
\item $M_k F$ is a code of distance $\delta n$ in $H_{k-1}$; 
\item $M_k F$ is not a code of distance $\delta n$, and hence has $\delta$-depth $D_{k-1}\ge 2$ in $H_{k-1}$.
\end{enumerate}
For the first case, we apply the induction hypothesis for (i) to obtain
$$\Big|\P_{M_1,\dots, M_{k-1}}(M_1\dots M_{k-1}(M_k F)=0|  \text{$M_k F$ is $\delta n$ code in $H_{k-1}$}) -\frac{n_{k-1}(H_{k-1})}{|H_{k-1}|^n}\Big|\le K  \frac{\exp(-c n)}{|H_{k-1}|^n}.$$ 
For the second case, we also apply the induction hypothesis for (ii) to obtain
$$\P_{M_1,\dots, M_{k-1}}\left(M_1\dots M_{k-1}(M_k F)=0|  \text{ $M_k F$ has $\delta$-depth $D_{k-1}\ge 2$ in $H_{k-1}$}\right)\le K \exp(-c n)n_{k-1}(H_{k-1}) \frac{D_{k-1}^n}{|H_{k-1}|^n}.$$ 
Hence
\begin{align*}
&\ \P_{}\left(\prod_{i=1}^k M_i F=0, \text{$M_k F$ spans $H_{k-1}$}\right) =\P_{M_1,\dots, M_{k-1}}(M_1\dots M_{k-1}(M_k F)=0|  \text{$M_k F$ is $\delta n$-code in $H_{k-1}$})  \times \\
& \times \P(\text{$M_k F$ is $\delta n$-code in $H_{k-1}$})\\ 
& +\sum_{\substack{D_{k-1} \geq 2 \\ D_{k-1} \big{\vert} |H_{k-1}|}} \P_{M_1,\dots, M_{k-1}}(M_1\dots M_{k-1} (M_k F)=0| \text{ $M_k F$ has $\delta$-depth $D_{k-1}$ in $H_{k-1}$}) \\
& \times \P(\text{$M_k F$ has $\delta$-depth $D_{k-1}$ in $H_{k-1}$}) \\
 &=: S_1(H_{k-1}) + \sum_{\substack{D_{k-1} \geq 2 \\ D_{k-1} \big{\vert} |H_{k-1}|}}S_2(H_{k-1}, D_{k-1}).
\end{align*}

For the first sum, by Claim \ref{claim:code:count}, and then by Lemma \ref{lemma:code:subgp} and the inductive hypothesis for (i) we have
\begin{align*}
S_1(H_{k-1})&= \left(\frac{n_{k-1}(H_{k-1})}{|H_{k-1}|^n} + O\left(\frac{n_{k-1}(H_{k-1}) \exp(-c n)}{|H_{k-1}|^n}\right)\right)  |\CC(H_{k-1})| \frac{1+ K \exp(-c n)}{|G|^n}\\
&= \left(\frac{n_{k-1}(H_{k-1})}{|H_{k-1}|^n} + O\left(\frac{n_{k-1}(H_{k-1}) \exp(-c n)}{|H_{k-1}|^n}\right)\right)   |H_{k-1}|^n \frac{1+ K \exp(-c n)}{|G|^n}\\
&=  \frac{n_{k-1}(H_{k-1})}{|G|^n} +  O\left(\frac{n_{k-1}(H_{k-1}) \exp(-cn)}{|G|^n}\right).
\end{align*}
For the second sum, for each $D_{k-1}$ we apply Lemma \ref{lemma:non-code:count:1} and Lemma \ref{lemma:code:single:matrix:1} to bound
\begin{equation}\label{eq:some_depth_prob_bound}
\P(\text{$M_k F$ has $\delta$-depth $D_{k-1}$ in $H_{k-1}$}) \leq K' \binom{n}{\lceil \ell(D_{k-1}) \delta n\rceil -1} |H_{k-1}|^n D_{k-1}^{-n+\ell(D_{k-1}) \delta n}  \frac{1+ K'' \exp(-c n)}{|G|^n}
\end{equation}
and apply the inductive hypothesis for (ii) to bound 
\begin{equation}\label{eq:ind(ii)}
\P_{M_1,\dots, M_{k-1}}(M_1\dots M_{k-1} (M_k F)=0| \text{ $M_k F$ has $\delta$-depth $D_{k-1}$ in $H_{k-1}$}) \le K \exp(-\al n/2)\frac{D_{k-1}^n}{|H_{k-1}|^n}.
\end{equation}
Combining \eqref{eq:some_depth_prob_bound} with \eqref{eq:ind(ii)} yields
\begin{align*}
S_2(H_{k-1},D_{k-1})& \le  K \exp(-\al n/2)\frac{D_{k-1}^n}{|H_{k-1}|^n}  \times K' \binom{n}{\lceil \ell(D_{k-1}) \delta n\rceil -1} |H_{k-1}|^n D_{k-1}^{-n+\ell(D_{k-1}) \delta n}  \frac{1+ K'' \exp(-c n)}{|G|^n} \\
&=O\left( \frac{ \exp(-\al n/4) }{|G|^n}\right),
\end{align*}
where for the second line we recall that $\delta$ was chosen sufficiently small and $n$ is sufficiently large. % 

Summing over divisors $D_{k-1}$ of $|H_{k-1}|$,
\begin{align*}
\sum_{D_{k-1} \geq 2, D_{k-1}\big{\vert} \left\vert H_{k-1}\right\vert} S_2(H_{k-1},D_{k-1}) =O\left( \frac{ \exp(-\al n/4) }{|G|^n}\right).
\end{align*} 
Summing over $H_{k-1}\le H_k$ we thus obtain
\begin{align}\label{eqn:code:M_k:1}
\begin{split}
\P_{}\left(\prod_{i=1}^k M_i F=0\right) &=\sum_{H_{k-1}} \P_{}\left(\prod_{i=1}^k M_i F=0 \wedge \text{$M_k F$ spans $H_{k-1}$}\right)   \\
& = \sum_{H_{k-1}} S_1(H_{k-1}) + \sum_{\substack{D_{k-1} \geq 2 \\ D_{k-1} \big{\vert} |H_{k-1}|}}S_2(H_{k-1}, D_{k-1})  \\
& = \frac{n_k(G)}{|G|^n}+O\left(\frac{\exp(-cn)}{|G|^n}\right) + O\left( \frac{ \exp(-\al n/4) }{|G|^n}\right).
\end{split}
\end{align}

completing the estimates for codes.

{\bf Non-codes.} We next prove (ii) by working with $F$ of $\delta$-depth $D_k\ge 2$, where $D_k$ also divides $|H_k|=|G|$. Let $H_{k-1}$ be a subgroup of $H_k$. Similarly to the previous part, we again compute the probability that $M_kF$ spans $H_{k-1}$ in the two possible ways:

\begin{enumerate}[(1)]
\item $M_k F$ is a code of distance $\delta n$ in $H_{k-1}$; 
\item $M_k F$ is not a code of distance $\delta n$, and hence has $\delta$-depth $D_{k-1}\ge 2$ in $H_{k-1}$.
\end{enumerate}

For the first case, the probability with respect to $M_k$ is bounded by
$$\P_{M_k}(   \text{$M_k F$ is a code of distance $\delta n$ in $H_{k-1}$} ) \le K |H_{k-1}|^n   \exp(-\al n)  \frac{D_{k}^n}{|G|^n}$$
by bounding the number of codes by $|H_{k-1}|^n$ and applying Lemma  \ref{lemma:non-code:1:matrix}. Hence, by induction and by the independence of $M_1,\dots, M_k$
\begin{align*}
& \P_{}(M_1\cdots M_k  F=0 \text{ and $M_k F$ is code of distance $\delta n$ in $H_{k-1}$}) \\
& \le \left(\frac{n_{k-1}(H_{k-1})}{|H_{k-1}|^n}+ K\frac{\exp(-c n)}{|H_{k-1}|^n}\right) \times  K |H_{k-1}|^n   \exp(-\al n)  \frac{D_{k}^n}{|G|^n} \\
&=  O\left(\exp(-\al n) \frac{n_{k-1}(H_{k-1}) D_{k}^n}{|G|^n}\right).
\end{align*}
Summing over the subgroups $H_{k-1}$, we obtain
\begin{equation}\label{eqn:non-code:M_k:1}
\P_{}(M_1\cdots M_k  F=0 \text{ and $M_k F$ is a $\delta n$ code in $H_{k-1}$ for some $H_{k-1}$}) =O\left( \exp(-\al n) \frac{n_{k}(G) D_{k}^n}{|G|^n}\right).
\end{equation}

For the second case (2), the probability with respect to $M_k$, by Lemma \ref{lemma:non-code:count:1} and Lemma  \ref{lemma:non-code:1:matrix}, is bounded by
$$\P_{M_k}(   \text{$M_k F$ is of $D_{k-1}$-depth in $H_{k-1}$} ) \le K\binom{n}{\lceil \ell(D_{k-1}) \delta n\rceil -1} |H_{k-1}|^n D_{k-1}^{-n+\ell(D_{k-1}) \delta n}  \times  K'\exp(-\al n)  \frac{D_{k}^n}{|G|^n} .$$
Hence, by induction (applied to $M_1 \cdots M_{k-1}$ with the starting vector $(M_kF)$)
\begin{align*}
& \P_{}(M_1\cdots M_k  F=0 \text{ and $M_k F$ is of $\delta$-depth $D_{k-1}$ in $H_{k-1}$}) \\
&= \P_{}(M_1\cdots M_k  F=0 |\text{$M_k F$ is of $\delta$-depth $D_{k-1}$ in $H_{k-1}$}) \cdot \P(\text{$M_k F$ is of $\delta$-depth $D_{k-1}$ in $H_{k-1}$})\\
& \le K \exp(-\al n/2) n_{k-1}(H_{k-1}) \frac{D_{k-1}^n}{|H_{k-1}|^n} \times K' \binom{n}{\lceil \ell(D_{k-1}) \delta n\rceil -1} |H_{k-1}|^n D_{k-1}^{-n+\ell(D_{k-1}) \delta n}  \times  \exp(-\al n)  \frac{D_{k}^n}{|G|^n}\\
& =O\left( \exp(-\al n)n_{k-1}(H_{k-1})  \frac{D_{k}^n}{|G|^n}\right),
\end{align*}
provided that $\delta$ was chosen sufficiently small and $n$ is sufficiently large.

Summing over $D_{k-1}$ a divisor of $|H_{k-1}|$, and then over the subgroup $H_{k-1}$
\begin{align}\label{eqn:non-code:M_k:2}
\begin{split}
& \P_{}(M_1\cdots M_k  F=0 \text{ and $M_k F$ is of $D_{k-1}$-depth in $H_{k-1}$ for some $D_{k-1}\ge 2$ and subgroup $H_{k-1}$}) \\ 
& =O\left( \exp(-\al n)n_{k}(G)  \frac{D_{k}^n}{|G|^n}\right),
\end{split}
\end{align}
proving our upper bound for non-codes $F$.
\end{proof}

Using the proof of Proposition \ref{prop:single:k} above we obtain

\begin{theorem}[Asymptotic moments of matrix products]\label{theorem:surmoment:k} Let $a \geq 2$ and $R=\Z/a\Z$, $G$ be any finite abelian group whose exponent is divisible by $a$, and $M_1,\ldots,M_k$ be random matrices in $\Mat_n(R)$ with iid $\alpha$-balanced entries. Then
$$\Big |\E \left[\# \Sur(\cok(M_1\cdots M_k ),G)) - n_k(G)\right]\Big| \le Ke^{-cn},$$
for some $K,c$ depending on $k,\al, G,a$.
\end{theorem}

\begin{proof}[Proof of Theorem \ref{theorem:surmoment:k}] By \eqref{E:expandF} it suffices to show that 
\begin{equation}\label{eqn:prod:sum}
\Big|\sum_{F\in \Sur(V,G)}  \P(M_1\cdots M_k  F=0 \mbox{ in $G$}) - n_k(G)\Big| \le Ke^{-cn}.
\end{equation}
From \eqref{eqn:code:M_k:1} and Claim \ref{claim:code:count}, we sum over $F$ as codes of distance $\delta n$ in $G$ to obtain

\begin{align}\label{eqn:code:sum:1}
\begin{split}
\sum_{F \mbox{ code of distance $\delta n$ in $G$}}\P_{}\left(\prod_{i=1}^k M_i F=0\right)& = |G|^n(1+K\exp(-c n)) \times \left(\frac{n_k(G)}{|G|^n}+K'''\frac{\exp(-c n)}{|G|^n}\right)  \\
& = n_k(G)(1+ O(\exp(-c'n))).
\end{split}
\end{align}

From \eqref{eqn:non-code:M_k:1} and Lemma \ref{lemma:non-code:count:1}, for each $D_k$ as divisor of $|G|$, summing over non-codes $F$ 
%\footnote{it is true for any non-code $F\in \Hom(V,G)$, but certainly we just focus on $F\in \Sur(V,G)$.  
of $\delta$-depth $D_k$\begin{align}\label{eqn:noncode:code:1}
\begin{split}
& \P_{}(\mbox{ $\exists F$ of depth $D_k$ in $G$ and } M_1\cdots M_k  F =0 \text{ and $M_k F$ is a $\delta n$ code in $H_{k-1}$ for some $H_{k-1}$}) \\
&\le  K'' \binom{n}{\lceil \ell(D_{k}) \delta n\rceil -1} |G|^n D_{k}^{-n+\ell(D_{k}) \delta n}   \exp(-\al n) \frac{n_{k}(G) D_{k}^n}{|G|^n} =O( \exp(-\al n/2)),
\end{split}
\end{align}
provided that $\delta$ was chosen sufficiently small and $n$ is sufficiently large. 

Also, from \eqref{eqn:non-code:M_k:2} and Lemma \ref{lemma:non-code:count:1}, % for each $D_k$ as divisor of $|G|$, summing over non-codes $F$ of $\delta$-depth $D_k$ 
\begin{align}\label{eqn:noncode:noncode}
\begin{split}
& \P(\text{$\exists F$ of some depth $D_k$ in $G$ with } M_1\cdots M_k  F=0 \text{ and $M_k F$ is of depth $D_{k-1}$ in $H_{k-1}$} \\
&\text{ for some $D_{k-1}\ge 2, H_{k-1} \le H_k$}) \le K \binom{n}{\lceil \ell(D_{k}) \delta n\rceil -1} |G|^n D_{k}^{-n+\ell(D_{k}) \delta n}  \exp(-\al n)n_{k}(G)  \frac{D_{k}^n}{|G|^n}  \\
&= O(\exp(-\al n/2)),
\end{split}
\end{align}
again as $\delta$ is small  and $n$ is sufficiently large.

Summing \eqref{eqn:noncode:code:1}, \eqref{eqn:noncode:noncode} over all $D_k| |G|$, together with \eqref{eqn:code:sum:1} we obtain \eqref{eqn:prod:sum} as claimed.
\end{proof}

We will not use the following result later, but include it because it demonstrates how moment bounds for finite rings can imply them for infinite ones such as $\Z$ and $\Z_p$.

\begin{corollary}\label{cor:sur:uniform:asym} Let $M_1,\dots, M_k$ have iid entries in $\Z_p$ which are not constant modulo $p$. Then for any $p$-groups $G$ we have 
$$\Big |\E (\# \Sur(\cok(M_1\cdots M_k \Z_p^n),G)) - n_k(G)\Big| \le Ke^{-cn}$$
for some $K,c$ depending on $k,\al, G$, where $\alpha \in (0,1/2]$ is such that the matrix entries modulo $p$ are $\alpha$-balanced.

\end{corollary}
\begin{proof}
Let $p^L$ be the exponent of $G$. First note that for any abelian $p$-group $H$, 
\begin{equation}\label{eq:reduce_surj_simple}
\#\Sur(H,G) = \#\Sur(H/p^LH,G),
\end{equation}
as any surjection from $H$ to $G$ automatically annihilates $p^{L}H$. Note also, with the notation $\tilde{M} := M \pmod{p^L} \in \Mat_n(\Z/p^L\Z)$ for $M \in \Mat_n(\Z_p)$, that $\cok(\tilde{M}) = \cok(M)/p^L \cok(M)$. Combining with \eqref{eq:reduce_surj_simple} yields that  
\begin{equation}
\# \Sur(\cok(M_1\cdots M_k \Z_p^n),G) = \# \Sur(\cok(\tilde{M}_1\cdots \tilde{M}_k (\Z/p^{L}\Z)^n),G).
\end{equation}
The result now follows from Theorem \ref{theorem:surmoment:k} applied with $R=\Z/p^L\Z$ and matrices $\tilde{M}_1,\ldots,\tilde{M}_k$, which are $\alpha$-balanced since the entries are not constant modulo $p$.
\end{proof}

\section{Counting joint surjections for Theorem \ref{theorem:main:joint}}\label{section:joint:moments}

Recall that $R =\Z/a\Z$ where $a$ is a positive integer. Let $G_1,\dots, G_k$ be finite abelian groups whose exponents divide $a$. For matrices $M_1,\ldots,M_k \in \Mat_n(R)$ and $(F_1', \dots, F_k')$ surjections from the quotients $R^n/ M_1R^n,\dots, R^n/M_1\cdots M_k R^n$ to $G_1,\dots,G_k$ we can lift to a surjection tuple $(F_1, \dots, F_k)$ from $V=R^n$. So we have  
 to $G_1,\dots, G_k$ respectively. 
$$\E \left(\# \Sur(R^n/M_1R^n, G_1) \times \cdots \times \# \Sur(R^n/M_1\cdots M_k  R^n,G_k)\right)$$ 
$$= \sum_{(F_1,\dots,F_k) \in \Sur(V,G_1) \times \cdots \times \Sur(V,G_k)} \P( M_1 F_1=0 \mbox { in $G_1$} \wedge \cdots \wedge M_1 \cdots M_k F_k=0 \mbox{ in $G_k$}).$$
% We notice that $R^n/M_1\dots M_{i-1} R^n \le R^n/M_1\dots M_i R^n$ because
%$$R^n/M_1\dots M_i R^n = (R^n/M_1 \dots M_{i-1}R^n) (M_1 \dots M_{i-1}R^n/ M_1 \dots M_{i-1}(M_iR^n)).$$
% However this does not affect the choices of $G_1,\dots, G_k$. In any choice of these groups we always have $\# \Sur(R^n/M_1R^n, G_1), \# \Sur(R^n/M_1\cdots M_k  R^n,G_2)$ are all at least one.

Recall $m_k(G_1,\dots, G_k)$ from Definition \ref{def:m_k}. Our main goal for the proof of Theorem \ref{theorem:main:joint} is the following counting formula for the joint surjections.

\begin{theorem}[Asymptotic joint moments of matrix products]\label{theorem:jointsur:k} Let $M_1,\dots, M_k$ be independent random elements of $\Mat_n(R)$ with iid entries which are copies of some $\alpha$-balanced $\xi$. Let $G_1,\dots, G_k$ be finite abelian groups whose exponents divide $a$. We have
$$\Big|\E \left(\# \Sur(R^n/M_1R^n, G_1) \times \dots \times \# \Sur(R^n/M_1\cdots M_k  R^n,G_k)\right) - m_k (G_1,\dots, G_k)\Big| \le K e^{-cn},$$
for some $K,c$ depending on $k,\al, G_i,a,\delta$.
\end{theorem}

As before, this yields a corresponding moment result for $p$-adic matrices.

\begin{corollary}\label{cor:jointsur:uniform:asym} Let $M_1,\dots, M_k$ have iid entries in $\Z_p$ which are not constant modulo $p$. Then for any finite abelian $p$-groups $G_1,\dots, G_k$
$$\Big | \E \left(\# \Sur(\Z_p^n/M_1 \Z_p^n, G_1) \times \dots \times \# \Sur(\Z_p^n/M_1\cdots M_k  \Z_p^n,G_k)\right)-m_k (G_1, \dots, G_k) \Big| \le K e^{-cn},$$
for some $K,c$ depending on $k,\al, G_i$.
\end{corollary}
\begin{proof}
Argue as in Corollary \ref{cor:sur:uniform:asym} by letting $p^L$ be the maximum exponent of $G_1,\ldots,G_k$, reducing matrices modulo $p^L$, and applying Theorem \ref{theorem:jointsur:k}.
\end{proof}

% In what follows we denote 
% $$G^\ast = \Hom(G,R).$$

\subsection{Multidimensional setting} In this part we will give some preparation for the proof of Theorem \ref{theorem:jointsur:k}. Recall that we are interested in the event that $M_1F_1=0,\dots, M_1(M_2 \cdots M_k F_k)=0$. Hence it is natural to consider a more general related problem of determining, for some maps $F_1',\dots, F_k'$ (which correspond to $F_1,M_2F_2,\ldots,M_2\cdots M_k F_k$ in the previous example), what is the probability of the joint events $MF_1'=0,\dots, MF_k'=0$ for a matrix $M$ with iid $\alpha$-balanced entries.

\begin{definition}  We say that $F_1 \in \Hom(V,H_1),\dots, F_k\in \Hom(V,H_k)$ are a joint code of distance $\delta n$ with respect to $H' \le \bigoplus_{i=1}^k H_i$ if $F=(F_1,\ldots,F_k)$ is a code of distance $\delta n$ in $\Hom\left(V,H'\right)$.%for any $C_1,\dots, C_k \in (H_1^\ast,\dots, H_k^\ast) \bs \{0,\dots,0\}$, the vector $(C_1F_1(\Bv_i)+ \dots+ C_kF_k(\Bv_i))_{i=1}^n$ has at least $\delta n$ non-zero components in $R$.
\end{definition}

\begin{remark}\label{rmk:joint_code_not_code}
To avoid a potential point of confusion: if $(F_1,\ldots,F_k)$ are a joint code, it is \emph{not} in general true that the $F_i$ are individually codes with respect to $H_i$---they do not even have to be surjections, since the definition of joint code is with respect to some subgroup $H'$.
\end{remark}

Our first result is Lemma \ref{lemma:code:single:1} restated under the ``multidimensional" setting. 

\begin{lemma}\label{lemma:code:joint} Let $F_1,\dots, F_k$ be a joint code of distance $\delta n$ with respect to $H'\le \bigoplus_{i=1}^k H_i$, and $X$ a random vector in $R^n$ with iid $\alpha$-balanced entries. Then for any $(h_1,\dots, h_k) \in H'$,
$$\left|\P(\lang X, F_1 \rang =h_1 \wedge \dots \wedge  \lang X, F_k \rang = h_k) - \frac{1}{|H'|}\right| \le \exp(-\al \delta n/a^2).$$
\end{lemma}

\begin{definition}\label{def:projection} Given finite abelian groups $G_1,\dots, G_k$, for each $1\le i\le k$ let $\CG_{i;i+1,\dots,k}$ be the set of subgroups $G \le G_i \oplus G_{i+1} \oplus \dots \oplus G_k$ where the projection onto $G_i$ is the whole group. That is
\begin{equation}\label{eqn:proj}
\pi_i(G) := \{g_1\in G_i: \exists g_2 \in G_{i+1}\oplus \dots \oplus G_k \text{ such that } (g_1,g_2) \in G\}=G_i.
\end{equation}

\end{definition}
The reason we have this projection condition is that later $G$ is generated by $(F_i,M_{i+1}F_{i+1},\dots, M_{i+1}\cdots M_k F_k)$. By assumption $F_i: V \to G_i$ is surjective, and hence the projection onto $G_i$ of the group above is $G_i$ itself. We note that the sets $H_i$ from Definition \ref{def:m_k} belongs to $\CG_{i;i+1,\dots,k}$.

\begin{remark}\label{rmk:proj_F}
Since projections $\pi: G \to H$ appear frequently in this section, we will use the notation $\pi(F)$ for the vector in $H^n$ given by $(\pi \circ F(\Bv_1),\ldots,\pi \circ F(\Bv_n))$.
\end{remark}

Next, consider $G \le G_1 \oplus \dots \oplus G_k$, and  let $F=(F_1,\dots, F_k) \in \Hom(V,G)$. We recall from Definition \ref{def:depth} that the  $\delta$-depth of $F$ is the maximal $D$ such that there exist $\sigma$ and $H \le G$ such that
\begin{itemize}
\item $|\sigma| < \ell(D) \delta n$;
\vskip .1in
\item $F(V_{\sigma^c})= (F_1(V_{\sigma^c}),\dots, F_k(V_{\sigma^c}) )= H, [G:H]=D$;
\end{itemize}
If there is no such $D$, then the depth is 1. 

Lemma \ref{lemma:non-code:count:1} applied to $G$ yields % \RC{I believe this also needs to be generalized to where the target space is a subgroup of $G_1 \oplus \cdots \oplus G_k$, as I think it is used in that context on p21}
 \begin{lemma}[Number of tuples with given depth]\label{lemma:depth:count:2}\label{lemma:depth:count:k}
  The number of $(F_1,\dots, F_k)\in \Hom(V,G)$ with depth $D$ is at most
$$K \binom{n}{\lceil \ell(D) \delta n\rceil -1} |G|^n D^{-n+\ell(D) \delta n},$$
where $K$ depends on $a, G$.
 \end{lemma}

Let $G' \in \CG_{1;2,\dots,k}$ (defined in Definition \ref{def:projection}). Lemma \ref{lemma:non-code:single:1} stated for $G'$ implies
\begin{lemma}\label{lemma:non-code:joint:2}\label{lemma:non-code:joint:k} Let $M_2,\dots, M_k$ be fixed matrices of size $n \times n$.
Let $a$ be an integer with $a\ge 2$. Let $G_1,\dots, G_k$ be finite abelian groups with exponents dividing $a$. Let $\al>0$ and $\delta>0$ be real numbers. Let $G' \in \CG_{1;2,\dots,k}$, and let $F=(F_1,M_2F_2,\dots, M_2 \cdots M_k F_k) \in \Hom(V,G')$ have $\delta$-depth $D>1$ and $|G'/F(V)|<D$. Then for all $\al$-balanced random vectors $X$ valued in $V$, for any $(g_1,\dots, g_k) \in G'$
$$\P_X\left(\lang F_1, X \rang = g_1 \wedge \lang M_2F_2, X \rang =g_2 \wedge \dots \wedge \lang M_2\cdots M_k F_k, X \rang =g_k \right) \le (1-\al) (\frac{D}{|G'|} + \exp(-\al \delta n/a^2)).$$
\end{lemma}
Similarly, Lemma \ref{lemma:non-code:1:matrix}  stated for $G'$ implies the following matrix form.
\begin{lemma}\label{lemma:non-code:joint:matrix:2}\label{lemma:non-code:joint:matrix:k}  
With the same assumption as in Lemma \ref{lemma:non-code:joint:k}, for any $A \in {G'}^n$ and $n \times n$ matrix $M_1$ with iid $\alpha$-balanced entries,
$$\P_{M_1}(M_1F = A) \le K \exp(-\al n)\frac{D^n}{|G'|^n},$$
where $K$ depends on $a, G, \al$ and $\delta$.
\end{lemma}

%\HNC{I rewrote the case $k=2$ (just) a little. The bigger changes are in the general case where in Prop. \ref{prop:joint:k} I added \eqref{eq:joint:k'}} 

Now we turn to Theorem \ref{theorem:jointsur:k}. As in the previous section, in what follows the implied constants in $O(.)$ are allowed to depend on $k,\al, G,a,\delta$. For expository purposes we focus on $k=2$ first.

\subsection{Proof of Theorem \ref{theorem:jointsur:k} for $k=2$}\label{sub:k=2}  
We have
\begin{align}\label{eq:bigsum_k=2}
\E \left(\# \Sur(R^n/M_1R^n, G_1)  \# \Sur(R^n/M_1M_2 R^n,G_2)\right)  = \sum_{\substack{F_1\in \Sur(V,G_1)\\ F_2 \in \Sur(V,G_2)}}\P(M_1 F_1=0 \wedge M_1M_2F_2=0).
\end{align}
The remainder of the proof consists of analyzing this sum; let us now fix surjections $F_1: V \to G_1$ and $F_2: V \to G_2$. For any fixed matrix $M$, the image of $(F_1, M F_2)$ is some subgroup $G' \leq G= G_1 \oplus G_2$, and since $F_1$ is a surjection, we must have $\pi_1(G') = G_1$. We will consider the random map $(F_1, M_2 F_2): V \to G$ (recall $F_i$ are fixed but $M_2$ is random), and we note that the events $\{M_2: \Im(F_1,M_2 F_2) = G'\}$ are disjoint and 
$$\bigsqcup_{\substack{G' \le G= G_1 \oplus G_2 \\ \pi_1(G') = G_1}} \{M_2: \Im(F_1,M_2 F_2) = G'\} = \Mat_{n \times n}(R).$$
For each such $G'$, we further partition the event $\{M_2: \Im(F_1,M_2 F_2) = G'\}$ into two subcases, based on whether
%For each subgroup $G' \le G= G_1 \oplus G_2$ where $\pi_1(G)=G_1$, we consider the case that $(F_1,M_2 F_2)(V)$ spans $G'$ (i.e. it is a surjection over $G'$) and  
\begin{enumerate}[(a)]
\item $(F_1, M_2F_2)$ form a joint code of distance $\delta n$ with respect to $G'$;
\vskip .1in
\item  $(F_1, M_2F_2)$ is not joint code of distance $\delta n$, hence $(F_1, M_2 F_2)$ has $\delta$-depth $D>1$ over $G'$.
\end{enumerate}
Here $\delta$ is a sufficiently small constant, to be fixed later. Let $S_1=S_1(F_1,F_2),S_2=S_2(F_1,F_2)$ be the corresponding contributions to the sum in \eqref{eq:bigsum_k=2}:

$$S_1:=\sum_{G'}   \P_{M_1,M_2}\left( M_1F_1=0 \wedge M_1 M_2 F_2=0 \wedge \mbox{ $(F_1,M_2F_2)$ is a joint code of distance $\delta n$ over $G'$}\right)$$
and
$$S_2:=\sum_{G'}   \P_{M_1,M_2}\left( M_1F_1=0 \wedge M_1 M_2 F_2=0 \wedge  \mbox{ $(F_1,M_2F_2)$ spans $G'$ but not joint }\right),$$
so 
\begin{equation}
\label{eq:split_S1_S2}
\text{RHS\eqref{eq:bigsum_k=2}} = \sum_{\substack{F_1\in \Sur(V,G_1)\\ F_2 \in \Sur(V,G_2)}} S_1 + \sum_{\substack{F_1\in \Sur(V,G_1)\\ F_2 \in \Sur(V,G_2)}} S_2.
\end{equation}
We now analyze these two contributions to \eqref{eq:split_S1_S2}, showing that the contribution $S_1$ is the main term while that of $S_2$ is asymptotically small. Here and below, we use $\sum_{G'}$ as shorthand for a sum over all $G' \leq G_1 \op G_2$ with $\pi_1(G')=G_1$. 

{\bf a. Analysis of $S_1$.} Note that as $(F_1,M_2F_2)$ is joint code of distance $\delta n$ over $G'$, since $\pi_1(G')=G_1$ by the definition of $G'$, $F_1$ is a code of distance $\delta n$ over $G_1$. By Lemma \ref{lemma:code:joint} we have 
\begin{equation}\label{eq:S1_expansion}
\P_{M_1}\left( M_1F_1=0 \wedge M_1 M_2 F_2=0 \ | \mbox{ joint over $G'$}\right) =\frac{1}{|G'|^n}(1+O( \exp(- cn))),
\end{equation}
where $c$ depends on $a,\delta, \al$, and $G$ (\emph{a priori} $c$ depends on $G'$, but we simply take the worst constant $c$ over the finitely many choices of $G'$, which hence depends only on $G$). Hence 
$$S_1=  \sum_{G'} \frac{1}{|G'|^n}(1+ O(\exp(-cn)) \times   \P_{M_2}((F_1,M_2 F_2) \mbox{ is a joint code of distance $\delta n$ over } G').$$
It remains to evaluate the probability with respect to $M_2$. For this we will sum over $F_1$ as well. % (i.e. $\sum_{F_1} S_1(F_1,F_2)$ by rewriting in the form $|G_1|^n \P_{F_1}(.)$ where the probability is uniform).
We will divide into two cases.

{\it (i) \underline{Main term: summation over codes $F_1$ and codes $F_2$}.} We first consider the case when $F_2$ is a code of distance $\delta n$ in $G_2$. Claim \ref{claim:code:count} applied to the group $G'$ immediately yields the following. 

\begin{claim}\label{claim:code:count'} The number of joint $ F=((f_{1},g_1),\dots, (f_{n},g_n)) $ of distance $\delta n$ in ${G'}^n$  is $(1+O(\exp(-c n)) |G'|^n$, where the constants are allowed to depend on $\delta$.
\end{claim}

Now for each such joint code of distance $\delta n$ $F$, because $F_2$ is a code  of distance $\delta n$ in $G_2$, by Lemma \ref{lemma:code:single:matrix:1} 
$$\P_{M_2} ((F_1, M_2F_2)=F) = \mathbf{1}_{F_1= \pi_1(F)} \frac{1}{|G_2|^n}(1+O(\exp(-cn))).$$

Summing over $F_1$, noting that $\pi_1(F)$ is a code of distance $\delta n$ over $G_1$, we thus obtain 
\begin{equation}\label{k=2:F_1}
\sum_{F_1} \P_{M_2} ((F_1, M_2F_2)=F) = \frac{1}{|G_2|^n}(1+O(\exp(-cn))).
\end{equation}
We then sum over codes $F$  of distance $\delta n$ to obtain
$$\sum_{F_1}  \P_{M_2}((F_1,M_2 F_2) \mbox{ is a code of distance $\delta n$ over } G') = \frac{|G'|^n}{|G_2|^n} (1+O(\exp(-cn))).$$
Hence in total we have, for a fixed code of distance $\delta n$ $F_2$ over $G_2$ 
\begin{align*}
&\sum_{F_1} \P_{M_1}\left( M_1F_1=0 \wedge M_1 M_2 F_2=0 \ | \mbox{ joint code of distance $\delta n$ over $G'$}\right) \\
& \times  \P_{M_2}((F_1,M_2 F_2) \mbox{ joint code of distance $\delta n$ over $G'$}) \\ 
&=(1+ O( \exp(-c n))) \frac{1}{|G'|^n} \times  \frac{|G'|^n}{|G_2|^n} (1+ O(\exp(-cn))) \\
&= \frac{1}{|G_2|^n} (1+ O(\exp(-cn))).
\end{align*}
Now we sum over codes of distance $\delta n$ $F_2$ (using Claim \ref{claim:code:count}), and then over $G'$ to obtain that
\begin{align*}
& \sum_{\substack{F_1 \in \Sur(V,G_1) \\ F_2 \mbox{ code over $G_2$}}}\sum_{G'}   \P_{M_1,M_2}\left( M_1F_1=0 \wedge M_1 M_2 F_2=0 \wedge \mbox{ $(F_1,M_2F_2)$ spans and is joint over $G'$}\right) \\
&= (1+O(\exp(-cn))) m_2(G_1, G_2).
\end{align*}
Before moving to the next estimate, we record below some other useful results by summing \eqref{k=2:F_1} over codes $F_2$,
\begin{equation}\label{k=2:F_1F_2code-code=code}
\sum_{F_1, F_2 \mbox{ codes of distance $\delta n$}}  \P_{M_2}((F_1,M_2 F_2)=F) = 1+O(\exp(-cn))
\end{equation}
and 
\begin{equation}\label{k=2:F_1F_2code-code}
\sum_{F_1, F_2 \mbox{ codes of distance $\delta n$}}  \P_{M_2}((F_1,M_2 F_2) \mbox{ is a code of distance $\delta n$ over } G') = |G'|^n (1+O(\exp(-cn))).
\end{equation}

{\it (ii) \underline{Error term: summation over codes $F_1$ and non-codes $F_2$}.} Assume that $F_2$ is not a code of distance $\delta n$ over $G_2$. 
%\textcolor{red}{Hence for each $D_2\big{\vert}|G_2|$ we will restrict to a $\delta$-depth of size $D_2$ for some given $\sigma \subset [n]$ of size at most $\delta n$. For each $F$ code in $G'$, by Lemma \ref{lemma:non-code:1:matrix}}
Then it has some $\delta$-depth $D_2 > 1$ for some $D_2 \big{\vert} |G_2|$. We will shortly sum over all such $D_2$ and $F_2$, but for now note that for any fixed $D_2$ and $F_2$ of $\delta$-depth $D_2$, by Lemma \ref{lemma:non-code:1:matrix}

$$\P_{M_2} ((F_1, M_2F_2)=F) = \mathbf{1}_{F_1 =  \pi_1(F)} \P(M_2 F_2 = \pi_2(F)) \le \mathbf{1}_{F_1 =  \pi_1(F)}  (D_2/|G_2|)^n K\exp(-\al n).$$

Sum over codes $F$, over $F_1$, and then $F_2$ we obtain 
\begin{align}\label{k=2:F1F2nocode-code}
\begin{split}
& \sum_{\substack{F_1 \mbox{ code over $G_1$}\\ F_2 \in \Sur(V,G_2) \mbox{ non-code}}}\P_{M_2}((F_1,M_2 F_2) \mbox{ is code over } G') \\ 
& = \sum_{F \mbox{ code over $G'$}}  \sum_{F_1 \mbox{ code over $G_1$},  F_2 \in \Sur(V,G_2) \mbox{ non-code}}\P_{M_2}((F_1,M_2 F_2)=F)\\
& = \sum_{\substack{D_2>1\\ D_2 \big{\vert}|G_2|}} \sum_{F \mbox{ code over $G'$}}  \sum_{\substack{F_1 \mbox{ code over $G_1$}\\  F_2 \in \Sur(V,G_2) \mbox{ non-code of depth $D_2$}}}\P_{M_2}((F_1,M_2 F_2)=F)\\
& \le\sum_{F \mbox{ code over $G'$}}   \sum_{\substack{D_2>1\\ D_2 \big{\vert}|G_2|}} \ \sum_{F_1 \mbox{ code over $G_1$}} \mathbf{1}_{F_1=  \pi_1(F)}   K (D_2/|G_2|)^n\exp(-\al n) \times \\
& \times K' \binom{n}{\lceil \ell(D_2) \delta n\rceil -1} |G_2|^n D_2^{-n+\ell(D_2) \delta n} \\
&\leq  \sum_{D_2>1, D_2 \bver |G_2|} |G'|^n K (D_2/|G_2|)^n\exp(-\al n) \times  K' \binom{n}{\lceil \ell(D_2) \delta n\rceil -1} |G_2|^n D_2^{-n+\ell(D_2) \delta n} \\
& =O( |G'|^n \exp(-\al n/2)),
\end{split}
\end{align}
where we used Lemma \ref{lemma:non-code:count:1} to enumerate $F_2$, and in the last bound assumed that $n$ is large and $\delta$ is sufficiently small so that the exponential growth of the $\delta$-dependent factors is not too large.

Combining the above with \eqref{eq:S1_expansion} yields
$$\sum_{\substack{F_1 \mbox{code over $G_1$}\\ F_2 \in \Sur(V,G_2)  \mbox{ non-code}}} S_1(F_1,F_2) \le   \sum_{G'} \frac{1}{|G'|^n}(1+ O(\exp(-cn))) \times O( |G'|^n \exp(- \al n/2)) =O( \exp(-\al n/2)).$$

In summary, we have shown the following
\begin{align*}
&\sum_{F_1 \in \Sur(V,G_1),F_2 \in \Sur(V,G_2)}\sum_{G'}   \P_{M_1,M_2}\left( M_1F_1=0 \wedge M_1 M_2 F_2=0 \wedge \mbox{ $(F_1,M_2F_2)$ is surjection and joint over $G'$}\right)\\
&= (1+O(\exp(-cn))) m_2(G_1, G_2).
\end{align*}
For later use, we record here another useful result by putting \eqref{k=2:F_1F_2code-code} and  \eqref{k=2:F1F2nocode-code} together
\begin{equation}\label{k=2:F_1F_2-code}
\sum_{F_1, F_2}  \P_{M_2}((F_1,M_2 F_2) \mbox{ is a code of distance $\delta n$ over } G') = |G'|^n (1+O(\exp(-cn))).
\end{equation}

{\bf b. Analysis of $S_2$.} Our treatment for the case of $S_2$ is similar to (ii) of {\bf a.} We first partition $S_2$ into contributions corresponding to $(F_1,M_2 F_2)$ of given depth, 
\begin{equation}\label{eq:S2_partition}
S_2= \sum_{G'} \sum_{\substack{D>1 \\ D\big{\vert} |G'|}} \sum_{F' \mbox{ of depth $D$ over $G'$}}  \P_{M_1}\left( M_1F'=0  \wedge (F_1,M_2 F_2) = F' \right).
\end{equation}
Since the first two sums in \eqref{eq:S2_partition} are finite, for the contribution of $S_2$ it suffices to fix $G'$ and $D$, and show that 
$$\sum_{\substack{F_1 \in \Sur(V,G_1) \\  F_2 \in \Sur(V,G_2)}}   \sum_{F' \mbox{ of depth $D$ over $G'$}}  \P_{M_1}\left( M_1F'=0  \wedge (F_1,M_2 F_2) = F' \right)$$
is small. 

First, by using Lemma \ref{lemma:non-code:joint:matrix:2} we can bound 
\begin{align}\label{eq:main_b_bound}
\begin{split}
&\sum_{\substack{F_1 \in \Sur(V,G_1) \\  F_2 \in \Sur(V,G_2)}} \sum_{F' \mbox{ of depth $D$ over $G'$}}  \P_{M_1}\left( M_1F'=0  \wedge (F_1,M_2 F_2) = F' \right) \\
& \le \sum_{\substack{F_1 \in \Sur(V,G_1) \\  F_2 \in \Sur(V,G_2)}}\sum_{F' \mbox{ of depth $D$ over $G'$}} K \exp(-\al n)\left(\frac{D}{|G'|}\right)^n \P((F_1,M_2 F_2) = F').
\end{split}
\end{align}
{\it (i) \underline{Summation over $F_1 \in \Sur(V,G_1)$ and codes of distance $\delta n$ $F_2$ over $G_2$}.} We first fix code of distance $\delta n$ $F_2$, and $F'$ of depth $D$ over $G'$. By Lemma \ref{lemma:code:single:matrix:1},
$$ \sum_{F_1 \in \Sur(V,G_1)} \P_{M_2}((F_1,M_2 F_2) = F') = (1/|G_2|^n)(1+O(\exp(-cn))).$$
Summing over codes of distance $\delta n$ $F_2$ over $G_2$ using Claim \ref{claim:code:count},
$$ \sum_{F_2 \mbox{ $\delta n$ code over $G_2$}} \sum_{F_1} \P((F_1,M_2 F_2) = F') = |G_2|^n (1+O(\exp(-c'n))) (1/|G_2|^n)(1+O(\exp(-cn)))=1+O(\exp(-c''n)),$$
where $c''=\min \{c,c'\}$.

As a consequence, by summing over the non-codes $F'$ of depth $D$ over $G'$ we obtain
\begin{align}\label{k=2:code-noncode}
\begin{split}
\sum_{\substack{F_1 \in \Sur(V,G_1) \\ F_2 \mbox{ $\delta n$ code over $G_2$}}}   &\P_{M_2}((F_1,M_2 F_2) \mbox{ is of depth $D$ over } G'  ) \\
&\leq \sum_{\substack{F_1 \in \Sur(V,G_1) \\ F_2 \mbox{ $\delta n$ code over $G_2$}}} \sum_{F'  \mbox{ of depth $D$ over $G'$}} \P((F_1,M_2 F_2) = F') \\
& \le K \binom{n}{\lceil \ell(D) \delta n\rceil -1} |G'|^n D^{-n+\ell(D) \delta n}.
\end{split}
\end{align}

More importantly, by summing over the non-codes $F'$ of depth $D$ over $G'$, by Lemma \ref{lemma:depth:count:2} and \eqref{eq:main_b_bound} we obtain
\begin{align*}
\sum_{\substack{F_1 \in \Sur(V,G_1) \\ F_2 \mbox{ $\delta n$ code over $G_2$}}} S_2  &\leq \sum_{G'} \sum_{\substack{D>1 \\ D\big{\vert} |G'|}} \sum_{\substack{F_1 \in \Sur(V,G_1) \\ F_2 \mbox{ $\delta n$ code over $G_2$}}} \sum_{F'  \mbox{ of depth $D$ over $G'$}} \exp(-\al n) \left(\frac{D}{|G'|}\right)^n \P((F_1,M_2 F_2) = F')\\
& \le  \sum_{G'}   \sum_{\substack{D>1 \\ D\big{\vert} |G'|}} \exp(-\al n) \left(\frac{D}{|G'|}\right)^n K \binom{n}{\lceil \ell(D) \delta n\rceil -1} |G'|^n D^{-n+\ell(D) \delta n} (1+O(\exp(-c''n))) \\
& =O( \exp(-\al n/2)),
\end{align*}
where again in the last bound we require that $\delta$ is suffiently small and $n$ is sufficiently large.

{\it (ii) \underline{Summation over $F_1 \in \Sur(V,G_1)$ and non-codes $F_2\in \Sur(V,G_2)$}.} The treatment here is also similar to  {\bf a} (ii). Indeed, as $F_2$ is not a code of distance $\delta n$ over $G_2$, it has some $\delta$-depth $D_2 > 1$ for some $D_2 \big{\vert} |G_2|$. By Lemma \ref{lemma:non-code:1:matrix}
$$\P_{M_2} ((F_1, M_2F_2)=F') = \mathbf{1}_{F_1 =  \pi_1(F')} \P(M_2 F_2 = \pi_2(F')) \le \mathbf{1}_{F_1 =  \pi_1(F')}  (D_2/|G_2|)^n K\exp(-\al n).$$

As a consequence, summing over $F'$ of depth $D$, over $F_1$, and then $F_2$ (by summing over $D_2$) we obtain 
\begin{align}\label{k=2:F1F2nocode-code'}
\begin{split}
& \sum_{\substack{F_1 \in \Sur(V,G_1)\\ F_2 \in \Sur(V,G_2) \mbox{ non-code}}}\P_{M_2}((F_1,M_2 F_2) \mbox{ is of depth $D$ over } G') \\ 
& = \sum_{F' \mbox{ of depth $D$ over $G'$}}  \sum_{F_1 \in \Sur(V,G_1)}\sum_{\substack{D_2>1 \\ D_2\big{\vert}|G_2|}} \mathbf{1}_{F_1 =  \pi_1(F')}  (D_2/|G_2|)^n K\exp(-\al n)  \\
&\times K' \binom{n}{\lceil \ell(D_2) \delta n\rceil -1} |G_2|^n D_2^{-n+\ell(D_2) \delta n}   \\
&\le  \sum_{F' \mbox{ of depth $D$ over $G'$}} \exp(-cn) \\
&\le \exp(-cn) \binom{n}{\lceil \ell(D) \delta n\rceil -1} |G'|^n D^{-n+\ell(D) \delta n} .
\end{split}
\end{align}
Also,
\begin{align*}
\begin{split}
& \sum_{\substack{F_1 \in \Sur(V,G_1) \\ F_2 \mbox{ non-code over $G_2$}}} S_2\\
& \le \sum_{G'}   \sum_{\substack{D>1 \\ D\big{\vert} |G'|}}  \sum_{\substack{F_1 \in \Sur(V,G_1)\\ F_2 \in \Sur(V,G_2) \mbox{ non-code}}}  \sum_{F' \mbox{ of depth $D$ over $G'$}}  \P_{M_1}\left( M_1F'=0  \wedge (F_1,M_2 F_2) = F' \right) \\
& \le   \sum_{G'}   \sum_{\substack{D>1 \\ D\big{\vert} |G'|}} \sum_{\substack{D_2>1 \\ D_2\big{\vert}|G_2|}} \sum_{F_1 \in \Sur(V,G_1)} \exp(-\al n)\left(\frac{D}{|G'|}\right)^n K \binom{n}{\lceil \ell(D) \delta n\rceil -1} |G'|^n D^{-n+\ell(D) \delta n}  \\
& \times \ \mathbf{1}_{F_1 =  \pi_1(F')}  (D_2/|G_2|)^n K\exp(-\al n) \times K' \binom{n}{\lceil \ell(D_2) \delta n\rceil -1} |G_2|^n D_2^{-n+\ell(D_2) \delta n} \\
&=O( \exp(-\al n)),
\end{split}
\end{align*}
completing the treatment in this case.

Finally, we remark from \eqref{k=2:code-noncode} and \eqref{k=2:F1F2nocode-code'}  
that
\begin{equation}\label{k=2:F_1F_2-noncode}
\sum_{F_1, F_2}  \P_{M_2}((F_1,M_2 F_2) \mbox{ is of depth $D>1$ over } G') =O\left(  \binom{n}{\lceil \ell(D) \delta n\rceil -1} |G'|^n D^{-n+\ell(D) \delta n}\right).
\end{equation}

%\HNC{There are many constants in the proof below appearing in the form of $\exp(-cn)$ that I did not have time to take care of. Perhaps we should change to $c',c''$ etc; but if we don't want to keep track those closely, we can write as $\exp(-\Theta(n))$. However, these need to absorb the bound $\exp(\delta \log(1/\delta) n)$ or so in many bounds; so we would need to say that the implied constants in the $\Theta(.)$ is independent of $\delta$.}

\subsection{Proof of Theorem \ref{theorem:jointsur:k} for general $k$}  
 We will proceed as in the case $k=2$. We have
$$\E \left(\# \Sur(R^n/M_1R^n, G_1)  \cdots \# \Sur(R^n/M_1\cdots M_k R^n,G_k)\right)$$ 
$$= \sum_{F_1\in \Sur(V,G_1), \dots, F_k \in \Sur(V,G_k)}\P(M_1 F_1=0 \wedge \dots \wedge M_1\cdots M_kF_k=0).$$

For each subgroup $G' \in \CG_{1;2,\dots,k}$, we consider the case that $(F_1,M_2F_2, \dots, M_2 \cdots M_k F_k)$ span $G'$ and 

\begin{enumerate}[(1)]
\item $F_1, M_2F_2,\dots, M_2\cdots M_k F_k$ form a joint code of distance $\delta n$ with respect to $G'$;
\vskip .1in
\item  They are not joint, hence have some $\delta$-depth $D>1$.
%Which can be the case if (i) $F_1$ is a surjection but non-code in $G_1$, or (ii) $M_2\dots M_i F_i $ is non-code in $H_i$ for some $i$, or (iii) more generally they are together has $D$-depth ($D>1$) with respect to $G'$. 
\end{enumerate}
Motivated by \eqref{k=2:F_1F_2-code} and \eqref{k=2:F_1F_2-noncode} in the proof of the $k=2$ case of Theorem \ref{theorem:jointsur:k}, we will show the following key result.

\begin{proposition}\label{prop:joint:k} With the same assumption as in Theorem \ref{theorem:main:prod} and $\delta$ sufficiently small, %: there exist $c,K$ depending on $k,\al, \delta, G_1,\dots, G_k$ such that 
for each subgroup $G' \in \CG_{1;2,\dots,k}$ the following inequalities hold with the randomness from $M_2,\dots, M_k$:
%\begin{enumerate}[(a)]
%\item  (Main term from codes) 
\begin{align}\label{eq:joint:k}
\begin{split}
& \sum_{\substack{F_i\in \Sur(V,G_i) \\ 1 \leq i \leq k}} \P_{M_2,\dots, M_k}((F_1,M_2 F_2, \dots, M_2 \cdots M_k F_k) \mbox{ spans and is a code of distance $\delta n$ in } G')\\
&=(1+O(\exp(-cn)))m_{k-1}(G')|G'|^n
\end{split}
\end{align}
and
\begin{align}\label{eq:joint:k'}
\begin{split}
& \sum_{\substack{F_i\in \Sur(V,G_i) \\ 1 \leq i \leq k}} \P_{M_2,\dots, M_k}((F_1,M_2 F_2, \dots, M_2 \cdots M_k F_k) \mbox{ spans and has depth $D$ in } G')\\
& =O\left(\binom{n}{\lceil \ell(D) \delta n\rceil -1} D^{-n+\ell(D) \delta n} m_{k-1}(G')  |G'|^n \right),
\end{split}
\end{align}
where $c$ and the implied constants are allowed to depend on  $k,\al, \delta, G_1,\dots, G_k$ and $a$. 
\end{proposition}

We return to the proof of this result after proving Theorem \ref{theorem:jointsur:k}.
\begin{proof}[Proof of Theorem \ref{theorem:jointsur:k}, assuming Proposition \ref{prop:joint:k}] 
By the above discussion, we must show 
\begin{equation}\label{eq:jointsurk_wts}
\sum_{F_1\in \Sur(V,G_1), \dots, F_k \in \Sur(V,G_k)}\P(M_1 F_1=0 \wedge \dots \wedge M_1\cdots M_kF_k=0) = m_k(G_1,\ldots,G_k)(1+O(\exp(-cn))).
\end{equation}
For any $G'$, we have by Lemma \ref{lemma:code:joint} (where ``code" is understood as ``code of distance $\delta n$") that
\begin{align*}
& \P_{M_1}\left( M_1F_1=0 \wedge \dots \wedge  M_1 \cdots M_k F_k=0 | F_1, \dots, M_2\cdots M_k F_k \mbox{ joint over $G'$}\right)=  \frac{1}{|G'|^n}(1 + O(\exp(-cn))).
\end{align*}
Hence
\begin{align*}
& \sum_{\substack{F_i\in \Sur(V,G_i) \\ 1 \leq i \leq k}}\P_{M_1,\dots, M_k}\left( M_1F_1=0 \wedge \dots \wedge  M_1 \cdots M_k F_k=0 \wedge (F_1, \dots, M_2\cdots M_k F_k) \mbox{ joint over $G'$}\right)\\
&= \sum_{\substack{F_i\in \Sur(V,G_i) \\ 1 \leq i \leq k}}\P_{M_1}\left( M_1F_1=0 \wedge \dots \wedge  M_1 \cdots M_k F_k=0 \Big | (F_1, \dots, M_2\cdots M_k F_k) \mbox{ joint over $G'$}\right) \\
&\times \P_{M_2,\dots, M_k}\left(F_1, \dots, M_2\cdots M_k F_k \mbox{ joint  over $G'$}\right)\\
&= \sum_{\substack{F_i\in \Sur(V,G_i) \\ 1 \leq i \leq k}} \frac{1}{|G'|^n}(1 + O(\exp(-cn)))  \P_{M_2,\dots, M_k}\left(F_1, \dots, M_2\cdots M_k F_k \mbox{ joint over $G'$}\right)\\
& = (1+O(\exp(-cn)))m_{k-1}(G')
\end{align*}
where the last equality follows from the first part of Proposition \ref{prop:joint:k}.

Summing over $G'$ and using \eqref{eqn:kk-1}
\begin{align*}
& \sum_{G' \in \CG_{1;2,\dots,k}}\sum_{\substack{F_i\in \Sur(V,G_i) \\ 1 \leq i \leq k}}\P_{M_1,\dots, M_k}\left( M_1F_1=0 \wedge \dots \wedge  M_1 \cdots M_k F_k=0 \wedge F_1, \dots, M_2\cdots M_k F_k \mbox{ joint over $G'$}\right)\\
& = (1+O(\exp(-cn)))m_{k}(G_1,\dots, G_k).
\end{align*}

The sum over non-codes can be treated as follows:

\begin{align*}
& \sum_{\substack{F_i\in \Sur(V,G_i) \\ 1 \leq i \leq k}} \P \left( M_1F_1=0 \wedge \dots \wedge  M_1 \cdots M_k F_k=0  \wedge F_1,M_2 F_2, \dots, M_2 \cdots M_k F_k \mbox{ span but are not code in } G' \right)\\
& =\sum_{D>1, D||G'|}\sum_{\substack{F_i\in \Sur(V,G_i) \\ 1 \leq i \leq k}} \P_{M_2,\dots, M_k} \left( F_1,M_2 F_2, \dots, M_2 \cdots M_k F_k \mbox{ span and of depth $D$ in } G'\right)  \\
&\times \P_{M_1} \left( M_1F_1=0 \wedge \dots \wedge  M_1 \cdots M_k F_k=0 \Big | F_1,M_2 F_2, \dots, M_2 \cdots M_k F_k \mbox{ span and of depth $D$ in } G' \right)
\\
& \le\sum_{D>1, D||G'|} \sum_{\substack{F_i\in \Sur(V,G_i) \\ 1 \leq i \leq k}} \P_{M_2,\dots, M_k} \left( F_1,M_2 F_2, \dots, M_2 \cdots M_k F_k \mbox{ span and of depth $D$ in } G' \right)\\
& \times K \exp(-\al n)  \left(\frac{D}{|G|}\right)^n \\
& \le \sum_{D>1, D||G'|} K \exp(-\al n)  \left(\frac{D}{|G'|}\right)^n \times K' \binom{n}{\lceil \ell(D) \delta n\rceil -1} D^{-n+\ell(D) \delta n} m_{k-1}(G')  |G'|^n  \\
&=O(\exp(-\al n))  m_{k-1}(G'),
\end{align*}
where we used Lemma \ref{lemma:non-code:joint:matrix:k} in the first bound, the second part of Proposition \ref{prop:joint:k} in the second bound, and assumed that $\delta$ is sufficiently small for the final step. 
\end{proof}

\begin{proof}[Proof of Proposition \ref{prop:joint:k}] Similarly to the proof of Proposition \ref{prop:single:k} we will induct on $k$. The base case $k=2$ are 
\eqref{k=2:F_1F_2-code} and \eqref{k=2:F_1F_2-noncode} in the previous proof of the $k=2$ case of Theorem \ref{theorem:jointsur:k} in Subsection \ref{sub:k=2}. We induct on both \eqref{eq:joint:k} and \eqref{eq:joint:k'} simultaneously, i.e. we need the $k-1$ case of both \eqref{eq:joint:k} and \eqref{eq:joint:k'} to prove the $k$ case of each.

{\bf \underline{Proof of \eqref{eq:joint:k}}.} As we are working with joint, since $\pi_1(G')=G_1$ by the definition of $G'$, it follows automatically that $F_1$ is a code of distance $\delta n$ in $G_1$. Recall that $\CC(G')$ denotes the set of codes in $(G')^n$. By Claim \ref{claim:code:count'} we know that this set has size $(1+O(\exp(-cn))) |G'|^n$. 

For each code $F'$ of $G'$ we are interested in $\P_{M_2,\dots, M_k}((F_1,M_2 F_2,\dots, M_2 \cdots M_k F_k)=F')$, which can be written as $\mathbf{1}_{F_1 = \pi_1(F')} \P_{M_2,\dots, M_k}((M_2 F_2,\dots, M_2 \cdots M_k F_k)=\pi_{\{2,\dots, k\}}(F'))$ (here the notation of $\pi_{\{2,\dots, k\}}(F')$ is as in Remark \ref{rmk:proj_F}).  Note that the entries of the vector $\pi_{\{2,\dots, k\}}(F')$ generate the group $\pi_{\{2,\dots, k\}}(G')$. This probability depends on whether $(F_2, M_3F_3,\dots, M_3 \dots M_k F_k)$ forms a code or not. Hence we will fix a code $F'$ of $G'$, and divide into two cases.

{\it (i) \underline{Summing over codes}.} 
Since $(F_1,M_2F_2,\ldots,M_2 \cdots M_k F_k)$ generates $G'$ by assumption, projecting to $G_2 \oplus \cdots \oplus G_k$ we have that $M_2(F_2,M_3 F_3,\dots, M_3 \cdots M_k F_k)$ generates $\pi_{\{2,\dots, k\}}(G')$. Hence $(F_2,M_3 F_3,\dots, M_3 \cdots M_k F_k)$ will always generate some group $G''$ which contains $\pi_{\{2,\dots, k\}}(G')$, and furthermore $G''$ must lie in $\CG_{2;3,\dots, k}$ because $F_2$ is surjective. 

This motivates us to fix a group $G''  \in \CG_{2;3,\dots, k}$ with $G'' \geq \pi_{\{2,\dots, k\}}(G')$, and consider the probability that $(F_2,M_3 F_3,\dots, M_3 \cdots M_k F_k)$ span and is a code of distance $\delta n$ in $G''$.
By the inductive hypothesis for \eqref{eq:joint:k},
\begin{align}\label{eqn:k:ind}
\begin{split}
& \sum_{F_i\in \Sur(V,G_i),2\le i\le k}\P_{M_3,\dots, M_k}((F_2,M_3 F_3,\dots, M_3 \cdots M_k F_k) \mbox{ span and is a code of distance $\delta n$ in } G'') \\
&=(1+O(\exp(-cn)))m_{k-2}(G'')|G''|^n.
\end{split}
\end{align}
%(Note that here we have to consider $\pi_{\{2,\dots, k\}}(G') \le G''$ because as $(F_2,M_3 F_3,\dots, M_3 \cdots M_k F_k)\in (G'')^n$, we also have $M_2(F_2,M_3 F_3,\dots, M_3 \cdots M_k F_k) \in  (G'')^n$, but by assumption $M_2(F_2,M_3 F_3,\dots, M_3 \cdots M_k F_k)$ generates $\pi_{\{2,\dots, k\}}(G')$. Also, $G'' \in \CG_{2;3,\dots, k}$ because it contain $F_2$ in the projection.) 

For short, let $\CE_{3,\dots,k, G''}$ be the event (depending on $F_i$)
$$\CE_{3,\dots,k,G''}:=\Big\{M_3,\ldots,M_k: (F_2,M_3 F_3,\dots, M_3 \cdots M_k F_k) \mbox{ span and is a code of distance $\delta n$ in } G''\Big\}.$$
By Lemma \ref{lemma:code:joint},
$$\mathbf{1}_{F_1 = \pi_1(F')}  \P_{M_2}(M_2 F_2,\dots, M_2 \cdots M_k F_k = \pi_{\{2,\dots, k\}}(F')|\CE_{3,\dots,k,G''})=\mathbf{1}_{F_1 = \pi_1(F')}  (1/|G''|^n) (1+O(\exp(-cn))).$$

% Summing over $F_i$ we will obtain $(1+o(1)) m_{k-2}(G'')$, and summing over $G''$ satisfying $\pi_{\{2,\dots, k\}}(G') \le G'' \in \CG_{2;3,\dots, k}$ we obtain $(1+o(1))m_{k-1}(G')$. 
We therefore have 
\begin{align*}
& \ \ \ \ \sum_{\substack{G'' \in \CG_{2;3,\dots, k} \\ G'' \geq \pi_{\{2,\dots, k\}}(G')}}  \sum_{\substack{F_i\in \Sur(V,G_i) \\ 1 \leq i \leq k}}\P_{M_2,\dots, M_k}((F_1,M_2 F_2,\dots, M_2 \cdots M_k F_k)=F' \wedge \CE_{3,\dots,k,G''})\\
&=\sum_{G''} \sum_{F_1\in \Sur(V,G_1)}\mathbf{1}_{F_1 = \pi_1(F')} \sum_{F_i\in \Sur(V,G_i),2\le i\le k}  \P_{M_2}((M_2 F_2,\dots, M_2 \cdots M_k F_k)=\pi_{\{2,\dots, k\}}(F')|\CE_{3,\dots,k,G''}) \times \\
& \times \P_{M_3,\dots,M_k}(\CE_{3,\dots,k,G''})\\
&= \sum_{G''}  \sum_{F_1\in \Sur(V,G_1)}\mathbf{1}_{F_1 = \pi_1(F')} \sum_{F_i\in \Sur(V,G_i),2\le i\le k} (1/|G''|^n)(1+O(\exp(-cn)))\times \P_{M_3,\dots,M_k}(\CE_{3,\dots,k,G''})\\
&= \sum_{G''} \sum_{F_i\in \Sur(V,G_i),2\le i\le k}  (1/|G''|^n)(1+O(\exp(-cn))) \times \P_{M_3,\dots,M_k}(\CE_{3,\dots,k,G''})\\
&= \sum_{G''}(1+O(\exp(-cn))) (1/|G''|^n) |G''|^n(1+O(\exp(-cn))) m_{k-2}(G'')= (1+O(\exp(-cn)))m_{k-1}(G'),
\end{align*}
where we used the induction hypothesis \eqref{eqn:k:ind} in the very last estimate, and the fact that 
$$\sum_{\substack{G'' \in \CG_{2;3,\dots, k} \\ G'' \geq \pi_{\{2,\dots, k\}}(G')}} m_{k-2}(G'') = m_{k-1}(G').$$

By Claim \ref{claim:code:count} there are $(1+O(\exp(-cn))) |G'|^n$ codes $F'$. Summing the above over them, we obtain
\begin{align}\label{eq:sum_code_induction}
\begin{split}
&\sum_{\substack{F_i\in \Sur(V,G_i) \\ 1 \leq i \leq k}} \P_{M_2,\dots, M_k}(F_1,M_2 F_2, \dots, M_2 \cdots M_k F_k \mbox{ span and is a code of distance $\delta n$ in } G' \\
 &\wedge (F_2,M_3 F_3,\dots, M_3 \cdots M_k F_k) \text{ is a code of distance $\delta n$ in its image}) \\
&=(1+O(\exp(-cn)))m_{k-1}(G')|G'|^n.
 \end{split}
\end{align}

It remains to show that the remaining part of LHS\eqref{eq:joint:k}, corresponding to the case when $(F_2,M_3 F_3,\dots, M_3 \cdots M_k F_k)$ is not a code in its image, is small.

{\it (ii) \underline{Summing over non-codes}.} Let $G''  \in \CG_{2;3,\dots, k}$ again be a group which contains $\pi_{\{2,\dots, k\}}(G')$. We are considering the case that $(F_2,M_3 F_3, \dots, M_3 \cdots M_k F_k)$ spans but is not a code of distance $\delta n$ over $G''$, hence it has depth $D>1$ for some $D||G''|$. Let $\CF_{3,\dots, k,G'',D}$ denote this event. We have
\begin{align*}
&\sum_{F_2,\dots, F_k}\P_{M_2, M_3,\dots, M_k}((M_2 F_2,\dots, M_2 \cdots M_k F_k)=\pi_{\{2,\dots, k\}}(F') \wedge \CF_{3,\dots, k,G'',D})\\
&= \sum_{F_2,\dots, F_k}\P_{M_2}((M_2 F_2,\dots, M_2 \cdots M_k F_k)=\pi_{\{2,\dots, k\}}(F')|\CF_{3,\dots, k,G'',D} )\P_{M_3,\dots, M_k}(\CF_{3,\dots, k,G'',D}).
\end{align*}
By Lemma \ref{lemma:non-code:joint:matrix:k}\begin{align*}
\P_{M_2}((M_2 F_2,\dots, M_2 \cdots M_k F_k)=\pi_{\{2,\dots, k\}}(F')|\CF_{3,\dots, k,G'',D} ) \le K \exp(-\al n) (D/|G''|)^n.
\end{align*}

Next, using the induction hypothesis \eqref{eq:joint:k'} for $k-1$ we have
$$\sum_{F_2,\dots, F_k}\P_{M_3,\dots, M_k } (\CF_{3,\dots, k,G'',D}) \leq K' \binom{n}{\lceil \ell(D) \delta n\rceil -1} D^{-n+\ell(D) \delta n} m_{k-2}(G'')  |G''|^n.$$
We thus obtain that 
\begin{align*}
&\sum_{F_2,\dots, F_k}\P_{M_2; M_3,\dots, M_k}((M_2 F_2,\dots, M_2 \cdots M_k F_k)=\pi_{\{2,\dots, k\}}(F') \wedge \CF_{3,\dots, k,G'',D})=O(\exp(-\al n/2)),
\end{align*}
provided that $\delta$ is sufficiently small.

Summing over $D$, over $G''$ and over codes $F'$ of $G'$ we thus obtain
\begin{align}\label{eq:sum_code_induction'}
\begin{split}
&\sum_{\substack{F_i\in \Sur(V,G_i) \\ 1 \leq i \leq k}} \P_{M_2,\dots, M_k}(F_1,M_2 F_2, \dots, M_2 \cdots M_k F_k \mbox{ span and is a code of distance $\delta n$ in } G' \\
 &\wedge (F_2,M_3 F_3,\dots, M_3 \cdots M_k F_k) \text{ is not a code in its image}) \\
&= O(\exp(-\al n/2))|G'|^n.
 \end{split}
\end{align}
 Combining \eqref{eq:sum_code_induction'} with \eqref{eq:sum_code_induction} yields \eqref{eq:joint:k}.

% Next we focus on \eqref{eq:joint:k'}.

{\bf \underline{Proof of \eqref{eq:joint:k'}}.} For each spanning $F'$ of depth $D$ in $G'$ we are interested in $\P_{M_2,\dots, M_k}((F_1,M_2 F_2,\dots, M_2 \cdots M_k F_k)=F')$, which can be written as $\mathbf{1}_{F_1 = \pi_1(F')} \P_{M_2,\dots, M_k}((M_2 F_2,\dots, M_2 \cdots M_k F_k)=\pi_{\{2,\dots, k\}}(F'))$. This probability again depends on whether $(F_2, M_3F_3,\dots, M_3 \dots M_k F_k)$ forms a code or not. Hence we will divide into two cases.

{\it (i) \underline{Summing over codes}.} 
 
Our treatment is similar to the case (i) in the proof of \eqref{eq:joint:k'}, except that the sum over $F'$ is small as the number of non-codes is small. 

First, since $(F_1,M_2F_2,\ldots,M_2 \cdots M_k F_k)$ generates $G'$ by assumption, projecting to $G_2 \oplus \cdots \oplus G_k$ we have that $M_2(F_2,M_3 F_3,\dots, M_3 \cdots M_k F_k)$ generates $\pi_{\{2,\dots, k\}}(G')$. Hence $(F_2,M_3 F_3,\dots, M_3 \cdots M_k F_k)$ will always generate some group $G''$ which contains $\pi_{\{2,\dots, k\}}(G')$, and {}furthermore $G''$ must lie in $\CG_{2;3,\dots, k}$ because $F_2$ is surjective. 

This motivates us to fix a group $G''  \in \CG_{2;3,\dots, k}$ with $G'' \geq \pi_{\{2,\dots, k\}}(G')$, and consider the probability that $(F_2,M_3 F_3,\dots, M_3 \cdots M_k F_k)$ span and is a code of distance $\delta n$ in $G''$.
By the inductive hypothesis, 
\begin{align}%\label{eqn:k:ind}
\begin{split}
& \sum_{F_i\in \Sur(V,G_i),2\le i\le k}\P_{M_3,\dots, M_k}((F_2,M_3 F_3,\dots, M_3 \cdots M_k F_k) \mbox{ span and is a code of distance $\delta n$ in } G'') \\
&=(1+O(\exp(-cn)))m_{k-2}(G'')|G''|^n.
\end{split}
\end{align}

For short, let $\CE_{3,\dots,k, G''}$ be the event (depending on $F_i$)
$$\CE_{3,\dots,k,G''}:=\Big\{M_3,\ldots,M_k: (F_2,M_3 F_3,\dots, M_3 \cdots M_k F_k) \mbox{ span and is a code of distance $\delta n$ in } G''\Big\}.$$
By Lemma \ref{lemma:code:joint},
$$\mathbf{1}_{F_1 = \pi_1(F')}  \P_{M_2}(M_2 F_2,\dots, M_2 \cdots M_k F_k = \pi_{\{2,\dots, k\}}(F')|\CE_{3,\dots,k,G''})=\mathbf{1}_{F_1 = \pi_1(F')}  (1/|G''|^n) (1+O(\exp(-cn))).$$

We therefore have 
\begin{align*}
& \ \ \ \ \sum_{\substack{G'' \in \CG_{2;3,\dots, k} \\ G'' \geq \pi_{\{2,\dots, k\}}(G')}}  \sum_{\substack{F_i\in \Sur(V,G_i) \\ 1 \leq i \leq k}}\P_{M_2,\dots, M_k}((F_1,M_2 F_2,\dots, M_2 \cdots M_k F_k)=F' \wedge \CE_{3,\dots,k,G''})\\
&=\sum_{G''} \sum_{F_1\in \Sur(V,G_1)}\mathbf{1}_{F_1 = \pi_1(F')} \sum_{F_i\in \Sur(V,G_i),2\le i\le k}  \P_{M_2}((M_2 F_2,\dots, M_2 \cdots M_k F_k)=\pi_{\{2,\dots, k\}}(F')|\CE_{3,\dots,k,G''}) \times \\
& \times \P_{M_3,\dots,M_k}(\CE_{3,\dots,k,G''})\\
&= \sum_{G''}  \sum_{F_1\in \Sur(V,G_1)}\mathbf{1}_{F_1 = \pi_1(F')} \sum_{F_i\in \Sur(V,G_i),2\le i\le k} (1/|G''|^n)(1+O(\exp(-cn)))\times \P_{M_3,\dots,M_k}(\CE_{3,\dots,k,G''})\\
&= \sum_{G''} \sum_{F_i\in \Sur(V,G_i),2\le i\le k}  (1/|G''|^n)(1+O(\exp(-cn))) \times \P_{M_3,\dots,M_k}(\CE_{3,\dots,k,G''})\\
&= \sum_{G''}(1+O(\exp(-cn))) (1/|G''|^n) |G''|^n(1+O(\exp(-cn))) m_{k-2}(G'')= (1+O(\exp(-cn)))m_{k-1}(G'),
\end{align*}
where we used the induction hypothesis \eqref{eqn:k:ind} in the very last estimate. 
%and again the fact that  $$\sum_{\substack{G'' \in \CG_{2;3,\dots, k} \\ G'' \geq \pi_{\{2,\dots, k\}}(G')}} m_{k-2}(G'') = m_{k-1}(G').$$

As there are at most $O(\binom{n}{\lceil \ell(D) \delta n\rceil -1} D^{-n+\ell(D) \delta n}  |G'|^n) $ vectors $F'$ of depth $D$, summing the above over them we obtain 

\begin{align} \label{eq:sum_code_induction''}
\begin{split}
&\sum_{\substack{F_i\in \Sur(V,G_i) \\ 1 \leq i \leq k}} \P_{M_2,\dots, M_k}(F_1,M_2 F_2, \dots, M_2 \cdots M_k F_k \mbox{ span and of depth $D>1$ over } G') \\
 &\wedge (F_2,M_3 F_3,\dots, M_3 \cdots M_k F_k) \text{ is a code of distance $\delta n$ in its image}) \\
& \le O\Big(\binom{n}{\lceil \ell(D) \delta n\rceil -1} D^{-n+\ell(D) \delta n}\Big)m_{k-1}(G')|G'|^n(1+O(\exp(-cn))).
 \end{split}
\end{align}

It remains to show the case when $(F_2,M_3 F_3,\dots, M_3 \cdots M_k F_k)$ is not a code in its image.

{\it (ii) \underline{Summing over non-codes}.} Let $G''  \in \CG_{2;3,\dots, k}$ be a group which contains $\pi_{\{2,\dots, k\}}(G')$. We next consider the case that $(F_2,M_3 F_3, \dots, M_3 \cdots M_k F_k)$ spans but is not a joint code of distance $\delta n$ over $G''$. If $(F_2,M_3 F_3, \dots, M_3 \cdots M_k F_k)$ spans but is not joint code of distance $\delta n$, then it has depth $D_2$ for some $D_2>1$. Let $\CF_{3,\dots, k,G'',D_2}$ denote this event.  Hence  
\begin{align*}
&\sum_{F_2,\dots, F_k}\P_{M_2; M_3,\dots, M_k}((M_2 F_2,\dots, M_2 \cdots M_k F_k)=\pi_{\{2,\dots, k\}}(F') \wedge \CF_{3,\dots, k,G'',D_2})\\
&\le \sum_{F_2,\dots, F_k}\P_{M_2}((M_2 F_2,\dots, M_2 \cdots M_k F_k)=\pi_{\{2,\dots, k\}}(F')|\CF_{3,\dots, k,G'',D_2} )\\
&\times \P_{M_3,\dots, M_k}(\CF_{3,\dots, k,G'',D_2}).
\end{align*}

By Lemma \ref{lemma:non-code:joint:matrix:k}
\begin{align*}
\P_{M_2}((M_2 F_2,\dots, M_2 \cdots M_k F_k)=\pi_{\{2,\dots, k\}}(F')|\CF_{3,\dots, k,G'',D_2} )= \exp(-\al n) (D_2/|G''|)^n.
\end{align*}

Next, by \eqref{eq:joint:k'} (using the induction hypothesis for $k-1$) we have
$$\sum_{F_2,\dots, F_k}\P_{M_3,\dots, M_k } (\CF_{3,\dots, k,G'',D_2}) =O\left(  \binom{n}{\lceil \ell(D_2) \delta n\rceil -1} D_2^{-n+\ell(D_2) \delta n}\right)m_{k-1}(G'')|G''|^n.$$
Hence
$$\sum_{F_2,\dots, F_k}\P_{M_2; M_3,\dots, M_k}((M_2 F_2,\dots, M_2 \cdots M_k F_k)=\pi_{\{2,\dots, k\}}(F') \wedge \CF_{3,\dots, k,G'',D_2}) = O(\exp(-\al n/2)),$$
provided that $\delta$ is sufficiently small.

Summing over $D_2$ and over $G''$ we obtain 
\begin{align*}
\sum_{F_2,\dots, F_k}\P_{M_2; M_3,\dots, M_k}((M_2 F_2,\dots, M_2 \cdots M_k F_k)=\pi_{\{2,\dots, k\}}(F'))& =O( \exp(-\al n/2)),
\end{align*}
provided that $\delta$ is sufficiently small. 

Finally, again as there are at most $O(\binom{n}{\lceil \ell(D) \delta n\rceil -1} D^{-n+\ell(D) \delta n}  |G'|^n) $ vectors $F'$ of depth $D$, summing  over $F'$ we obtain
\begin{align*}%\label{eq:sum_code_induction''}
\begin{split}
&\sum_{\substack{F_i\in \Sur(V,G_i) \\ 1 \leq i \leq k}} \P_{M_2,\dots, M_k}(F_1,M_2 F_2, \dots, M_2 \cdots M_k F_k \mbox{ span and of depth $D>1$ over } G') \\
 &\wedge (F_2,M_3 F_3,\dots, M_3 \cdots M_k F_k) \text{ is not a code in its image}) \\
&=O( \exp(-\al n/2))\binom{n}{\lceil \ell(D) \delta n\rceil -1} D^{-n+\ell(D) \delta n}|G'|^n.
 \end{split}
\end{align*}
Together with \eqref{eq:sum_code_induction''}, this estimate completes the proof of  \eqref{eq:joint:k'}.

\end{proof}

%To conclude, we relate our proof above to the treatment for $k=2$ in Subsection \ref{sub:k=2}. In part ({\bf a}) of Subsection \ref{sub:k=2} the group $G''$ is just $G_2$. There, the case $F_2$ is a code in $G_2$ in (i) of ({\bf a}) corresponds to the case (i) of this current discussion where $(F_2,M_3 F_3,\dots, M_3 \cdots M_k F_k)$ spans and is a code over $G''$. Similarly, the case (ii) of ({\bf a}) of Subsection \ref{sub:k=2} that $F_2$ spans but is not a code in $G_2$ corresponds to the case (ii) of this current subsection where $(F_2,M_3 F_3,\dots, M_3 \cdots M_k F_k)$ spans but is not a code over $G''$.

\section{Hall-Littlewood polynomial background}\label{sec:HL_background}

This section contains standard definitions and results on Hall-Littlewood polynomials. We also introduce the ring of symmetric functions, which may be thought of as a ring meant to model symmetric polynomials in infinitely many variables, and is needed to obtain measures such as the Cohen-Lenstra measure in the Hall-Littlewood process formalism. This material may be found in \cite[Chapter III]{Macd}, and the setup of Hall-Littlewood processes which may be found for instance in \cite{BC14}; much of the material below is quoted with little modification from \cite[Section 2]{VP21q}.

We denote by $\BBY$ the set of all integer partitions $(\la_1,\la_2,\ldots)$, i.e. sequences of nonnegative integers $\la_1 \geq \la_2 \geq \cdots$ which are eventually $0$. We call the integers $\la_i$ the \emph{parts} of $\la$, set $\la_i' = \#\{j: \la_j \geq i\}$, and write $m_i(\la) = \#\{j: \la_j = i\} = \la_i'-\la_{i+1}'$. We write $\len(\la)$ for the number of nonzero parts, and denote the set of partitions of length $\leq n$ by $\BBY_n$. We write $\mu \prec \la$ or $\la \succ \mu$ if $\la_1 \geq \mu_1 \geq \la_2 \geq \mu_2 \geq \cdots$, and refer to this condition as \emph{interlacing}. Finally, we denote the partition with all parts equal to zero by $\emptyset$.

We denote by $\Lambda_n$ the ring $\C[x_1,\ldots,x_n]^{S_n}$ of symmetric polynomials in $n$ variables $x_1,\ldots,x_n$. It is a very classical fact that the power sum symmetric polynomials $p_k(x_1,\ldots,x_n) = \sum_{i=1}^n x_i^k, k =1,\ldots,n$, are algebraically independent and algebraically generate $\L_n$. For a symmetric polynomial $f$, we will often write $f(\bx)$ for $f(x_1,\ldots,x_n)$ when the number of variables is clear from context. We will also use the shorthand $\bx^\la := x_1^{\la_1} x_2^{\la_2} \cdots x_n^{\la_n}$ for $\la \in \BBY_n$. 

One has a chain of maps
\[
\cdots \to \L_{n+1} \to \L_n \to \L_{n-1} \to \cdots \to 0
\]
where the map $\L_{n+1} \to \L_n$ is given by setting $x_{n+1}$ to $0$. %The inverse limit of this system is $\wt{\L}$, the \emph{extended ring of symmetric functions}. Alternatively, 
In fact, writing $\Lambda_n^{(d)}$ for symmetric polynomials in $n$ variables of total degree $d$, one has 
\[
\cdots \to \L_{n+1}^{(d)} \to \L_n^{(d)} \to \L_{n-1}^{(d)} \to \cdots \to 0
\]
with the same maps. The inverse limit $\L^{(d)}$ of these systems may be viewed as symmetric polynomials of degree $d$ in infinitely many variables. From the ring structure on each $\L_n$ one gets a natural ring structure on $\L := \bigoplus_{d \geq 0} \L^{(d)}$, and we call this the \emph{ring of symmetric functions}. %There is a natural inclusion $\L \inj \wt{\L}$, but $\wt{\L}$ includes elements such as 
%These maps preserve the degree of polynomials, and taking the inverse limit of this system in each degree and collecting these in a direct sum yields $\Lambda$, the ring of symmetric functions, see \cite{Macd}. 
An equivalent definition is $\Lambda := \BBC[p_1,p_2,\ldots]$ where $p_i$ are indeterminates; under the natural map $\Lambda \to \Lambda_n$ one has $p_i \mapsto p_i(x_1,\ldots,x_n)$. 

Each ring $\L_n$ has a natural basis $\{p_\la: \la_1 \leq n\}$ where
\begin{equation*}
    p_\la := \prod_{i \geq 1} p_{\la_i}.
\end{equation*}
Another natural basis, with the same index set, is given by the \emph{Hall-Littlewood polynomials}. Recall the $q$-Pochhammer symbol defined as $(a;q)_n := \prod_{i=0}^{n-1} (1-aq^i)$, and for $n \geq 0, \la \in \BBY_n$ define
\begin{equation*}
    v_{\la,n}(t) = \frac{(t;t)_{n-\len(\la)}}{(1-t)^{n-\len(\la)}}\prod_{i \geq 1} \frac{(t;t)_{m_i(\la)}}{(1-t)^{m_i(\la)}}.
\end{equation*}

\begin{definition}\label{def:HL}
The Hall-Littlewood polynomial indexed by $\la \in \BBY_n$ is
\begin{equation}\label{eq:hlP_formula}
    P_\la(x_1,\ldots,x_n;t) = \frac{1}{v_{\la,n}(t)} \sum_{\sigma \in S_n} \sigma\left(\bx^\la \prod_{1 \leq i < j \leq n} \frac{x_i-tx_j}{x_i-x_j}\right)
\end{equation}
where $\sigma$ acts by permuting the variables. We often drop the `$;t$' when clear from context.
\end{definition}

It follows from the definition that $P_\la(x_1,\ldots,x_n,0;t) = P_\la(x_1,\ldots,x_n;t)$, hence for each $\la \in \BBY$ there is a \emph{Hall-Littlewood symmetric function} $P_\la \in \L$.

\begin{definition}
For $\la \in \BBY$, we define the dual Hall-Littlewood polynomial by
\[
Q_\la(\bx;t) = \prod_{i \geq 1} (t;t)_{m_i(\la)} P_\la(\bx;t).
\]
These similarly are consistent under maps $\L_{n+1} \to \L_n$ and hence define symmetric functions.
\end{definition}

%We note that in the case where $\la \in \BBY_n$ has some parts equal to $0$, this normalization is \emph{not} the same as the standard one in e.g. \cite{Macd}, though the two agree when $\l$ has all parts positive. We choose this normalization because of our later conventions on skew polynomials.

Because the $P_\la$ form a basis for the vector space of symmetric polynomials in $n$ variables, there exist symmetric polynomials $P_{\la/\mu}(x_1,\ldots,x_{n-k};t) \in \L_{n-k}$ indexed by $\la \in \BBY_n, \mu \in \BBY_k$ which are defined by
\begin{equation}\label{eq:def_skewP}
    P_\la(x_1,\ldots,x_n;t) = \sum_{\mu \in \BBY_k} P_{\la/\mu}(x_{k+1},\ldots,x_n;t) P_\mu(x_1,\ldots,x_k;t).
\end{equation}
The definition of $Q_{\la/\mu}$ is exactly analogous. As with non-skew Hall-Littlewood polynomials, the skew versions are consistent under the maps $\L_{n+1} \to \L_n$ and hence define symmetric functions in $\L$, which we also denote by $P_{\la/\mu}$ and $Q_{\la/\mu}$.

\begin{comment}%don't need branching rule
\begin{definition}\label{def:psi_varphi_coefs}
For $\mu,\la \in \BBY$ with $\mu \succ \l$, let
\begin{equation*}\label{eq:pbranch}
    \psi_{\mu/\la} :=  \prod_{\substack{i > 0\\ m_i(\la) = m_i(\mu)+1}} (1-t^{m_i(\la)}) 
\end{equation*}
and
\begin{equation*}\label{eq:qbranch}
    \varphi_{\mu/\la} :=  \prod_{\substack{i > 0 \\  m_i(\mu) = m_i(\la)+1}} (1-t^{m_i(\mu)}) 
\end{equation*}
\end{definition}

The following branching rule is standard.

\begin{lemma} \label{thm:branching_formulas}
For $\l,\mu \in \BBY$, we have
\begin{equation}\label{eq:skewP_branch_formula}
    P_{\la/\mu}(x_1,\ldots,x_k) = \sum_{\mu = \la^{(1)} \prec \la^{(2)} \prec \cdots \prec \la^{(k)}= \la} \prod_{i=1}^{k-1} x_i^{|\la^{(i+1)}|-|\la^{(i)}|}\psi_{\la^{(i+1)}/\la^{(i)}}%P_{\la^{(i)}/\la^{(i+1)}}(x_i)
\end{equation}
and
\begin{equation}\label{eq:skewQ_branch_formula}
    Q_{\la/\mu}(x_1,\ldots,x_k) = \sum_{\mu = \la^{(1)} \prec \la^{(2)} \prec \cdots \prec \la^{(k)}=\la} \prod_{i=1}^{k-1} x_i^{|\la^{(i+1)}|-|\la^{(i)}|}\varphi_{\la^{(i+1)}/\la^{(i)}}.
\end{equation}
\end{lemma}
\end{comment}

Hall-Littlewood polynomials and symmetric functions satisfy the \emph{skew Cauchy identity}, upon which most probabilistic constructions rely. For polynomials in a finite number of variables, it reads
\begin{multline}\label{eq:finite_cauchy}
    \sum_{\kappa \in \BBY} P_{\kappa/\nu}(x_1,\ldots,x_n;t)Q_{\kappa/\mu}(y_1,\ldots,y_m;t) \\
    = \prod_{\substack{1 \leq i \leq n \\ 1 \leq j \leq m}} \frac{1-tx_iy_j}{1-x_iy_j} \sum_{\la \in \BBY} Q_{\nu/\la}(y_1,\ldots,y_m;t) P_{\mu/\la}(x_1,\ldots,x_n;t).
\end{multline}

For later convenience we set 
\begin{equation}\label{eq:def_cauchy_kernel}
     \Pi_t(\bx;\by) := \prod_{\substack{1 \leq i \leq n \\ 1 \leq j \leq m}} \frac{1-tx_iy_j}{1-x_iy_j} = \exp\left(\sum_{\ell = 1}^\infty \frac{1-t^\ell}{\ell}p_\ell(\bx)p_\ell(\by)\right).
\end{equation}
The second equality in \eqref{eq:def_cauchy_kernel} is not immediate but is shown in \cite{Macd}. The RHS of \eqref{eq:def_cauchy_kernel} makes sense as formal series in a suitable completion of $\L \ot \L$, and \eqref{eq:finite_cauchy} generalizes straightforwardly with the skew $P$ and $Q$ functions in $x$ and $y$ replaced by corresponding elements of $\L \ot \L$.

\begin{proposition}\label{prop:cauchy}
Let $\nu, \mu \in \BBY$. Then
\begin{equation}\label{eq:infinite_cauchy}
    \sum_{\kappa \in \BBY} P_{\kappa/\nu}(\bx;t)Q_{\kappa/\mu}(\by;t) 
    = \Pi_t(\bx;\by) \sum_{\la \in \BBY} Q_{\nu/\la}(\by;t) P_{\mu/\la}(\bx;t).
\end{equation}
In particular, if $\mu = \emptyset$ we have 
\begin{equation}\label{eq:infinite_cauchy_special}
    \sum_{\kappa \in \BBY} P_{\kappa/\nu}(\bx;t)Q_{\kappa}(\by;t) 
    = \Pi_t(\bx;\by) Q_{\nu}(\by;t).
\end{equation}
\end{proposition}

Let us now set the parameter $t$ to be real. We would like to define probabilities by substituting real numbers for the variables of Hall-Littlewood polynomials. The analogue of `specializing infinitely many variables' for a general symmetric function is captured as follows.

\begin{definition}
Given a sequence of real numbers $\alpha_1 \geq \alpha_2 \geq \ldots \geq 0$ with $\sum_i \alpha_i < \infty$, the pure alpha specialization with parameters $\balpha = \{\alpha_i\}_{i \geq 1}$ is the homomorphism $\L \to \BBC$ defined on the generators $p_k$ by
$$p_k(\balpha) = \sum_{i \geq 1} \alpha_i^k.$$
For a general symmetric function $f \in \L$, we write $f(\balpha)$ for the image of $f$ under this homomorphism.
\end{definition}

When $\balpha$ has only finitely many (say, $k$) nonzero $\alpha_i$ parameters, $f(\alpha)$ is just the symmetric polynomial $\Im(f) \in \Lambda_k$ with $x_1=\alpha_1,\ldots,x_k = \alpha_k$ plugged in for the variables. Extending the notation of \eqref{eq:def_cauchy_kernel}, we write 
$$\Pi_t(\balpha;\balpha') =\exp\left(\sum_{\ell = 1}^\infty \frac{1-t^\ell}{\ell}p_\ell(\balpha)p_\ell(\balpha')\right).$$

Of course, one can substitute any real or complex numbers for the variables, but choosing $\alpha_i \geq 0$ has the advantage for probability that when $t \in [0,1]$, $P_\la(\balpha)$ and $Q_\la(\balpha)$ are nonnegative.

\begin{remark}
There are other homomorphisms $\L \to \C$ with nonnegative values on the $P_\la$, of which the pure alpha specializations form one family. These were classified in \cite{M19}.
\end{remark}

One obtains probability measures on sequences of partitions using Proposition~\ref{prop:cauchy} as follows.

\begin{definition}\label{def:HL_proc}
Let $\mathbf{\theta}$ and $\balpha_1,\ldots,\balpha_k$ be pure alpha specializations satisfying
$$\sum_{\la \in \BBY}P_\la(\mathbf{\theta})Q_\la(\balpha_i) < \infty$$
for each $i$. Then the associated \emph{ascending Hall-Littlewood process} is the probability measure on sequences $(\la^{(1)},\ldots,\la^{(k)}) \in \BBY^k$ given by 
\[
\Pr(\la^{(1)},\ldots,\la^{(k)}) = \frac{Q_{\la^{(1)}}(\balpha_1) Q_{\la^{(2)}/\la^{(1)}}(\balpha_2) \cdots Q_{\la^{(k)}/\la^{(k-1)}}(\balpha_k) P_{\la^{(k)}}(\mathbf{\theta})}{\prod_{i=1}^k \Pi_t(\balpha_i;\mathbf{\theta})}.
\]
\end{definition}

The $k=1$ case of Definition~\ref{def:HL_proc} is a measure on partitions, referred to as a \emph{Hall-Littlewood measure}. Two special cases will be relevant for our setting. Below and subsequently, we use the notation $x[k]$ in the arguments of Hall-Littlewood polynomials to denote the variable $x$ repeated $k$ times. 

\begin{definition}\label{def:P(k)} For $t \in (0,1)$ and $k \in \BBZ_{\geq 1}$, we define the probability measure $\P^{(k)}_{\infty; t}$ on $\BBY^k$ by
$$\P^{(k)}_{\infty; t} (\la^{(1)},\dots, \la^{(k)}) :=\frac{Q_{\la^{(1)}} (1,t,\dots) Q_{\la^{(2)}/\la^{(1)}}(1,t,\dots) \dots Q_{\la^{(k)}/\la^{(k-1)}}(1,t,\dots)  P_{\la^{(k)}} (t,t^2, \dots,)}{\Pi_t (1,t, \dots; t[k],t^2[k],\ldots)}.$$
\end{definition}

\begin{definition}\label{def:Past} For $t \in (0,1)$ and $k \in \BBZ_{\geq 1}$, we define the probability measure $\P^{\ast(k)}_{\infty;t}$ on $\BBY$ by
 $k$
$$\P^{\ast (k)}_{\infty; t}(\la) = \frac{Q_\la(1[k],t[k],\ldots)P_\la(t,t^2,\ldots)}{\Pi_t (1,t, \dots; t[k],t^2[k],\ldots)}.$$
\end{definition}

It follows from the branching rule \eqref{eq:def_skewP} that the marginal distribution of $\la^{(k)}$ under $\P^{(k)}_{\infty; t}$ is $\P^{\ast (k)}_{\infty; t}$. When $t=1/p$, random sequence of partitions specified by $\P^{(k)}_{\infty; t}$ is related to the random sequence of groups in Theorem \ref{theorem:main:joint} with $P=\{p\}$, and similarly $\P^{\ast (k)}_{\infty; t}$ is related to Theorem \ref{theorem:main:prod}. We discuss this in the next section.

\begin{remark}
We use a slightly different setup for Hall-Littlewood polynomials than in some previous works \cite{VP21limits,VP22hall} on Hall-Littlewood polynomials and $p$-adic random matrices. The reason is that those works considered random matrices over $\BBQ_p$, and consequently it was desirable to extend the indices of Hall-Littlewood polynomials to `partitions' with negative parts allowed, which required modifying the standard notation of \cite{Macd} slightly. Because we work only over $\BBZ_p$, there is no necessity to do this, and our notation follows that of \cite{Macd}. 
\end{remark}

\section{Hall-Littlewood polynomials and abelian $p$-groups}\label{sec:HL_and_groups}

The goal of this section is to prove several basic results giving formulas, in terms of Hall-Littlewood polynomials, for various counts of maps between abelian $p$-groups. All follow straightforwardly from the material in \cite[Chapters II and III]{Macd}, but most do not seem to be present in the random matrix theory literature. We therefore hope this section will have some value in translating between the usual terminology of moments and the Hall-Littlewood notation used in e.g. \cite{FK,VP21limits,VP22hall,VP22+universal}. Here and in the next section, we will usually fix a prime $p$ and let $t=1/p$ to declutter notation and keep with standard Hall-Littlewood usage. 

%The reason

%Within this section we denote by $G_\lambda$ the abelian $p$-group of type $\lambda$ for any partition $\lambda$. 

\begin{definition}\label{def:sbgp_number}
For any partition $\la \in \BBY$, we denote by $G_\la$ the abelian $p$-group $\bigoplus_{i \geq 1} \Z/p^{\la_i}\Z$ when $p$ is fixed and clear from context. The \emph{type} of a finite abelian $p$-group $H$ is the partition $\la$ for which $H \cong G_\la$. For partitions $\la,\mu \in \BBY$, we denote by $G_{\mu,\lambda}$ the set of subgroups of $G_\lambda$ which have type $\mu$.
\end{definition}

\begin{lemma} \label{lemma:from_mac}
With $t=1/p$ as above,
$$|G_{\mu,\lambda}| = \frac{Q_{\la/\mu}(1,t,\ldots) Q_\mu(1,t,\ldots)}{Q_\la(1,t,\ldots)}.
$$
\end{lemma}
\begin{proof}
The result follows by collecting a few facts from chapters II and III of \cite{Macd}. By their definition \cite[Chapter II.2 (2.1)]{Macd}, the Hall algebra structure constant is 
\begin{equation}\label{eq:def_hall_poly}
G_{\mu,\nu}^\la(\BBZ_p) = \#\{H \subset G_\la: H \simeq G_\mu \text{ and }G_\la/H \simeq G_\nu\}.
\end{equation}
It is then shown in \cite[Chapter III.3 (3.4)]{Macd} that
\begin{equation}\label{eq:hall_and_hl}
G_{\mu,\nu}^\la(\BBZ_p) = c_{\mu,\nu}^\la(t) t^{n(\mu)+n(\nu)-n(\la)} = c_{\mu,\nu}^\la(t) \frac{Q_\mu(1,t,\ldots)Q_\nu(1,t,\ldots)}{Q_\la(1,t,\ldots)},
\end{equation}
where $n(\la) = \sum_{i \geq 1} (i-1)\la_i$ and the second equality follows by \cite[Chapter III.3 Ex. 2(a)]{Macd}. Here $c_{\mu,\nu}^\la(t)$ are the multiplicative structure constants defined by 
$$P_\mu(\mathbf{x}) \cdot P_\nu(\mathbf{x}) = \sum_{\la \in \BBY} c_{\mu,\nu}^\la(t) P_\la(\mathbf{x})$$
or equivalently the comultiplicative structure constants of the $Q$ polynomials defined by 
$$Q_\la(X,Y) = \sum_{\mu,\nu \in \BBY} c_{\mu,\nu}^\la(t) Q_\mu(\mathbf{x}) \cdot Q_\nu(\mathbf{y}).$$
The second definition is equivalent to 
$$Q_{\la/\mu}(\mathbf{x}) = \sum_{\nu \in \BBY} c_{\mu,\nu}^\la(t) Q_\nu(\mathbf{x}).$$
Hence the number of subgroups of $G_\lambda$ of type $\mu$ is
$$\sum_{\nu \in \BBY}G_{\mu,\nu}^\la(\BBZ_p) = \sum_{\nu \in \BBY} c_{\mu,\nu}^\la(t) \frac{Q_\mu(1,t,\ldots)Q_\nu(1,t,\ldots)}{Q_\la(1,t,\ldots)}= \frac{Q_{\la/\mu}(1,t,\ldots) Q_\mu(1,t,\ldots)}{Q_\la(1,t,\ldots)}.$$
\end{proof}

\begin{proposition}\label{prop:HL_pgrp_translate}
For any $\lambda,\mu \in \BBY$, we have
\begin{equation}\label{eq:HL_aut}
P_\la(t,t^2,\ldots)Q_\la(1,t,\ldots) = \frac{1}{\#\Aut(G_\la)}
\end{equation}
and more generally
\begin{equation}\label{eq:HL_inj_surj}
\frac{P_{\la/\mu}(t,t^2,\ldots)}{P_\la(t,t^2,\ldots)Q_\mu(1,t,\ldots)} = \frac{Q_{\la/\mu}(1,t,\ldots)}{Q_\la(1,t,\ldots)P_\mu(t,t^2,\ldots)} = \#\Sur(G_\la,G_\mu) = \#\Inj(G_\mu,G_\la).
\end{equation}
\end{proposition}

\begin{proof}
The first part, \eqref{eq:HL_aut}, follows directly from \cite[Chapter III.3 Ex. 2(a)]{Macd}. This together with Lemma~\ref{lemma:from_mac} yields
%\begin{equation}\label{eq:compute_inj}
$$
\#\Inj(G_\mu,G_\la) = \#\Aut(G_\mu) \cdot \#\{H \leq G_\la: H \simeq G_\mu\} = \frac{1}{P_\mu(t,t^2,\ldots)Q_\mu(1,t,\ldots)}\frac{Q_{\la/\mu}(1,t,\ldots) Q_\mu(1,t,\ldots)}{Q_\la(1,t,\ldots)}.
$$
The fact that the RHS is equal to 
$$\frac{P_{\la/\mu}(t,t^2,\ldots)}{P_\la(t,t^2,\ldots)Q_\mu(1,t,\ldots)} $$
follows since $Q_\la$ is a constant multiple of $P_\la$ and is homogeneous. Since a surjection $G_\la \surj G_\mu$ induces an injection $G_\mu^* \inj G_\la^*$ and vice versa, and finite abelian groups are isomorphic to their dual groups,
$$\#\Sur(G_\la,G_\mu) = \#\Inj(G_\mu^*,G_\la^*) = \#\Inj(G_\mu,G_\la),$$
hence we have established \eqref{eq:HL_inj_surj}.
\end{proof}

The following `joint moment' result will be useful later. It is also a natural generalization of the well-known fact that moments of the Cohen-Lenstra distribution are $1$, and reducing to this fact when $M$ is trivial.

\begin{proposition}\label{prop:what_cauchy_means}
Let $M,N$ be finite abelian $p$-groups. Then 
$$\frac{1}{\Pi_t(1,t,\ldots;t,t^2,\ldots)} \sum_K \frac{\#\Inj(M,K)\#\Sur(K,N)}{\#\Aut(K)} = \#\Hom(M,N)$$
where the sum is one representative $K$ from each isomorphism class of finite abelian $p$-groups.
\end{proposition}

\begin{proof}
Letting $\mu,\nu$ be the types of $M,N$ respectively, by Proposition~\ref{prop:HL_pgrp_translate} the LHS is
$$\frac{1}{\Pi_t(1,t,\ldots;t,t^2,\ldots)}\sum_{\kappa \in \BBY} \frac{\PP{\kappa/\mu}}{\PP{\kappa}\QQ{\mu}} \frac{\QQ{\kappa/\nu}}{\QQ{\kappa}\PP{\nu}} \QQ{\kappa}\PP{\kappa}.$$
By the skew Cauchy identity stated in Proposition~\ref{prop:cauchy}, this is 
\begin{equation}\label{eq:hom_sum}
\frac{1}{\QQ{\mu}\PP{\nu}}\sum_{\la \in \BBY} \PP{\nu/\la}\QQ{\mu/\la} = \sum_{\la \in \BBY} \frac{\#\Sur(M,G_\la)\#\Inj(G_\la,N)}{\#\Aut(G_\la)}.
\end{equation}
where we again used Proposition~\ref{prop:HL_pgrp_translate}. Every map $\varphi: M \to N$ factors as $\varphi = \phi \circ \psi$ where $\psi: M \surj \Im(\varphi)$ is surjective and $\phi: \Im(\varphi) \inj N$ is injective. There are $\#\Aut(\Im(\varphi))$ such factorizations, because for any $\sigma \in \Aut(\Im(\varphi))$, $(\phi \circ \sigma^{-1}) \circ (\sigma \circ \psi)$ also defines a pair of an injection and surjection, and $\Aut(\Im(\varphi))$ acts transitively on the set of such pairs with trivial stabilizer. Hence 
$$\text{RHS\eqref{eq:hom_sum}} = \#\Hom(M,N),$$
completing the proof.
\end{proof}

\begin{lemma}\label{lemma:n_k_HL}
For $\la \in \BBY, k \geq 1$ and $n_k$ as in Definition~\ref{def:n_k}, we have
\begin{equation}\label{eq:HL_sbgp_chain}
n_k(G_\la) = \frac{P_\la(t[k],t^2[k],\ldots)}{P_\la(t,t^2,\ldots)} = \frac{Q_\la(1[k],t[k],\ldots)}{Q_\la(1,t,\ldots)} .
\end{equation}
\end{lemma}
%and
\begin{proof}
We induct on $k$, the base case $k=1$ being trivial. It follows from \eqref{eq:HL_aut} and \eqref{eq:HL_inj_surj} that for fixed $G_\la$,
\begin{align*}
n_{k+1}(G_\la) &= \sum_{H \leq G_\la} n_k(H) \\
&= \sum_{\mu \in \BBY} \frac{\#\Inj(G_\mu,G_\la)}{\#\Aut(G_\mu)} n_k(G_\mu) \\
&= \sum_{\mu \in \BBY} \frac{\PP{\la/\mu}\PP{\mu}}{\PP{\la}} \frac{P_\mu(t[k],t^2[k],\ldots)}{P_\mu(t,t^2,\ldots)} \\
&= \frac{P_\la(t[k+1],\ldots)}{\PP{\la}},
\end{align*}
where the last equality is by the branching rule. Since $Q_\la$ is a constant multiple of $P_\la$ and is homogeneous, the second equality of \eqref{eq:HL_sbgp_chain} follows.
\end{proof}

%\begin{proposition}\label{prop:product_HL_measure}
%Let $M_1,\ldots,M_k$ be iid 
%this is the prop that says that the haar model converges to hl meaesure, maybe we don't actually need that though if we prove a universality statement
%\end{proposition}

\section{Moments and joint moments of the candidate limit distributions}\label{sect:Haar}

In this section we compute the moments and joint moments of the limiting distributions appearing in Theorems \ref{theorem:main:prod} and \ref{theorem:main:joint}, by relating them to the Hall-Littlewood framework of the last two sections. Below, for a finite set of primes $P$ we use the notation $\CA_P$ for the set of all abelian groups $G$ such that every prime factor of $|G|$ lies in $P$. We begin by defining notation for the probability measures appearing in Theorems \ref{theorem:main:prod} and \ref{theorem:main:joint}; we shortly show that these expressions do indeed define probability measures.

\begin{definition}\label{def:intro_measures} 
For $k \in \BBZ_{\geq 1}$ and $P$ a finite set of primes, given abelian groups $B,B_1,\dots, B_k \in \CA_P$, we let 
\begin{equation}\label{eq:intro_prod_measure}
\P_P^{\ast(k)}(B) := \left(\prod_{p\in P} (p^{-1};p^{-1})_\infty^k\right)  \frac{\#\{0 = G_0 \leq G_1 \leq \ldots \le G_k = B\}}{\#\Aut(B)}
\end{equation}
and 
\begin{equation}\label{eq:intro_joint_measure}
\P_P^{(k)}(B_1,\dots, B_k) := \left(\prod_{p\in P}(p^{-1};p^{-1})_\infty^k\right) \prod_{i=1}^k \frac{\#\Sur(B_i, B_{i-1})}{\#\Aut(B_i)},
\end{equation}
where we take $B_0$ to be the trivial group in the product.
\end{definition}

The notation is meant to be suggestive of the Hall-Littlewood measures $\P_{t;\infty}^{\ast(k)}$ and $\P_{t;\infty}^{(k)}$, and we relate the two measures in this section. In the remainder of the section we use $P$ and $k$ as in Definition \ref{def:intro_measures} without comment.

\begin{theorem}\label{thm:intro_dist_moment_prod}
The map $B \mapsto \P_P^{\ast(k)}(B)$ defines a probability measure on $\CA_P$, with moments 
$$\E_{B \sim \P_P^{\ast(k)}}[\#\Sur(B,G)] = n_k(G)$$
for any $G \in \CA_P$.
\end{theorem}

\begin{theorem}\label{thm:intro_dist_moment_joint}
The map $(B_1,\ldots,B_k) \mapsto \P_P^{(k)}(B_1,\ldots,B_k)$ defines a probability measure on $\CA_P^k$, with joint moments
$$\E_{(B_1,\ldots,B_k) \sim \P_P^{(k)}}[\#\Sur(B_1,G_1) \cdots \#\Sur(B_k,G_k)] = m_k(G_1,\ldots,G_k)$$
for any $G_1,\ldots,G_k \in \CA_P$.
\end{theorem}

We begin with factorization properties which reduce the theorem to the case of a single prime.

\begin{lemma}\label{lem:n_m_factorize}
With the notations above, the following factorizations hold:
\begin{enumerate}
  \item For any $G \in \CA_P$,
$$n_k(G) = \prod_{p \in P} n_k(G[p^\infty])$$
and 
$$\P_P^{\ast(k)}(G) = \prod_{p \in P} \P_{\{p\}}^{\ast(k)}(G[p^\infty]).$$
  \item For any $G^{(1)},\ldots,G^{(k)} \in \CA_P$,
$$m_k(G^{(1)},\ldots,G^{(k)}) = \prod_{p \in P} m_k(G^{(1)}[p^\infty],\ldots,G^{(k)}[p^\infty])$$
and 
$$\P_P^{(k)}(G^{(1)},\ldots,G^{(k)}) = \prod_{p \in P} \P_{\{p\}}^{(k)}(G^{(1)}[p^\infty],\ldots,G^{(k)}[p^\infty]).$$
\end{enumerate}
\end{lemma}

The proof is trivial. We next relate to the Hall-Littlewood measures $\P_{\infty;1/p}^{\ast(k)}, \P_{\infty;1/p}^{(k)}$ in Definitions \ref{def:Past} and \ref{def:P(k)}.
We will not actually need the joint version, Proposition \ref{prop:joint_aut_hl}, in the proofs, but it helps explain the origins of the measures we consider.

\begin{proposition}\label{prop:prod_aut_hl}
For any prime $p$ and $\la \in \BBY$,
$$\P_{\{p\}}^{\ast(k)}(G_\la) = \P_{\infty;1/p}^{\ast(k)}(\la).$$
\end{proposition}

\begin{proposition}\label{prop:joint_aut_hl}
For any prime $p$ and $\la^{(1)},\ldots,\la^{(k)} \in \BBY$,
$$\P_{\{p\}}^{(k)}(G_{\la^{(1)}},\ldots,G_{\la^{(k)}}) = \P_{\infty;1/p}^{(k)}(\la^{(1)},\ldots,\la^{(k)}).$$
\end{proposition}

% Finally, we compute the moments and joint moments.

% \begin{proposition}\label{prop:surmoment:k:Haar} For any finite abelian $p$-group $G$ we have
% $$\sum_{\la} \P^{\ast(k)}_{\infty; 1/p}(\la) \#\Sur(G_\la, G) = n_k(G).$$
% \end{proposition}

% \begin{proposition}\label{prop:jointsur:k:Haar} For any finite abelian $p$-groups $G_1,\dots, G_k$ we have
% $$\E_{\la^{(1)},\ldots,\la^{(k)} \sim \P_{\{p\}}^{(k)}} \left[\# \Sur(G_{\la^{(1)}}, G_1) \times \dots \times \# \Sur(G_{\la^{(k)}},G_k)\right] = m_k (G_1,\ldots,G_k).$$
% \end{proposition}

% Clearly, combining the above results yields Theorems \ref{thm:intro_dist_moment_prod} and \ref{thm:intro_dist_moment_joint}. We now treat the intermediate steps, beginning with those for Theorem \ref{thm:intro_dist_moment_prod}. \fix{remove}

\begin{proof}[Proof of Proposition \ref{prop:prod_aut_hl}]
Apply Lemma \ref{lemma:n_k_HL} to the numerator and Proposition \ref{prop:HL_pgrp_translate} to the denominator of the formula for $\P_{\{p\}}^{\ast(k)}(G_\la)$, and for the normalizing constants note that 
\begin{equation}\label{eq:qpoc_cauchy}
\Pi_{1/p}(1,p^{-1},\ldots; p^{-1},p^{-2},\ldots)  = \frac{1}{(p^{-1};p^{-1})_\infty}.
\end{equation}
\end{proof}

% \begin{proof}[Proof of Proposition \ref{prop:surmoment:k:Haar}]
% Let $\mu$ be the type of $G$. By Proposition~\ref{prop:HL_pgrp_translate},
% \begin{align*}
% \sum_{\la \in \BBY} \P^{\ast(k)}_{\infty; 1/p}(\la) \#\Sur(G_{\la}, G) =  \sum_{\la \in \BBY} \frac{Q_\la(1[k],\ldots)P_\la(t,\ldots)}{\Pi_t(t,\ldots;1[k],t[k],\ldots)} \frac{\PP{\la/\mu}}{\PP{\la}\QQ{\mu}}.
% \end{align*}
% By the Cauchy identity (Proposition~\ref{prop:cauchy}) this is 
% $$\frac{Q_\mu(1[k],\ldots)}{\QQ{\mu}},$$
% which is $n_k(G)$ by Lemma~\ref{lemma:n_k_HL}.
% \end{proof}

\begin{proof}[Proof of Theorem \ref{thm:intro_dist_moment_prod}]
By the factorizations of Lemma~\ref{lem:n_m_factorize} and the factorization 
$$\#\Sur(B,G) = \prod_{p \in P} \#\Sur(B[p^\infty],G[p^\infty]),$$
it suffices to prove Theorem~\ref{thm:intro_dist_moment_prod} in the case $P=\{p\}$ for some prime $p$. The fact that $\P^{\ast(k)}_{\{p\}}$ is a probability measure follows by Proposition~\ref{prop:prod_aut_hl} since the Hall-Littlewood measure is a probability measure. For the computation of moments in Theorem \ref{thm:intro_dist_moment_prod}, by Proposition~\ref{prop:prod_aut_hl} and the above reduction to a single prime it suffices to show that 
\begin{equation}\label{eq:used_to_be_prop}
\sum_{\la} \P^{\ast(k)}_{\infty; 1/p}(\la) \#\Sur(G_\la, G) = n_k(G)
\end{equation}
for any abelian $p$-group $G$. To prove this, let $\mu$ be the type of $G$. By Proposition~\ref{prop:HL_pgrp_translate},
\begin{align*}
\sum_{\la \in \BBY} \P^{\ast(k)}_{\infty; 1/p}(\la) \#\Sur(G_{\la}, G) =  \sum_{\la \in \BBY} \frac{Q_\la(1[k],\ldots)P_\la(t,\ldots)}{\Pi_t(t,\ldots;1[k],t[k],\ldots)} \frac{\PP{\la/\mu}}{\PP{\la}\QQ{\mu}}.
\end{align*}
By the Cauchy identity (Proposition~\ref{prop:cauchy}) this is 
$$\frac{Q_\mu(1[k],\ldots)}{\QQ{\mu}},$$
which is $n_k(G)$ by Lemma~\ref{lemma:n_k_HL}.
%The computation of moments in Theorem \ref{thm:intro_dist_moment_prod} follows by the computation for the Hall-Littlewood measure in Proposition \ref{prop:surmoment:k:Haar}, together with Proposition \ref{prop:prod_aut_hl} to equate the two measures.
\end{proof}

We now turn to Theorem \ref{thm:intro_dist_moment_joint}. The intermediate steps are, as one might expect, slightly more involved.

\begin{proof}[Proof of Proposition \ref{prop:joint_aut_hl}]
By Proposition \ref{prop:HL_pgrp_translate}, 
$$\frac{\#\Sur(G_{\la^{(i)}},G_{\la^{(i-1)}})}{\#\Aut(G_{\la^{(i)}})} = Q_{\la^{(i)}/\la^{(i-1)}}(1,t,\ldots) \frac{P_{\la^{(i)}}(t,t^2,\ldots)}{P_{\la^{(i-1)}}(t,t^2,\ldots)},$$
and the proof follows from this and \eqref{eq:qpoc_cauchy}.
\end{proof}

Finally, the proof of Theorem~\ref{thm:intro_dist_moment_joint} goes by iterating the following key lemma.

\begin{lemma}\label{lemma:jointsur_iterate}
Let $L,M, N$ be finite abelian $p$-groups. Then 
\begin{equation}\label{eq:LKMN}
%\E_{}\left[  \right] = 
\E_{\kappa \sim \P_{\infty;1/p}^{(1)}} \left[\#\Inj(L, G_\kappa) \#\Sur(G_\kappa,M) \#\Hom(G_\kappa,N)   \right] = \sum_{\substack{H \leq M \op N: \\ \pi_1(H) = M}} \#\Hom(L,H).
\end{equation}
\end{lemma}

We note that when $N$ is trivial, Lemma~\ref{lemma:jointsur_iterate} reduces to Proposition~\ref{prop:what_cauchy_means}.

\begin{proof}[Proof of Lemma \ref{lemma:jointsur_iterate}]
For any $K$, composing $\varphi \in \Hom(K, M \op N)$ with projections onto the two factors yields a natural map
$$\{\varphi \in \Hom(K, M \op N): \pi_1 \circ \varphi \text{ surjective}\} \to \Sur(K,M) \times \Hom(K,N),$$
which is a bijection by the universal property of direct products\footnote{Since we only have a direct sum of a finite number of factors, $M \op N \simeq M \times N$.}. Hence summing over the possible images of $\varphi$ yields
$$\#\Sur(K,M) \#\Hom(K,N) = \sum_{\substack{H \leq M \op N \\ \pi_1(H) = M}} \#\Sur(K,H).$$

By the definition of $\P_{\infty;1/p}^{(1)}$ together with the above discussion, the LHS of \eqref{eq:LKMN} is equal to
\begin{equation}\label{eq:jointsur_almost_done}
(p^{-1};p^{-1})_\infty \sum_K \frac{\#\Inj(L,K) \#\Sur(K,M) \#\Hom(K,N)}{\#\Aut(K)} =  (p^{-1};p^{-1})_\infty \sum_K \frac{\#\Inj(L,K)}{\#\Aut(K)} \sum_{\substack{H \leq M \op N \\ \pi_1(H) = M}} \#\Sur(K,H)
\end{equation}
where again the sum is over all isomorphism classes of finite abelian $p$-groups and $K$ is a representative from the class. Interchanging the sums and applying Proposition~\ref{prop:what_cauchy_means} to \eqref{eq:jointsur_almost_done} completes the proof.
\end{proof}

\begin{proof}[Proof of Theorem~\ref{thm:intro_dist_moment_joint}]%Proposition \ref{prop:jointsur:k:Haar}]
By the factorizations Lemma~\ref{lem:n_m_factorize} and factorization of the number of surjections, it suffices to prove Theorem~\ref{thm:intro_dist_moment_joint} in the case $P=\{p\}$. The fact that $\P_{\{p\}}^{(k)}$ is a probability measure follows from Proposition \ref{prop:joint_aut_hl}, since Hall-Littlewood processes are probability measures. 

It remains to compute moments. In the proof below we use $\la^{(1)},\ldots,\la^{(k)}$ for the types of $B_1,\ldots,B_k$ and $\mu^{(1)},\ldots,\mu^{(k)}$ for the types of $G_1,\ldots,G_k$, and to manage subscripts we will abuse notation and write $\la$ for the group $G_\la$. By duality
\begin{equation}\label{eq:interpret_hl_process}
\P_{\{p\}}^{(k)}(\la^{(1)},\ldots,\la^{(k)}) = (p^{-1};p^{-1})_\infty^k \frac{\#\Inj(\la^{(k-1)}, \la^{(k)}) \cdots \#\Inj(\la^{(1)},\la^{(2)})}{\#\Aut(\la^{(k)}) \cdots \#\Aut(\la^{(1)})}.
\end{equation}

Hence we wish to show 
\begin{align}\label{eq:jsthm_wts}
\begin{split}
& (p^{-1};p^{-1})_\infty^k\sum_{\la^{(1)},\ldots,\la^{(k)} \in \BBY} \frac{\#\Inj(\la^{(k-1)}, \la^{(k)}) \cdots \#\Inj(\la^{(1)},\la^{(2)})}{\#\Aut(\la^{(k)}) \cdots \#\Aut(\la^{(1)})} \#\Sur(\la^{(k)},\mu^{(k)}) \cdots \#\Sur(\la^{(1)},\mu^{(1)}) \\
& \quad \quad \quad = m_k(\mu^{(1)},\ldots,\mu^{(k)})
\end{split}
\end{align}
where again $\mu^{(i)}$ is the type of $G_i$. First apply Proposition~\ref{prop:what_cauchy_means} with $N=\mu^{(k)},M=\la^{(k-1)},K=\la^{(k)}$ to the sum over $\la^{(k)}$ in the LHS to obtain
\begin{align*}
(p^{-1};p^{-1})_\infty^{k-1} \sum_{\la^{(1)},\ldots,\la^{(k-1)} \in \BBY} \#\Hom(\la^{(k-1)},\mu^{(k)}) \frac{\#\Inj(\la^{(k-2)}, \la^{(k-1)}) \cdots \#\Inj(\la^{(1)},\la^{(2)})}{\#\Aut(\la^{k-1}) \cdots \#\Aut(\la^{(1)})} \\ 
\times  \#\Sur(\la^{(k-1)},\mu^{(k-1)}) \cdots \#\Sur(\la^{(1)},\mu^{(1)}).
\end{align*}

Then apply Lemma~\ref{lemma:jointsur_iterate} with $L=\la^{(k-2)},\kappa=\la^{(k-1)},M=\mu^{(k-1)},N=\mu^{(k)}$ to obtain
\begin{multline}
(p^{-1};p^{-1})_\infty^{k-2} \sum_{\la^{(1)},\ldots,\la^{(k-2)} \in \BBY} \left(\sum_{\substack{H_{k-1} \leq \mu^{(k-1)} \op \mu^{(k)}: \\ \pi_1(H_{k-1}) = \mu^{(k-1)}}} \Hom(\la^{(k-2)},H_{k-1})\right) \\
\times \frac{\#\Inj(\la^{(k-3)}, \la^{(k-2)}) \cdots \#\Inj(\la^{(1)},\la^{(2)})}{\#\Aut(\la^{k-2}) \cdots \#\Aut(\la^{(1)})} \#\Sur(\la^{(k-2)},\mu^{(k-2)}) \cdots \#\Sur(\la^{(1)},\mu^{(1)}).
\end{multline}
Continuing to apply Lemma~\ref{lemma:jointsur_iterate} with $L=\la^{(k-i-1)},\kappa=\la^{(k-i)},M=\mu^{(k-i)},N=\mu^{(k-i+1)}$ with $i=2,\ldots,k-1$ (recall $\la^{(0)} := \emptyset$ by convention), we obtain 
$$\text{LHS\eqref{eq:jsthm_wts}} = \sum_{\substack{H_{k-1} \leq \mu^{(k-1)} \op \mu^{(k)}: \\ \pi_1(H_{k-1}) = \mu^{(k-1)}}} \sum_{\substack{H_{k-2} \leq \mu^{(k-2)} \op H_{k-1}: \\ \pi_1(H_{k-2}) = \mu^{(k-2)}}} \cdots \sum_{\substack{H_{1} \leq \mu^{(1)} \op H_2: \\ \pi_1(H_1) = \mu^{(1)}}} 1 = m_k(\mu^{(1)},\ldots,\mu^{(k)}),$$
completing the proof.
\end{proof}

%\begin{remark}\label{rmk:didn't_need_sections}
%There is also another proof of Proposition \ref{prop:jointsur:k:Haar} using our previous computations in the $\alpha$-balanced setting together with results of \cite{VP21limits} previously established by Hall-Littlewood techniques. By taking a limit of \cite[Corollary 3.4]{VP21limits} the joint distribution of $\cok(M_1 \cdots M_j), 1 \leq j \leq k$, for $M_i \in \Mat_n(\BBZ_p)$ with iid additive Haar entries, converges to $\P^{(k)}_{\infty;1/p}$ as $n \to \infty$. By Corollary \ref{cor:jointsur:uniform:asym}, the joint moments of $\cok(M_1 \cdots M_j), 1 \leq j \leq k$ converge to the desired values. Together with certain uniform convergence estimates to commute the limit with the expectation, this yields Proposition \ref{prop:jointsur:k:Haar}.  
%The same method also proves Proposition \ref{prop:surmoment:k:Haar} using Corollary \ref{cor:sur:uniform:asym}. We choose to give the above `exact' proofs, though they may be slightly longer, as their independence from earlier estimates makes it easier to focus on the algebraic skeleton underlying the computation.
%\end{remark}

\section{Moment comparison and the proof of Theorem~\ref{theorem:main:prod}}\label{sect:comparison}

Fix a finite set of primes $P$, and let $Y \sim \P_P^{\ast(k)}$ as defined in Definition \ref{def:intro_measures}. For any $a$ divisible only by primes in $P$, Theorem \ref{theorem:surmoment:k} and Theorem \ref{thm:intro_dist_moment_prod} imply that $Y$ and $\cok(M_1\cdots M_k )$ (in the setting of Theorem \ref{theorem:main:prod}) have asymptotic matching moments with respect to all groups $G$ of exponent dividing $a$. To pass this information back to distribution, we then use the following result on the moment problem for finite abelian groups, a direct analogue of \cite[Theorem 8.3]{W0}, which suffices to prove Theorem~\ref{theorem:main:prod}. 

\begin{theorem}\label{theorem:t:distribution}Let $X_n$ and $Y_n$ be sequences of random finitely generated abelian groups. Let $a$ be a positive integer and $\CA_a$ be the set of isomorphism classes of abelian groups with exponent dividing $a$. Suppose that for every $G\in \CA_a$ we have 
$$\lim_{n \to \infty} \E[ \#\Sur(X_n,G)) = \lim_{n \to \infty}  \E[\#\Sur(Y_n,G)] = n_k(G).$$ Then we have that for every $H \in \CA_a$, 
 $\lim_{n \to \infty} \P\left(X_n \otimes (\Z/a\Z) \isom H\right)$ exists and
 $$\sum_{H\in \CA_a}\lim_{n \to \infty} \P\left(X_n \otimes (\Z/a\Z) \isom H\right) \# \Sur(H,G) =n_k(G).$$
Furthermore,
$$\lim_{n \to \infty} \P\left(X_n \otimes (\Z/a\Z) \isom H\right) =\lim_{n \to \infty} \P(Y_n \otimes (\Z/a\Z) \isom H).$$
\end{theorem}

\begin{proof}[Proof of Theorem~\ref{theorem:main:prod}, assuming Theorem~\ref{theorem:t:distribution}] Assume that the exponent of the group $B$ under consideration has prime factorization $\prod_{p\in P} p^{e_p}$. Theorem \ref{theorem:t:distribution}, applied to the sequence $X_n=\cok(M_1\cdots M_k )$ and $Y_n=Y\sim \P_P^{\ast(k)}$ with $a= \prod_{p\in P} p^{e_p+1}$, implies that  
$$\lim_{n\to \infty} \P\left(\cok(M_1\cdots M_k ) \otimes (\Z/a\Z) \isom B\right) = \P(Y \otimes (\Z/a\Z) \isom B).$$
The proof is then complete because $\cok(M_1\cdots M_k ) \otimes (\Z/a\Z) \isom B$ if and only if $\cok(M_1\cdots M_k )[P] \isom B$.%, and similarly for $Y$.
\end{proof}

We note also that the analogue of Theorem~\ref{theorem:main:prod} over $\Z_p$ holds by the exact same proof, with $P=\{p\}$.

\begin{theorem}\label{theorem:main:prod_Zp} Let $\xi$ be a $\Z_p$-valued random variable which is not constant modulo $p$, and for each $n$ let $M_1,\dots, M_k$ be $k$ independent random matrices with iid $\xi$-distributed entries. Then for any finite abelian $p$-group $B$, %Then for each $p\in P$
%$$\lim_{n\to \infty}\P\left((\cok(M_1\cdots M_k ))[p^\infty] \simeq  B_p \right) = \frac{\left(\prod_{i=1}^\infty 1-p^{-i}\right)^k \#\{0 = G_0 \leq G_1 \leq \}.$$
%Furthermore,
$$\lim_{n\to \infty}\P\left(\cok(M_1\cdots M_k ) \simeq  B\right) =(p^{-1};p^{-1})_\infty^k  \frac{\#\{0 = G_0 \leq G_1 \leq \ldots \le G_k = B\}}{\#\Aut(B)}.$$
\end{theorem}

It remains to prove Theorem~\ref{theorem:t:distribution}. For this one follows the treatment of {\cite[Theorem 8.3]{W0}}, which roughly speaking can be summarized as follows: 
\begin{enumerate}[(i)] 
\item From $\lim_{n \to \infty} \E[ \#\Sur(X_n,G)) = \lim_{n \to \infty}  \E[\#\Sur(Y_n,G)] = S_G$ for all $G\in \CA_a$, under some appropriate condition on the growth of $S_G$ and assuming that $\lim_{n \to \infty} \P(X_n \otimes (\Z/a\Z) \isom H)$ exists, one can show that $\sum_{H\in \CA_a} \lim_{n \to \infty} \P(X_n \otimes (\Z/a\Z) \isom H)   \#\Hom(H,G) = \sum_{H \le G} S_H$ for all $G\in \CA_a$. 
\vskip .1in
\item The latter can be written as a system of equalities involving 
$$\lim_{n \to \infty}   \P\left(X_n \otimes (\Z/a\Z) \isom H\right),  \lim_{n \to \infty}  \P\left(Y_n \otimes (\Z/a\Z) \isom H\right),$$
from which, under appropriate growth condition on $S_G$, one can deduce that these two limits are actually the same as desired (see \cite[Theorem 8.2, 8.3]{W0}, and also Theorem \ref{t:distribution:joint'} below). 
\vskip .1in
\item Lastly, one can show the limits $\lim_{n \to \infty} \P(X_n \otimes (\Z/a\Z) \isom H)$ exist for all $H\in \CA_a$ by contradiction, passing to subsequences where the limits exist for all $H$ and use (i) and (ii) above. 
\end{enumerate}
To our current situation, we just need to guarantee that the growth of $S_G=n_k(G)$ for each $p$-group $G$ of type $\la$ is appropriate so that we can apply (i)-(iii) outlined above. It is worth noting for the reader that in this section we refer more details to \cite{W0}, while in the next section we give a more self-contained argument because the extensions of \cite{W0} to the setting of joint moments are more nontrivial. We require the following strengthened (but more cumbersome to state) version of \cite[Theorem 8.2]{W0}, and recall that $\BBY_n$ denotes the set of partitions with at most $n$ parts.

\begin{proposition}\label{prop:w8.2_adapted}
Let $p_1,\ldots,p_s$ be distinct primes. Let $m_1,\ldots,m_s \geq 1$ be integers, let $M = \BBY_{m_1} \times \cdots \times \BBY_{m_s}$, and write elements of $M$ as $(\mu^1,\ldots,\mu^s)$, where $\mu^j$ has parts $\mu_1^j \geq \cdots \geq \mu_{m_j}^j$. Let $x_\mu,y_\mu$ be nonnegative reals for every $\mu \in M$, such that for every $\la \in M$,
\begin{equation}\label{eq:w7}
\sum_{\mu \in M} x_\mu \prod_{j=1}^s p_j^{\sum_i \la_i^j \mu_i^j} = \sum_{\mu \in M} y_\mu \prod_{j=1}^s p_j^{\sum_i \la_i^j \mu_i^j} =: C_\la
\end{equation}
for some nonnegative reals $C_\la$. Suppose these satisfy a bound of the form
\begin{equation}\label{eq:C_bound}
C_\la \leq \prod_{j=1}^s f_{p_j,m_j}(\la^j),
\end{equation}
for any collection of functions $f_{p,m}: \BBY_{m} \to \R_{\geq 0}$ with the property that for every $b \geq 0$ and $d_2,\ldots,d_m$ with $d_2+\ldots+d_m \leq b$, the sum
\begin{equation}\label{eq:w_better_bound}
\sum_{d_1 \geq 0} p^{-b d_1 - \frac{d_1(d_1+1)}{2}} f_{p,m}(d_1+d_2+\ldots+d_{m},d_2+\ldots+d_m,\ldots,d_m)
\end{equation}
converges. Then $x_\mu = y_\mu$ for all $\mu \in M$.
\end{proposition}
\begin{proof}
This is proven in \cite[Theorem 8.2]{W0} in the special case $f_{p,m}(\la) = F^m p^{\sum_{i=1}^m \frac{\la_i(\la_i-1)}{2}}$, but the only property of this bound which is needed in the proof is the convergence of the sum \eqref{eq:w_better_bound}, hence the result holds in our more general setup.
\end{proof}

%More precisely, it suffices to show that it is significantly smaller than $p^{\frac{1}{2} \sum_i (\la_i')^2}$ so that we can apply  \cite[Theorem 8.2]{W0} to show that the moments determine a unique distribution and moment convergence implies convergence in probability. 
By viewing the conjugate partitions $(\la^j)'$ of each partition $(\la^1,\ldots,\la^{s}) \in M$ as specifying an abelian $p_j$-group of exponent dividing $p_j^{m_j}$, we see that $M$ is in bijection with $\CA_a$ where $a=\prod_{j=1}^s p_j^{m_j}$. In applications of the above proposition, $x_\mu,y_\mu$ will be the limiting probabilities that certain random elements of $\CA_a$ has isomorphism type specified by $\mu$ in this manner, and $C_\la$ are the so-called Hom-moments $\E_K[\#\Hom(K,G_\la)]$. In order to get good bounds on the Hom-moments in our setup, we first consider the usual (Sur-)moments. We require the following estimate, which is \cite[Lemma 7.4]{W0}. 

\begin{lemma}\label{lemma:sbgp_num_bound}
For $G_{\mu,\la}$ as in Definition \ref{def:sbgp_number}, 
$$|G_{\mu,\la}| \leq \frac{1}{\left(\prod_{i \geq 1} (1-2^{-i})\right)^{\la_1}} p^{\sum_{i=1}^{\la_1} \mu_i' \la_i' - (\mu_i')^2}.$$
\end{lemma}

\begin{lemma}\label{lemma:growth:n_k} There exist positive constants $F_k$ and $0<c_k<1$ such that for any $p$,
$$n_k(G_\la) \le F_k^{\la_1} p^{\frac{1-c_k}{2} \sum_i (\la_i')^2}.$$
\end{lemma}

\begin{proof} We will induct on $k$. We first show
\begin{equation}
n_2(G) \le F^{\la_1} p^{\sum_i (\la_i')^2/4}.
\end{equation}
Indeed, using the notations from \cite[Section 7]{W0}, with $C= \prod_{i\ge 1} (1-2^{-i})$,
\begin{align*}
n_2(G) &= \sum_{\mu} |G_{\mu,\la}| \\
&\le \frac{1}{C^{\la_1}} \sum_{\mu, \mu_1 \le \la_1} p^{\sum_i \mu_i' \la_i' - (\mu_i')^2}\\
&\le  \frac{1}{C^{\la_1}} \sum_{d_1,\ldots,d_{\la_1} \geq 0} p^{\sum_{i=1}^{\la_1}  d_i \la_i' - d_i^2} \\ 
&=\frac{p^{\sum_i (\la_i')^2/4}}{C^{\la_1}}  \prod_{i=1}^{\la_1} \sum_{d_i \geq 0} p^{-(\la_i'/2 - d_i)^2} \\ 
&\le F_2^{\la_1} p^{\sum_i (\la_i')^2/4},
\end{align*}

where the first inequality is Lemma \ref{lemma:sbgp_num_bound}. In the last inequality, we are using the fact that for any $\la_i$,
$$\sum_{d \geq 0} p^{-(\la_i'/2 - d)^2} \leq \max\left(\sum_{d \in \Z} p^{-d^2},\sum_{d \in \Z+\frac{1}{2}} p^{-d^2}\right) < \infty,$$
and absorbing this bound into the constant $F_2$.

For induction, assume that for some $c_{j-1}>0$
$$n_{j-1}(G) \le F_{j-1}^{\la_1} p^{\sum_i (1-c_{j-1})(\la_i')^2/2}.$$
We will show that then 
$$n_j(G) \le F_j^{\la_1} p^{\sum_i (1-c_j)(\la_i')^2/2}$$
where 
$$c_j = 1 -\frac{1}{(1+c_{j-1})}>0.$$
To see this, we proceed exactly as before:
\begin{align*}
n_j(G) &= \sum_{\mu} |G_{\mu,\la}| n_{j-1}(\mu) \\
&\le \frac{1}{C^{\la_1}} \sum_{\mu, \mu_1 \le \la_1} p^{\sum_i \mu_i' \la_i' - (\mu_i')^2} F_{j-1}^{\la_1} p^{(1-c_{j-1})\sum_i (\mu_i')^2/2} \\
&=  \frac{F_{j-1}^{\la_1}}{C^{\la_1}} \sum_{d_1,\dots, d_{\la_1}} p^{\sum_{i=1}^{\la_1}  d_i \la_i' - (1+c_{j-1})d_i^2/2} \\
&=\frac{F_{j-1}^{\la_1}p^{\sum_i \frac{1}{2(1+c_{j-1})}(\la_i')^2}}{C^{\la_1}}  \sum_{d_1,\dots, d_{\la_1} \ge 0} p^{\sum_{i=1}^{\la_1}  -(\frac{1}{\sqrt{1+c_{j-1}}}\la_i' - \sqrt{1+c_{j-1}}d_i)^2/2 } \\
&\le F_j^{\la_1} p^{\sum_i \frac{1}{2(1+c_{j-1})}(\la_i')^2}.
\end{align*}

Here again we use Lemma \ref{lemma:sbgp_num_bound} in the first inequality and essentially the same bound in the last. So we can take 
$$c_j = 1 -\frac{1}{(1+c_{j-1})},$$
completing the proof.
\end{proof}

We now explain in more detail how to adapt the proof in \cite{W0} to our setting.

\begin{proof}[Proof of Theorem \ref{theorem:t:distribution}]
The proof uses \cite[Theorem 8.2]{W0}, and is exactly the same as the proof of \cite[Theorem 8.3]{W0} after substituting $|\wedge^2 G|$ in \cite{W0} for $n_k(G)$. The role of the convergence-of-moments statement \cite[Theorem 1.2]{W0} in that proof is played in our argument by Theorem \ref{theorem:surmoment:k}. The only ingredient of that proof which we are still missing is a bound \eqref{eq:C_bound} on the Hom-moments which satisfies the hypothesis \eqref{eq:w_better_bound} in Proposition \ref{prop:w8.2_adapted}; in \cite{W0} this role is played by \cite[Lemma 7.5]{W0}. In our setting we must show that for any $p$ and $m \geq 1$, one has a bound on Hom-moments
$$\sum_{H \leq G_{\la'}} n_k(H) \leq f_{p,m}(\la)$$
for all $\la \in \BBY_m$, where $f_{p,m}$ satisfies \eqref{eq:w_better_bound}. The LHS is $n_{k+1}(G_{\la'})$, which by Lemma \ref{lemma:growth:n_k} is bounded above by 
$$f_{p,m}(\la) = F_{k+1}^{\la_1'} p^{\frac{1-c_{k+1}}{2} \sum_i (\la_i)^2}.$$
Hence the summand in \eqref{eq:w_better_bound} is of the form 
$$p^{-\frac{c_{k+1}}{2}d_1^2 + \text{const} \cdot d_1},$$
so the sum converges. The remainder of the proof is identical to that of \cite[Theorem 8.3]{W0}.
\end{proof}

 \section{Joint moments comparison and the proof of Theorem~\ref{theorem:main:joint}}\label{sect:comparison:2} 

 The main goal of this section is the following analog of Theorem~\ref{theorem:t:distribution}, which informally says that if the limits of the joint moments are the same and not too large, then the joint distributions must be asymptotically the same.

 \begin{theorem}\label{theorem:t:distribution:joint}For each $n \geq 1$ let $(X_n^{(1)},\dots, X_n^{(k)})$ and $(Y_n^{(1)},\dots, Y_n^{(k)})$ be two sequences of random finitely generated abelian groups. Let $a$ be a positive integer and $\CA_a$ be the set of isomorphism classes of abelian groups with exponent dividing $a$. Suppose that for every $G_1,\dots, G_k\in \CA_a$ we have 
 \begin{equation}\label{eq:limit_joint_moment_hyp}
\lim_{n \to \infty} \E\left[ \prod_{i=1}^k \#\Sur(X_n^{(i)},G_i)\right] = \lim_{n \to \infty}  \E\left[\prod_{i=1}^k\#\Sur(Y_n^{(i)},G_i)\right] = m_k(G_1,\ldots,G_k).
\end{equation}
 Then we have that for every $H_1,\dots, H_k \in \CA_a$, $\lim_{n \to \infty} \P(X_n^{(1)} \otimes (\Z/a\Z) \isom H_1 \wedge  \dots \wedge X_n^{(k)} \otimes (\Z/a\Z) \isom H_k)$ exists and 
$$\sum_{H_i \in \CA_a} \lim_{n \to \infty} \P(X_n^{(1)} \otimes (\Z/a\Z) \isom H_1 \wedge \dots \wedge X_n^{(k)} \otimes (\Z/a\Z) \isom H_k) \prod_{i=1}^k \#\Sur(X_n^{(i)},G_i) =  m_k(G_1,\ldots,G_k).$$ 
Furthermore, 
\begin{align}\label{eq:joint_limit_determined}
\begin{split}
&\lim_{n \to \infty} \P\left(X_n^{(1)} \otimes (\Z/a\Z) \isom H_1 \wedge \dots \wedge X_n^{(k)} \otimes (\Z/a\Z) \isom H_k\right) \\
&=\lim_{n \to \infty} \P(Y_n^{(1)} \otimes (\Z/a\Z) \isom H_1 \wedge \dots \wedge  Y_n^{(k)} \otimes (\Z/a\Z) \isom H_k).
\end{split}
\end{align}
\end{theorem}

To complete the proof of Theorem~\ref{theorem:main:joint} one just needs to combine the above result together with Theorem \ref{theorem:jointsur:k} and Theorem \ref{thm:intro_dist_moment_joint}.

\begin{proof}[Proof of Theorem \ref{theorem:main:joint}, assuming Theorem \ref{theorem:t:distribution:joint}]
To show that for any $P$ the random groups $\cok(M_1 \cdots M_j)[P], 1 \leq j \leq k$ converge in distribution, it suffices to show for each positive integer $a$ that $\cok(M_1 \cdots M_j) \otimes \BBZ/a\BBZ, 1 \leq j \leq k$ converge in distribution. Theorem \ref{theorem:jointsur:k} shows that under the assumptions of Theorem \ref{theorem:main:joint}, the joint moments of $\cok(M_1 \cdots M_j) \otimes \BBZ/a\BBZ, 1 \leq j \leq k$ converge to $m_k(G_1,\ldots,G_k)$ for any groups $G_1,\ldots,G_k$ with exponent dividing $a$. Theorem \ref{thm:intro_dist_moment_joint} shows similarly that the joint moments of $(Y^{(1)},\ldots,Y^{(k)}) \sim \P_P^{(k)}$ are $m_k(G_1,\ldots,G_k)$ for all such $G_1,\ldots,G_k$. Theorem \ref{theorem:t:distribution:joint} with $X_n^{(j)} = \cok(M_1 \cdots M_j)[P]$ and $(Y_n^{(1)},\dots, Y_n^{(k)}) \sim \P_P^{(k)}$ then shows that the matching of these moments implies convergence of the joint distribution of $\cok(M_1 \cdots M_j) \otimes \BBZ/a\BBZ, 1 \leq j \leq k$ to $\P_P^{(k)}$, completing the proof.

%It remains to show that the limiting joint distribution of $\cok(M_1 \cdots M_j) \otimes \BBZ/a\BBZ, 1 \leq j \leq k$ is as given in Theorem \ref{theorem:main:joint}. For the first part, the case of a single prime, this is exactly Proposition \ref{prop:jointsur:k:Haar}. For the general case, it follows since the moments of $\prod_{p \in P} \P_{\infty;1/p}^{(k)}$ factorize.
\end{proof}

As before, the exact same proof above with $a$ a power of $p$ shows the following $p$-adic analogue: 

\begin{theorem}\label{theorem:main_joint_Zp}
For matrices $M_i \in \Mat_n(\Z_p)$ under the same assumptions as in Theorem \ref{theorem:main:prod_Zp} and finite abelian $p$-groups $B_1, \ldots, B_k$, one has
\begin{equation}\label{eq:limit_joint_Zp}
\lim_{n\to \infty}\P\left(\cok(M_1 \dots M_j)\simeq  B_j, 1\le j\le k\right)=(p^{-1};p^{-1})_\infty^k \prod_{i=1}^k \frac{\#\Sur(B_i, B_{i-1})}{\#\Aut(B_i)},
\end{equation}
where we take $B_0 = 0$.
\end{theorem}

It remains to justify Theorem \ref{theorem:t:distribution:joint}. For this one follows the three steps (i)-(iii) outlined above in the proof of \cite[Theorem 8.3]{W0}, which we do now.%. A close investigation shows that the only place we need extra treatment is in (ii), which is our main focus below. 

\begin{theorem}\label{t:distribution:joint'}  Let $p_1,\dots,p_s$ be distinct primes. Let $m_1,\dots,m_s \ge 1$ be integers. Let $M_j$ be the set of partitions $\la$ with at most $m_j$ parts. Let $M = M_1 \times \cdots \times M_s$. For $\mu \in M$, we write $\mu^j$ for its $j$th entry, which is a partition consisting of non-negative integers $\mu^j_i$ with $\mu^j_1 \ge \dots \ge \mu^j_{m_j}$. Suppose we have non-negative reals $x_{\mu(1),\dots,\mu(k)} , y_{\mu(1),\dots, \mu(k)}$, for each $k$-tuple of sequences of partitions $\mu(i) \in M$. Further suppose that for all $\la(1),\dots, \la(k) \in M$,
\begin{align}\label{eq:joint_hom_mmts_equal}
\begin{split}
\sum_{\mu(1),\dots,\mu(k) \in M} x_{\mu(1),\dots, \mu(k)} \prod_{j=1}^s\prod_{\ell=1}^k p_j^{\sum_i \la(\ell)_i^j \mu(\ell)_i^j}  &= \sum_{\mu(1),\dots,\mu(k) \in M} y_{\mu(1),\dots, \mu(k)} \prod_{j=1}^s\prod_{\ell=1}^k p_j^{\sum_i \la(\ell)_i^j \mu(\ell)_i^j} \\
&=C_{\la(1),\dots,\la(k)},
\end{split}
\end{align}
satisfies the growth condition
\begin{equation}\label{eq:c_growth}
C_{\la(1), \dots,\la(k)} \le  \prod_{j=1}^s F_k^{m_j} p_j^{\frac{1-c_k}{2}\sum_i (\sum_{\ell=1}^k {\la(\ell)^{j}_i})^2},
\end{equation}
where $F_k, 0<c_k<1$ are constants depending on $k$.
Then for all $\mu(1),\dots, \mu(k) \in M$ we have 
$$x_{\mu(1),\dots,\mu(k)} = y_{\mu(1),\dots, \mu(k)}.$$
\end{theorem}

\begin{proof}[Proof of Theorem \ref{t:distribution:joint'}]

We proceed by induction. The base case $k=1$ follows from Proposition \ref{prop:w8.2_adapted} since the function 
$$f_{p_j,m_j}(\la(1)) = F_1^{m_j} p_j^{\frac{1-c_1}{2}\sum_i ( {\la(1)^{j}_i})^2}$$
appearing in the bound \eqref{eq:c_growth} satisfies the hypothesis \eqref{eq:w_better_bound} of that result. Hence we suppose Theorem \ref{t:distribution:joint'} holds for $k-1$ and verify it for $k$. By hypothesis, \eqref{eq:joint_hom_mmts_equal} holds, so 
\begin{align}\label{eq:telescope_xy}
\begin{split}
&\sum_{\mu(k) \in M} \prod_{j=1}^s p_j^{\sum_i \la(k)_i^j \mu(k)_i^j}\left(\sum_{\mu(1),\ldots,\mu(k-1) \in M} \prod_{j=1}^s\prod_{\ell=1}^{k-1} p_j^{\sum_i \la(\ell)_i^j \mu(\ell)_i^j}  x_{\mu(1),\dots, \mu(k)}\right)  \\
&=\sum_{\mu(k) \in M} \prod_{j=1}^s p_j^{\sum_i \la(k)_i^j \mu(k)_i^j}\left(\sum_{\mu(1),\ldots,\mu(k-1) \in M} \prod_{j=1}^s\prod_{\ell=1}^{k-1} p_j^{\sum_i \la(\ell)_i^j \mu(\ell)_i^j}  y_{\mu(1),\dots, \mu(k)}\right)  \\
&= C_{\la(1),\ldots,\la(k)}.
\end{split}
\end{align}
We now apply Proposition \ref{prop:w8.2_adapted} to 
\begin{align*}
\tilde{x}_{\mu(k)} &:= \sum_{\mu(1),\ldots,\mu(k-1) \in M} \prod_{j=1}^s\prod_{\ell=1}^{k-1} p_j^{\sum_i \la(\ell)_i^j \mu(\ell)_i^j}  x_{\mu(1),\dots, \mu(k)} \\
\tilde{y}_{\mu(k)} &:= \sum_{\mu(1),\ldots,\mu(k-1) \in M} \prod_{j=1}^s\prod_{\ell=1}^{k-1} p_j^{\sum_i \la(\ell)_i^j \mu(\ell)_i^j}  y_{\mu(1),\dots, \mu(k)} \\
\tilde{C}_{\la(k)} &:= C_{\la(1),\ldots,\la(k)}.
\end{align*}
By using the hypothesis \eqref{eq:c_growth} to bound $\tilde{C}_{\la(k)}$ as a function of $\la(k)$, it is easy to check that \eqref{eq:w_better_bound} is satisfied because of the factor $(1-c_k)/2$, as in the proof of Proposition \ref{prop:w8.2_adapted}. This yields that 
$$\tilde{x}_{\mu(k)} = \tilde{y}_{\mu(k)},$$
and applying the $k-1$ case of Theorem \ref{t:distribution:joint'} to the sums defining $\tilde{x}_{\mu(k)}$ and $\tilde{y}_{\mu(k)}$ completes the proof.
\end{proof}

To apply Theorem \ref{t:distribution:joint'} to our situation, we need the following moment bound.

\begin{claim} \label{claim:mk_bound} Assume that $G_1,\dots, G_k$ are $p$-groups corresponding to $\la(1), \dots, \la(k) \in \BBY$. Then there exist absolute constants $F_k>0$ and $0<c_k<1$ such that 
$$n_k(G_1 \oplus \dots \oplus G_k) \le F_k^{\sum_{\ell=1}^k \la(\ell)_1} p^{\frac{1-c_k}{2}\sum_i (\sum_{\ell=1}^k \la(\ell)'_i)^2},$$
where $\la(\ell)'$ is the conjugate partition of $\la(\ell)$.
\end{claim}

\begin{proof} Note that the group $G_1 \oplus \dots \oplus G_k$ corresponds to the union of $\{\la(1), \dots, \la(k)\}$. The conjugate parts of this partition are given by $\{\sum_\ell {\la(\ell)_i}': i \geq 1\}$. We then use Lemma \ref{lemma:growth:n_k} to obtain
$$n_k(G_1 \oplus \dots \oplus G_k) \le \widetilde{F}_k^{\max_{1 \leq \ell \leq k} \la(\ell)_1} p^{\frac{1-c_k}{2}\sum_i (\sum_{\ell=1}^k \la(\ell)'_i)^2},$$
and the claim follows by bounding the maximum above by the sum and letting $F_k = \max(\widetilde{F}_k,1)$.
\end{proof}

\begin{proof}[Proof of Theorem \ref{theorem:t:distribution:joint}]
We closely follow the proof of \cite[Theorem 8.3]{W0}. Let us first suppose that the limits 
$$\lim_{n \to \infty} \P(X_n^{(1)} \ot \Z/a\Z \simeq H_1 \wedge \cdots \wedge X_n^{(k)} \ot \Z/a\Z \simeq H_k)$$
exist, and show that 
\begin{multline}\label{eq:value_limit_sur}
%\begin{split}
\sum_{H_1,\ldots,H_k \in \CA_a} \lim_{n \to \infty}\P(X_n^{(1)} \ot \Z/a\Z \simeq H_1 \wedge \cdots \wedge X_n^{(k)} \ot \Z/a\Z \simeq H_k) \#\Sur(H_1,G_1) \times \dots \times \#\Sur(H_k,G_k) \\ = m_k(G_1,\ldots,G_k).
%\end{split}
\end{multline}
This would follow from the hypotheses of the theorem if we could interchange the sum and limit, which we now argue. First note that since $\#\Hom(H,G) = \sum_{K \leq G} \#\Sur(H,K)$, by taking finite linear combinations it suffices to show the statement with surjections replaced by homomorphisms,
\begin{align}\label{eq:value_limit_hom}
\begin{split}
&\sum_{H_1,\ldots,H_k \in \CA_a} \lim_{n \to \infty}\P(X_n^{(1)} \ot \Z/a\Z \simeq H_1 \wedge \cdots \wedge X_n^{(k)} \ot \Z/a\Z \simeq H_k) \#\Hom(H_1,G_1) \times \dots \times \#\Hom(H_k,G_k) \\ 
&=\lim_{n \to \infty} \sum_{H_1,\ldots,H_k \in \CA_a} \P(X_n^{(1)} \ot \Z/a\Z \simeq H_1 \wedge \cdots \wedge X_n^{(k)} \ot \Z/a\Z \simeq H_k) \#\Hom(H_1,G_1) \times \dots \times \#\Hom(H_k,G_k).
\end{split}
\end{align}
For each $(G_1,\ldots,G_k) \in \CA_a^k$, we claim there exists $(G_1',\ldots,G_k') \in \CA_a^k$ such that 
\begin{equation}\label{eq:hom_quot_conv}
\sum_{(H_1,\ldots,H_k) \in \CA_a^k} \frac{\prod_{i=1}^k \#\Hom(H_i,G_i)}{\prod_{i=1}^k \#\Hom(H_i,G_i')}
\end{equation}
converges. By factoring the sum into a product of $k$ sums and factoring each one over the primes $p_j$ dividing $a$, it suffices to show when $a=p^e$ is a prime power that for any $G \in \CA_a$ there exists $G' \in \CA_a$ such that 
$$\sum_{H \in \CA_a} \frac{\#\Hom(H,G)}{\#\Hom(H,G')}$$
converges. This follows as in \cite{W0} by letting $\la$ be the type of $G$ and taking $G'$ to have type $\pi$ with $\pi_i' = 2\la_i'+1$ for $1 \leq i \leq e$. By \cite[Lemma 7.1]{W0}, 
\begin{equation}\label{eq:mel_hom_count}
\#\Hom(G_\mu,G_\la) = p^{\sum_i \mu_i' \la_i'},
\end{equation}
and the above convergence (and hence convergence of \eqref{eq:hom_quot_conv}) follows by a simple computation. 

%Since 
%$$\P(X_n^{(1)} \ot \Z/a\Z \simeq H_1 \wedge \cdots \wedge X_n^{(k)} \ot \Z/a\Z \simeq H_k) \#\Sur(H_1,G_1) \times \dots \times \#\Sur(H_k,G_k) )$$
%converges by hypothesis
Since 
\begin{align}\label{eq:E_hom_sur}
\begin{split}
&\E[\#\Hom(X_n^{(1)},G'_1) \times \dots \times \#\Hom(X_n^{(k)},G'_k)] \\
&= \sum_{K_1 \leq G'_1,\ldots,K_k \leq G'_k} \E[\#\Sur(X_n^{(1)},K_1) \times \dots \times \#\Sur(X_n^{(k)},K_k)]
\end{split}
\end{align}
converges by hypothesis, the LHS is bounded above in $n$. Hence there exists a constant $D_{G_1,\ldots,G_k}$ such that 
\begin{equation}\label{eq:D_bound}
\P(X_n^{(1)} \ot \Z/a\Z \simeq H_1 \wedge \cdots \wedge X_n^{(k)} \ot \Z/a\Z \simeq H_k) \#\Hom(H_1,G'_1) \times \dots \times \#\Hom(H_k,G'_k) \leq D_{G_1,\ldots,G_k}
\end{equation}
for all $n$, since the LHS of \eqref{eq:D_bound} is clearly bounded above by \eqref{eq:E_hom_sur}. Therefore the function $f_n: \CA_a^k \to \BBR_{\geq 0}$ given by 
$$f_n(H_1,\ldots,H_k) = \P(X_n^{(1)} \ot \Z/a\Z \simeq H_1 \wedge \cdots \wedge X_n^{(k)} \ot \Z/a\Z \simeq H_k) \#\Hom(H_1,G_1) \times \dots \times \#\Hom(H_k,G_k)$$
is bounded above by the function 
$$g(H_1,\ldots,H_k) = D_{G_1,\ldots,G_k}  \frac{\#\Hom(H_1,G_1) \times \dots \times \#\Hom(H_k,G_k)}{\#\Hom(H_1,G'_1) \times \dots \times \#\Hom(H_k,G'_k)}.$$
Since
$$\sum_{(H_1,\ldots,H_k) \in \CA_a^k} g(H_1,\ldots,H_k) < \infty$$ 
by the convergence of \eqref{eq:hom_quot_conv}, and $g$ dominates $f_n$ by \eqref{eq:D_bound}, \eqref{eq:value_limit_hom} follows by the Lebesgue dominated convergence theorem. This in turn shows \eqref{eq:value_limit_sur} as mentioned above.

We now show \eqref{eq:joint_limit_determined}, still assuming without proof that the limits in that equation exist. Let $a = \prod_{j=1}^s p_j^{m_j}$ be the prime factorization of $a$, and let $M$ be as in Theorem \ref{t:distribution:joint'} so that $M$ is in bijection with $\CA_a$ as in the previous section. Write $H_\mu \in \CA_a$ to be the element corresponding to $\mu \in M$. Note this differs from the notation of $G_\mu$ in that $H_\mu$ is not a $p$-group, and that even when $s=1$ so $H_\mu$ is a $p$-group we still have $H_\mu = G_{\mu'}$ with the conjugate partition. Then \eqref{eq:joint_limit_determined} is the statement that the limiting probabilities 
\begin{align*}
x_{\mu(1),\ldots,\mu(k)} &:= \lim_{n \to \infty} \P\left(X_n^{(1)} \otimes (\Z/a\Z) \isom H_{\mu(1)} \wedge \dots \wedge X_n^{(k)} \otimes (\Z/a\Z) \isom H_{\mu(k)}\right) \\
y_{\mu(1),\ldots,\mu(k)} &:= \lim_{n \to \infty} \P(Y_n^{(1)} \otimes (\Z/a\Z) \isom H_{\mu(1)} \wedge \dots \wedge  Y_n^{(k)} \otimes (\Z/a\Z) \isom H_{\mu(k)})
\end{align*}
 are equal. For each $1 \leq \ell \leq k$, let $G_\ell = H_{\la(i)}$ where $\la(\ell) = (\la(\ell)^1,\ldots,\la(\ell)^s) \in M$. By \eqref{eq:value_limit_hom} and \eqref{eq:mel_hom_count},
\begin{align*}%\label{eq:hom_sum_limits}
\sum_{(\mu(1),\ldots,\mu(k)) \in M^k} x_{\mu(1),\ldots,\mu(k)} \prod_{i=1}^k \prod_{j=1}^s p_j^{\sum_i \mu(\ell)^j_i \la(\ell)^j_i} %\#\Hom(H_{\mu(1)},G_1) \times \dots \times \#\Hom(H_{\mu(k)},G_k)) 
&= \sum_{K_1 \leq G_1,\ldots,K_k \leq G_k} m_k(K_1,\ldots,K_k)\\
&=\sum_{(\mu(1),\ldots,\mu(k)) \in M^k} y_{\mu(1),\ldots,\mu(k)}  \prod_{\ell=1}^k \prod_{j=1}^s p_j^{\sum_i \mu(\ell)^j_i \la(\ell)^j_i}.
\end{align*}
%By the two inequalities of Claim \ref{claim:mk_bound} applied separately, 
We bound the middle term above by
\begin{align}\label{eq:hom_mmt_bd_inproof}
\begin{split}
\sum_{K_1 \leq G_1,\ldots,K_k \leq G_k} m_k(K_1,\ldots,K_k) &\leq \sum_{K_1 \leq G_1,\ldots,K_k \leq G_k} n_k(K_1 \oplus \cdots \oplus K_k) \\
&\leq  n_{k+1}(G_1 \oplus \cdots \oplus G_k) \\
& = \prod_{j=1}^s n_{k+1}((G_1)[p_j^\infty] \oplus \cdots \oplus (G_k)[p_j^\infty]) \\
& \leq \prod_{j=1}^s F_{k+1}^{\sum_{\ell=1}^{k+1} \la(\ell)^j_1} p_j^{\frac{1-c_{k+1}}{2}\sum_i (\sum_{\ell=1}^{k+1} {\la(\ell)^j_i}')^2},
\end{split}
\end{align}
 and we used \eqref{eq:m_n_bound} and Claim \ref{claim:mk_bound} in the first and last line respectively. Letting 
$$C_{\la(1),\ldots,\la(k)} := \sum_{K_1 \leq G_1,\ldots,K_k \leq G_k} m_k(K_1,\ldots,K_k),$$
Theorem \ref{t:distribution:joint'} applies and hence 
$$x_{\mu(1),\ldots,\mu(k)} = y_{\mu(1),\ldots,\mu(k)}$$
for every $\mu(1),\ldots,\mu(k) \in M$, showing \eqref{eq:joint_limit_determined}.

The last thing to show is our initial supposition that the limits 
\begin{equation}\label{eq:limits_exist}
\lim_{n \to \infty} \P(X_n^{(1)} \ot \Z/a\Z \simeq H_1 \wedge \cdots \wedge X_n^{(k)} \ot \Z/a\Z \simeq H_k)
\end{equation}
(and similarly for $Y_n^{(i)}$) exist. Suppose for the sake of contradiction that there exist $G_1,\ldots,G_k \in \CA_a$ for which this is not true. Then by a diagonalization argument we can find two different subsequences of $(i_n)_{n \geq 1},(j_n)_{n \geq 1}$ such that the limits $\P(X_{i_n}^{(1)} \ot \Z/a\Z \simeq H_1 \wedge \cdots \wedge X_{i_n}^{(k)} \ot \Z/a\Z \simeq H_k)$ (and similarly for $j_n$) exist but are different. The above argument shows that these limiting probabilities are uniquely determined by their joint moments, which are the same for both subsequences by \eqref{eq:limit_joint_moment_hyp}. This is a contradiction, hence the limits \eqref{eq:limits_exist} exist, completing the proof.
\end{proof}

\section{An automorphism interpretation of the joint cokernel distribution} \label{sec:automorphisms}

The goal of this section, encapsulated in Theorem \ref{theorem:seq_aut_intro}, is to give an interpretation of the universal limiting distribution of Theorem \ref{theorem:main:joint} in terms of automorphisms of appropriate objects. For simplicity we consider a single prime for most of this section, so let us fix a prime $p$.

As motivation for the measure introduced below, note that for matrices $M_1,\ldots,M_k \in \Mat_n(\BBZ_p)$, the sequence of groups $\cok(M_1),\cok(M_1M_2),\ldots,\cok(M_1 \cdots M_k)$ comes with additional structure. Namely, for each $1 \leq i \leq k-1$, there is a surjection
$$\cok(M_1 \cdots M_{i+1} ) = \BBZ_p^n/\Im(  M_1 \cdots M_{i+1}) \surj \BBZ_p^n/\Im( M_1 \cdots M_{i} ) = \cok(M_1 \cdots M_i)$$
induced by the natural inclusion $\Im( M_1 \cdots M_{i+1}) \subset \Im(M_1 \cdots M_{i})$. Our tools do not currently allow us to prove universality of this random sequence 
$$\cok(M_1 \cdots M_k) \surj \ldots \surj \cok(M_1)$$
of abelian $p$-groups with maps between them, but the perspective of this extra data is nonetheless useful for interpreting the limit distribution on the isomorphism types of this sequence of groups.

\begin{definition}\label{def:k-sequence}
We refer to a collection $G_k \xsurj{\phi_{k-1}} \ldots \xsurj{\phi_1} G_1$ of $k$ groups together with surjections $\phi_i: G_{i+1} \surj G_{i}$ between them as a \emph{$k$-sequence of groups} or simply $k$-sequence. For a fixed sequence of groups\footnote{We must be careful to specify \emph{groups} (as in, a set with a group operation on it), not just groups up to isomorphism, in order to define the set $\KK$.} $G_1,\ldots,G_k$, let $\KK(G_1,\ldots,G_k)$ denote the set of such $k$-sequences.%, and will often not write explicit variables $\phi_i$ in the notation even though their data is included in the definition. 
\end{definition}

\begin{definition}\label{def:k-seq_equivalence}
An isomorphism between two $k$-sequences $(G_k \xsurj{\phi_{k-1}} \ldots \xsurj{\phi_1} G_1)$, $ (H_k \xsurj{\psi_{k-1}} \ldots \xsurj{\psi_1} H_1)$ is a sequence of group isomorphisms $\varphi_1,\ldots,\varphi_k$ such that the diagram 
\begin{center}
\begin{tikzcd}
G_k \arrow[r, twoheadrightarrow,"\phi_{k-1}"] \arrow[d, "\varphi_k"]
& \cdots \arrow[r, twoheadrightarrow,"\phi_{1}"] & G_1 \arrow[d,"\varphi_1"] \\
H_k \arrow[r, twoheadrightarrow, "\psi_{k-1}" ] & \cdots \arrow[r, twoheadrightarrow,"\psi_{1}"] & H_1
\end{tikzcd}
\end{center}
commutes. 
We further write
$$(G_k \xsurj{\phi_{k-1}} \ldots \xsurj{\phi_1} G_1) \simeq (H_k \xsurj{\psi_{k-1}} \ldots \xsurj{\psi_1} H_1)$$
if such an isomorphism exists, and write
 $$\Aut\left(G_k \xsurj{\phi_{k-1}} \ldots \xsurj{\phi_1} G_1\right)$$
for the group of isomorphisms from a given $k$-sequence to itself.
\end{definition}

\begin{remark}\label{rmk:grp_extn_data}
Two $k$-sequences with the same $G_1,\ldots,G_k$ may still not be isomorphic as $k$-sequences. In fact, $G \xsurj{\phi} H$ and $G' \xsurj{\phi'} H'$ may not be isomorphic as $2$-sequences even if $G \simeq G',H \simeq H',$ and $\ker \phi \simeq \ker \phi'$ as groups. An example due to \cite{S} is the following: let $G =G'= \Z/p^3\Z \oplus \Z/p^2\Z \oplus \Z/p\Z$ and $H =H'= \Z/p^2\Z \oplus \Z/p\Z$, and let $\phi$ be the direct sum of maps 
\begin{align*}
\Z/p^3\Z &\to \Z/p^2\Z \\%\text{ (quotient by $p\Z/p^2\Z$)} \\
\Z/p^2\Z &\to 0 \\
\Z/p\Z &\to \Z/p\Z %\text{  (identity)}
\end{align*}
and $\phi'$ be the direct sum of maps 
\begin{align*}
\Z/p^3\Z & \to \Z/p\Z \\ 
\Z/p^2\Z & \to \Z/p^2 \Z \\ 
\Z/p\Z & \to 0
\end{align*}
where all nontrivial maps are quotient or identity maps. Then $\phi$ and $\phi'$ are not equivalent, and to check this it suffices to observe that $p \ker \phi \cap p^2G = \{0\}$ while $p \ker \phi' = p^2 G$, as these two facts are unchanged by any $2$-sequence automorphism.
\end{remark}

Clearly, $\simeq$ is an equivalence relation on $k$-sequences. 

\begin{definition}
Given fixed groups $G_1,\ldots,G_k$, we write $\cC(G_1,\ldots,G_k)$ for the set of equivalence classes of $k$-sequences $G_k \xsurj{\phi_{k-1}} \ldots \xsurj{\phi_1} G_1$, and for a given $k$-sequence we denote by 
$$[G_k \xsurj{\phi_{k-1}} \ldots \xsurj{\phi_1} G_1] \in \cC(G_1,\ldots,G_k)$$
its equivalence class. Finally, let 
$$\cC_k = \bigcup_{\la^{(1)},\ldots,\la^{(k)} \in \BBY} \cC(G_{\la^{(1)}},\ldots,G_{\la^{(k)}})$$
denote the set of all equivalence classes of $k$-sequences of finite abelian $p$-groups. 
\end{definition}

For fixed finite abelian $p$-groups $G_1,\ldots,G_k$, the equivalence classes are finite and may be counted. 

\begin{example}
We find the equivalence classes of surjections ($2$-sequences) $G=\Z/p^2\Z \oplus \Z/p\Z \surj \Z/p\Z=H$. Such a map is determined by the image of $(1,0)$ and $(0,1)$, and is surjective if it is nonzero, so there are $p^2-1$ surjections. If $f: G \surj H$ is a surjection and $f(1,0) =w \neq 0, f(0,1) = v$, then it is equivalent to the map 
\begin{align*}
f':G& \surj H \\
(1,0) & \mapsto 1 \\
(0,1) & \mapsto 0
\end{align*}
by precomposing with the linear map
$$\begin{pmatrix} \tilde{w}^{-1} & 0 \\ -vw^{-1} & 1 \end{pmatrix} \in \Aut(G),$$
where $\tilde{w}$ is a lift of $w$ to $\Z/p^2\Z$. For a surjection $g$ with $g(1,0) = 0$ (so $g(0,1) = v \neq 0$), $g$ is equivalent to the map
\begin{align*}
g': G & \surj H \\ 
(1,0) & \mapsto 0 \\
(0,1) & \mapsto 1
\end{align*}
by composing with the map $v \mapsto 1$ in $\Aut(\Z/p\Z)$. Hence there are two equivalence classes, and 
\begin{align*}
\#[f'] &= p^2-p \\ 
\#[g'] &= p-1.
\end{align*}
Note that in this simple case the equivalence classes are characterized by the isomorphism types of their kernels ($\Z/p^2\Z$ versus $\Z/p\Z \oplus \Z/p\Z$), but this is not true in general as the example in Remark \ref{rmk:grp_extn_data} showed.
\end{example}

The central object of this section is the following probability measure on $\cC_k$. For now we give an explicit formula for its density, and later in Theorem \ref{thm:seq_dist_iso_types} we verify that it is in fact a probability measure.

\begin{definition}\label{def:aut_seq_dist}
Let $\tP(\cdot): \cC_k \to \BBR$ be defined by
$$\tP([G_k \xsurj{\phi_{k-1}} \ldots \xsurj{\phi_1} G_1]) = (p^{-1};p^{-1})_\infty^k \frac{1}{\#\Aut(G_k \surj \ldots \surj G_1)}.$$
\end{definition}

While it is clear that the RHS is positive, it is not obvious that the value of the normalizing constant above is correct, or even that the sum over $\cC_k$ of the RHS converges. However, this follows from the main result of the section below. %The point of this result is that the marginal distribution of the isomorphism types of the $G_i$ under $\tP$ is given by our universal cokernel distribution $\P^{(k)}_{\infty,1/p}$.%, but this will follow from the discussion below. For now we may simply view $\tP: \cC_k \to \BBR_{\geq 0}$ as a function. The main result of the section is that th

\begin{theorem}\label{thm:seq_dist_iso_types}
For any prime $p$, $\tP$ defines a probability measure on the discrete set $\cC_k$. Furthermore, the marginal distribution of $G_1,\ldots,G_k$ under $\tP$ is given by $\P^{(k)}_{\{p\}}$, i.e. for any sequence of partitions $\la^{(1)},\ldots,\la^{(k)}$ one has
\begin{equation}\label{eq:iso_type_marginals}
\sum_{\S \in \cC(G_{\la^{(1)}},\ldots,G_{\la^{(k)}})} \tP(\S) = \P^{(k)}_{\{p\}}(G_{\la^{(1)}},\ldots,G_{\la^{(k)}}).
\end{equation}
\end{theorem}

\begin{proof}
The fact that $\tP$ defines a probability measure follows from \eqref{eq:iso_type_marginals}, since $\P^{(k)}_{\{p\}}$ is a probability measure and $\tP$ is manifestly nonnegative, so we prove \eqref{eq:iso_type_marginals}. First, rewrite
\begin{equation}\label{eq:class_to_seq}
\sum_{\S \in \cC(G_{\la^{(1)}},\ldots,G_{\la^{(k)}})} \tP(\S) = \sum_{\substack{(G_{\la^{(k)}} \xsurj{\phi_{k-1}} \ldots \xsurj{\phi_1} G_{\la^{(1)}}) \\ \in \KK(G_{\la^{(1)}},\ldots,G_{\la^{(k)}})}} \frac{\tP([G_{\la^{(k)}} \xsurj{\phi_{k-1}} \ldots \xsurj{\phi_1} G_{\la^{(1)}}])}{\#[G_{\la^{(k)}} \xsurj{\phi_{k-1}} \ldots \xsurj{\phi_1} G_{\la^{(1)}}]}.
\end{equation}
We consider the group action of $\Aut(G_k) \times \cdots \times \Aut(G_1)$ on $\KK(G_1,\ldots,G_k)$ induced by the definition of isomorphism of $k$-sequences, explicitly 
$$(\varphi_k,\ldots,\varphi_1) \cdot (G_k \xsurj{\phi_{k-1}} \ldots \xsurj{\phi_1} G_1) = (G_k \xsurj{\varphi_{k-1} \circ \phi_{k-1} \circ \varphi_k^{-1}} G_{k-1} \xsurj{\varphi_{k-2} \circ \phi_{k-2} \circ \varphi_{k-1}^{-1}} \ldots \xsurj{\varphi_1 \circ \phi_1 \circ \varphi_2^{-1}} G_1).$$
Clearly the induced action on the set $\cC(G_{\la^{(1)}},\ldots,G_{\la^{(k)}})$ of equivalence classes is transitive. Hence by the orbit-stabilizer theorem, 
\begin{equation}\label{eq:use_orb_stab}
\frac{\tP([G_{\la^{(k)}} \xsurj{\phi_{k-1}} \ldots \xsurj{\phi_1} G_{\la^{(1)}}])}{\#[G_{\la^{(k)}} \xsurj{\phi_{k-1}} \ldots \xsurj{\phi_1} G_{\la^{(1)}}]} = (p^{-1};p^{-1})_\infty^k \frac{1}{\#\left(\Aut(G_k) \times \cdots \times \Aut(G_1)\right)}.
\end{equation}
Combining \eqref{eq:class_to_seq} and \eqref{eq:use_orb_stab} yields 
\begin{equation}
\text{LHS\eqref{eq:iso_type_marginals}} = (p^{-1};p^{-1})_\infty^k \frac{\#\KK(G_{\la^{(1)}},\ldots,G_{\la^{(k)}})}{\#\left(\Aut(G_k) \times \cdots \times \Aut(G_1)\right)}.
\end{equation}
Recalling the definition of $\KK$ and Definition \ref{def:intro_measures}, the RHS above is equal to the RHS of \eqref{eq:iso_type_marginals}, completing the proof.
\end{proof}

\begin{proof}[Proof of Theorem \ref{theorem:seq_aut_intro}]
It suffices to note that the measure $\tPP$ of Theorem \ref{theorem:seq_aut_intro} is a product over $p \in P$ of the measure $\tP$ considered above, and apply Theorem \ref{thm:seq_dist_iso_types}.
\end{proof}

Using results on the \emph{Hall algebra} of \cite[Chapter III]{Macd}, one may similarly check that the limiting joint distribution of 
$$\cok(M_1),\ldots,\cok(M_1 \cdots M_k),\cok(M_2),\cok(M_3),\ldots,\cok(M_k),$$
for $M_i$ iid additive Haar matrices, agrees with the joint distribution of 
$$G_1,\ldots,G_k, \ker \phi_1,\ldots,\ker \phi_{k-1}$$
under $\tP$ (note that while the maps $\phi_i$ depend on which element of the isomorphism class $\S$ is chosen, the isomorphism types of their kernels do not, hence we may speak of $\ker \phi_i$ for a random isomorphism class $\S$ distributed under $\tP$). However, as Remark \ref{rmk:grp_extn_data} shows, this is not enough to say that the limiting distribution of 
$$\cok(M_1 \cdots M_k) \surj \ldots \surj \cok(M_1)$$
is given by $\tP$. We conjecture that this is true in the case for Haar matrices, and in fact true universally in the setting of Theorem \ref{theorem:main:joint}. However, since the focus of this paper is on isomorphism types, we do not pursue this here. % 

\section{Universality of product coranks over $\F_p$}\label{sec:fp}

In this section we prove the universality of the limiting joint distribution of coranks of matrix products over $\F_p$, Theorem \ref{thm:prod:rank}.

\begin{proof}
Because $\xi$ is nonconstant, the $\Z$-valued random variable $\tilde{\xi}$ with
$$\P(\tilde{\xi} = r) = \P(\xi \equiv r \pmod{p})$$
for $r=0,\ldots,p-1$ is $\alpha$-balanced for some $\alpha$. Theorem \ref{theorem:main:joint} applied to $\tilde{\xi}$ implies that the limiting joint distribution of the $p$-parts of cokernels of matrix products with iid $\tilde{\xi}$-distributed entries exists and is independent of $\tilde{\xi}$. Since the cokernel determines the rank, this means that the limit in Theorem \ref{thm:prod:rank} exists and is independent of $\xi$.

This \emph{a priori} gives an expression for the limit in Theorem \ref{thm:prod:rank} as a sum over all cokernels corresponding to a given corank. However, to obtain a simpler expression, we note that it implies that the limit for any nonconstant $\xi$ is the same as when $\xi$ is uniform. Hence it suffices to compute the uniform case explicitly.

We claim that for this computation it suffices to show that for any $B \in \Mat_n(\F_q)$ with $\rank(B) = n-k$, and $A \in \Mat_n(\F_p)$ uniformly random,
\begin{equation}\label{eq:rank_step}
 \P(\rank(BA) = n-k-d) = \frac{(p^{-1};p^{-1})_{n-k}(p^{-1};p^{-1})_n}{(p^{-1};p^{-1})_d(p^{-1};p^{-1})_{k+d}(p^{-1};p^{-1})_{n-k-d}} p^{-d(k+d)}.
\end{equation}
Taking the $n \to \infty$ limit yields
 $$\lim_{n \to \infty} \P(\rank(BA) = n-k-d) = \frac{(p^{-1};p^{-1})_\infty p^{-(d(k+d))}}{(p^{-1};p^{-1})_d(p^{-1};p^{-1})_{k+d}}, $$
which is recognizable as the factor in the product in Theorem \ref{thm:prod:rank} (together with a normalization constant which was outside the product in the theorem statement). Hence \eqref{eq:rank_step} suffices because iterating it with $B=I,A=M_1$, then $B=M_1, A=M_2$ (conditioned on fixed $M_1$), then $B=M_1M_2, A=M_3$ (conditioned on fixed $M_1M_2$, etc. yields the desired formula. So let us prove \eqref{eq:rank_step}.

By Smith normal form there exist $U,V \in \GL_n(\F_p)$ so that $UBV = \tI_{n-k}$, where
$$\tI_{n-k} = \begin{pmatrix}
I_{n-k} & 0_{(n-k) \times k} \\ 
0_{k \times (n-k)} & 0_{k \times k}
\end{pmatrix},$$
Because the uniform measure on $\Mat_n(\F_p)$ is invariant under multiplication by any element of $\GL_n(\F_p)$, we therefore have
$$\P(\rank(BA) = n-k-d) = \P(\rank(UBVA)=n-k-d) = \P(\rank(\tI_{n-k}A = n-k-d)).$$
But $\tI_{n-k}A$ is a uniform $(n-k) \times n$ matrix padded with zeroes, so denoting this $(n-k) \times n$ matrix by $A'$, the above is equal to $\P(\rank(A') = n-k-d )$. An explicit formula for this probability is given for example in \cite[(1.2)]{Ge} (see also \cite{Bes, FG, W1}), and translating their formula slightly yields \eqref{eq:rank_step}, completing the proof.
\end{proof}


\begin{thebibliography}{99}


















\bibitem{ABK} G.~Akemann, Z.~Burda, and M.~Kieburg, From integrable to chaotic systems: Universal local
statistics of Lyapunov exponents, EPL (Europhysics Letters) 126, no. 4, 40001, 2019.
\vskip .05in
\bibitem{BBCC} T.~Banica, S.~Belinschi, M.~Capitaine and B.~Collins, Free
Bessel laws, Canad. J. Math. 63 3--37, 2011.
\vskip .05in
\bibitem{Bellman} R.~Bellman, Limit theorems for non-commutative operations I, Duke Math. J., 21 no. 3, 491--500, 1954.
\vskip .05in
\bibitem{Ben} F.~Benaych-Georges, On a surprising relation between the Marchenko-Pastur law, rectangular and square free convolutions, Ann. Inst. Henri Poincar\'e Probab. Stat.  46, 644--652, 2010.
\vskip .05in
\bibitem{Bes} E.~D.~Belsley, (1993), Rates of convergence of Markov chains related to association schemes, Ph.D. thesis, Harvard Univ.
\vskip .05in
\bibitem{Bhargava2014b} M.~Bhargava, The geometric sieve and the density of squarefree values of invariant polynomials, preprint, \url{arxiv.org/abs/1402.0031}.
\vskip .05in
\bibitem{BKLPR} M.~Bhargava, D.~Kane, H.~W.~Lenstra Jr., B.~Poonen, E.~Rains, Modeling the distribution of ranks, Selmer groups, and Shafarevich-Tate groups of elliptic curves, Camb. J. Math. 3 (2015), no. 3, 275--321.
\vskip .05in

\bibitem{Bore} I.~Boreico, Statistics of random integral matrices, Ph.D. dissertation, Stanford University, 2016. MR 4172218.
\vskip .05in
\bibitem{BC14} A.~Borodin and I.~Corwin, Macdonald processes, Probab. Theory Related Fields, 158(1--
2):225--400, 2014.
\vskip .05in
\bibitem{BJW} Z.~Burda, R.~A.~Janik, and B.~Waclaw, Spectrum of the product of independent random Gaussian matrices, Phys. Rev. E 81, 041132 (2010).
\vskip .05in
\bibitem{Burda} Z.~Burda, A. Jarosz, G. Livan, M. A. Nowak, and A. Swiech, Eigenvalues and singular values of products of rectangular Gaussian random matrices, Phys. Rev. E 82, 061114 (2010).
\vskip .05in
\bibitem{LB} L.~Butler, Subgroup lattices and symmetric functions, American Mathematical Soc. Vol. 539, 1994.
\vskip .05in
\bibitem{C} L.~Carlitz, Representations by quadratic forms in a finite field. Duke Math. J., 21 (1965), 123--137.
\vskip .05in
\bibitem{CH} G.~Cheong and Y.~Huang, Cohen–Lenstra distributions via random matrices over complete discrete valuation rings with finite residue fields. Illinois J. Math. 65, no. 2 (2021): 385--415.
\vskip .05in
\bibitem{CK} G.~Cheong and N.~Kaplan, Generalizations of results of Friedman and Washington on cokernels of random $p$-adic matrices, J. Algebra, 604:636--663, 2022.
\vskip .05in
\bibitem{CLS} G.~Cheong, Y.~Liang, and M.~Strand, The distribution of the cokernel of a polynomial
push-forward of a random $\Z_p$-matrix with a fixed residue class modulo $p$. arXiv preprint arXiv:2209.03626,
2022.
\vskip .05in 
\bibitem{CKK} G.~Chinta, N.~Kaplan, and S.~Koplewitz, The cotype zeta function of $\Z^d$. Indag. Math. 34 (2023) no. 3, 643--659.%preprint, \url{arxiv.org/abs/1708.08547}.
\vskip .05in
\bibitem{CLP} J.~Clancy, T.~Leake and S.~Payne, A note on Jacobians, Tutte polynomials, and two-variable zeta functions of graphs. Experiment. Math. 24, 1--7 (2015).
\vskip .05in
\bibitem{Clancy2015} J.~Clancy, T.~Leake, N.~Kaplan, S.~Payne, and M.~M.~Wood, On a Cohen-Lenstra heuristic for Jacobians of random graphs, J. Algebraic Combin. 42 (2015), no. 3, 701--723.
\vskip .05in

\bibitem{CL} H.~Cohen and H.~W.~Lenstra, Jr. Heuristics on class groups of number fields, In Number theory, Noordwijkerhout 1983 (Noordwijkerhout, 1983), volume 1068 of Lecture Notes in Math., pages 331--762. Springer, Berlin, 1984.
\vskip .05in
\bibitem{CPV} A. Crisanti, G. Paladin, and A. Vulpiani, Products of random matrices in statistical physics, Springer Series in
Solid-State Sciences, vol. 104, Springer-Verlag, Berlin, 1993. With a foreword by Giorgio Parisi. 
\vskip .05in
 \bibitem{Delaunay2001} C.~Delaunay, Heuristics on Tate-Shafarevitch Groups of Elliptic Curves Defined over $\mathbb{Q}$, Experiment. Math.,
              Volume 10, Number 2 (2001), 191--196.

\vskip .05in
\bibitem{DJ} C.~Delaunay and F.~Jouhet. $p$-torsion points in finite abelian groups and combinatorial
identities. Adv. Math., 258:13–45, 2014.
% \vskip .05in
%        \bibitem{CNg} N. Cook and H. Nguyen, Universality of the minimum modulus for random trigonometric polynomials, Discrete Analysis, 20 (2021), 45pp. 
  \vskip .05in
\bibitem{Eke}T.~Ekedahl, An infinite version of the Chinese remainder theorem, Comment. Math. Univ. St. Paul., 40 (1991), no.~1, 53--59.
\vskip .05in

\bibitem{FA} S.~D.~Fisher\ and\ M.~N.~Alexander, Classroom Notes: Matrices over a Finite Field, Amer. Math. Monthly, 73 (1966), no.~6, 639--641.
\vskip .05in
\bibitem{FW} E.~Friedman and L.~C.~Washington, On the distribution of divisor class groups of curves over a finite field. In Theorie des nombres (Quebec, PQ, 1987), pages 227--239. de Gruyter, Berlin, 1989.
\vskip .05in
\bibitem{Fulman1999probabilistic} J.~Fulman, A probabilistic approach toward conjugacy classes in the finite general linear and unitary groups. J. Algebra (1999), 212(2), 557--590.
\vskip .05in
\bibitem{Fulman2014cohen} J.~Fulman, Cohen--Lenstra heuristics and random matrix theory over finite fields. J. Group Theory, 17(4):619–648, 2014.
\vskip .05in
\bibitem{Fulman2016hall} J.~Fulman, Hall--Littlewood polynomials and Cohen-Lenstra heuristics for Jacobians of random graphs. Ann. Comb., 20(1):115–124, 2016.
\vskip .05in
\bibitem{FG} J.~Fulman and L.~Goldstein, Stein's method and the rank distribution of random matrices over finite fields, Ann. Probab. Volume 43, Number 3 (2015), 1274--1314.
\vskip .05in
\bibitem{FK} J.~Fulman,  N.~Kaplan, Random Partitions and Cohen–Lenstra Heuristics. Ann. Comb. 23, 295–315 (2019). https://doi.org/10.1007/s00026-019-00425-y.
\vskip .05in 
\bibitem{FursKes} H.~Furstenberg and H.~Kesten. Products of random matrices. The Annals of Mathematical Statistics, 31(2), 457--469.
\vskip .05in 
\bibitem{Ge} F.~Gerth III. Limit probabilities for coranks of matrices over $GF(q)$. Linear Multilinear Algebra 19, no. 1 (1986): 79--93.
\vskip .05in 

\bibitem{Gin} J.~Ginibre, Statistical Ensembles of Complex, Quaternion, and Real Matrices, J. Math. Phys., 6 (1965), 440–449.
\vskip .05in

\bibitem{GS} V.~Gorin and Y.~Sun, Gaussian fluctuations for products of random matrices. American J. Math., 144 no. 2 (2022), 287--393.
\vskip .05in
\bibitem{GTcov} F.~G\"otze and A.~Tikhomirov, Rate of convergence in probability to the Marchenko-Pastur law. Bernoulli 10(3) (2004), 503--548.
\vskip .05in
\bibitem{GTprod} F.~G\"otze and A.~Tikhomirov, On the asymptotic spectrum of products of independent random matrices.
arXiv:1012.2710.
\vskip .05in
\bibitem{GT} F.~G\"{o}tze and A.~Tikhomirov, The circular law for random matrices, Ann. Probab. 38, No.
4, 1444--1491, (2010).
\vskip .05in
\bibitem{HN} B.~Hanin and M.~Nica, Products of many large random matrices and gradients in deep neural networks,
Comm. Math. Phys. 376 (2020), no. 1, 287–322. 
\vskip .05in
\bibitem{KKS} M.~Kieburg, A.~B.~J.~Kuijlaars, and D.~Stivigny. Singular value statistics of matrix products with truncated unitary matrices. Int. Math. Res. Not. IMRN, 2016(11):3392–3424, 2016.
\vskip .05 in
\bibitem{KS} A.~B.~J.~Kuijlaars and D.~Stivigny. Singular values of products of random matrices and polynomial
ensembles. Random Matrices Theory Appl., 3(03):1450011, 2014.
\vskip .05in
\bibitem{KZ} A.~B.~J.~Kuijlaars and L.~Zhang. Singular values of products of Ginibre random matrices, multiple orthogonal polynomials and hard edge scaling limits. Comm. Math. Phys., 332:759–781,
2014.
\vskip .05in
\bibitem{Lee} J.~Lee, Joint distribution of the cokernels of random $p$-adic matrices. Forum Math. 35(2023), no.4, 1005--1020.

\vskip .05in

\bibitem{Lee2} J.~Lee, Mixed moments and the joint distribution of random groups. arXiv preprint
arXiv:2210.04278, 2022.

\vskip .05in

\bibitem{L} J.~Lengler, The Cohen--Lenstra heuristic: methodology and results. J. Algebra, 323(10):2960–2976, 2010.
\vskip .05in
\bibitem{LWZ} D.-Z.~Liu, D.~Wang, and L.~Zhang. Bulk and soft-edge universality for singular values of products of Ginibre random matrices. Ann. Inst. Henri Poincar\'e Probab. Stat., 52(4):1734–1762, 2016.
\vskip .05in
\bibitem{Mac} J.~MacWilliams, Orthogonal matrices over finite fields. Amer. Math. Monthly,
76(2):152--164, February 1969. %ArticleType: research-article / Full publication date: Feb., 1969 / Copyright 1969 Mathematical Association of America.
\vskip .05in
 \bibitem{Macd} I.~G.~Macdonald, Symmetric Functions and Hall Polynomials. Second Edition. Oxford University Press, New York (1995).
 \vskip .05in
\bibitem{M1} K.~Maples, Singularity of Random Matrices over Finite Fields, preprint, \url{arxiv.org/abs/1012.2372}.
\vskip .05in
\bibitem{Msym} K.~Maples, Symmetric random matrices over finite fields announcement, preprint, \url{user.math.uzh.ch/maples/maples.symma.pdf}.
\vskip .05in
%\bibitem{M2} K.~Maples, Cokernels of random matrices satisfy the Cohen-Lenstra heuristics, preprint, \url{arxiv.org/abs/%1301.1239}.
%\vskip .05in
\vskip .05in

\bibitem{MP} V.A.~Marcenko and L.A.~Pastur, Distribution of eigenvalues in certain sets of random matrices. Mat. Sb. (N.S.) 72 (114) 507--536.
\vskip .05in
\bibitem{M19} K.~Matveev, Macdonald-positive specializations of the algebra of symmetric functions: Proof
of the Kerov conjecture. Ann. of Math. (2), 189(1):277–316, 2019.
\vskip .05in
\bibitem{Mehta} M.L.~Mehta, Random Matrices and the Statistical Theory of Energy Levels, Academic Press,
New York, NY, 1967.
\vskip .05in



\bibitem{Mu} R.~R.~Muller, IEEE Trans. Inf. Theory 48, 2086, (2002).
\vskip .05in
\bibitem{Ng-L} H.~Nguyen, Asymptotic Lyapunov exponents for large random matrices (Ann. Appl. Probab. (2017), Vol. 27, No. 6, 3672--3705.
\vskip ,05in 
\bibitem{NgP} H.~Nguyen and E.~Paquette, Surjectivity of near square matrices, Combin. Probab. Comput., 29 (2020), no. 2, 267--292.
\vskip .05in
\bibitem{NgW-iid} H.~Nguyen and M.~M.~Wood,  Random integral matrices: universality of surjectivity and the cokernel, Invent. Math. 228 (2022), no. 1, 1--76.
\vskip .05in
\bibitem{NgW-lap} H.~Nguyen and M.~M.~Wood, Local and global statistics of random matrix cokernels. arXiv preprint arXiv:2210.08526.
\vskip .05in
\bibitem{OS}  S.~O'Rourke and A.~Soshnikov, Products of Independent Non-Hermitian Random Matrices, Electron. J. Probab., Vol. 16, Art. 81, 2219--2245 (2011). 
\vskip .05in
\bibitem{ORSV}  S.~O'Rourke, D.~Renfrew, A.~Soshnikov and V.~Vu, Products of independent elliptic random matrices, J. Stat. Phys. Vol. 160, No. 1 (2015), 89--119.
\vskip .05in

\bibitem{PZ} G.~Pan and W.~Zhou, Circular law, extreme singular values and potential theory, J. Multivariate
Anal. (2010), 101, 645--656.
\vskip .05in

\bibitem{P} L.A.~Pastur, Spectra of random selfadjoint operators. Russian Math. Surveys 28 1--67.
\vskip .05in
\bibitem{Pet} V.~M.~Petrogradsky, Multiple zeta functions and asymptotic structure of free abelian groups of finite rank, J. Pure Appl. Algebra, 208 (2007), no.~3, 1137--1158.
\vskip .05in

\bibitem{S} D.~Speyer, (Short) Exact sequences with no commutative diagram between them. URL (version: 2014-02-19): \url{https://mathoverflow.net/q/157955}
\vskip .05in
\bibitem{TVcir} T.~Tao and V.~Vu, Random Matrices: Universality of ESDs and the Circular Law, Ann. Probab.. 38, No. 5, 2023--2065 (2010).
\vskip .05in


\bibitem{VP21q} R.~Van Peski, $q$-TASEP with position-dependent slowing, Electronic Journal of Probability, 27:1--35, 2022.
\vskip .05in
\bibitem{VP21limits} R.~Van Peski, Limits and fluctuations of p-adic random matrix products. Selecta Mathematica, 27(5):1--71, 2021.
\vskip .05in
\bibitem{VP22hall} R.~Van Peski, Hall-Littlewood polynomials, boundaries, and $p$-adic random matrices. Int. Math. Res. Not. IMRN, 2023(13):11217–11275, 2022
\vskip .05in
\bibitem{VP22+universal} R.~Van Peski, Local limits in $p$-adic random matrix theory. arXiv preprint arXiv:2310.12275.
\vskip .05in
\bibitem{Wa} K.~W.~Wachter, The strong limits of random matrix spectra for sample matrices of independent elements. Ann. Probab. 6 1--18.
\vskip .05in

\bibitem{woodexpos} M.~M.~Wood, Asymptotics for number fields and class groups. Directions in Number Theory, 291--339, 2016.
\vskip .05in

\bibitem{W0} M.~M.~Wood, The distribution of sandpile groups of random graphs, J. Amer. Math. Soc.,  30 (2017), pp. 915--958.
\vskip .05in
\bibitem{W1}  M.~M.~Wood, Random integral matrices and the Cohen-Lenstra Heuristics,  Amer. J. Math., Volume 141, Number 2 (2019), pp. 383--398 .
\vskip .05in
\bibitem{WICM}  M.~M.~Wood, Probability theory for random groups arising in number theory, Lecture notes for ICM. 
\vskip .05in
\bibitem{Wright} D.~J.~Wright, Distribution of discriminants of abelian extensions, Proc. Lond. Math. Soc. (3) 58(1), 17--50 (1989).
\vskip .05in

\bibitem{Yin} Y.~Q.~Yin, Limiting spectral distribution for a class of random matrices, J. Multivariate Anal. 20 50--68.
\vskip .05in

\bibitem{ZPNC}  K.~Zyczkowski, K.~A.~Penson, I.~Nechita and B.~Collins,  Generating random density matrices, J. Math. Phys. 52, 062201 (2011).
\vskip .05in

\end{thebibliography}
\end{document}